\documentclass[onefignum,onetabnum]{siamart171218}

\usepackage{lipsum}
\usepackage{graphicx}
\usepackage{mwe}     
\usepackage{graphbox}
\usepackage{multirow,multicol}
\usepackage{epstopdf}
\usepackage{algorithmicx} 
\usepackage[margin=1in]{geometry}
\usepackage{amsmath,amsfonts,amssymb,mathrsfs,cancel}
\usepackage{accents}
\usepackage{mathtools}
\usepackage{stmaryrd}
\usepackage{comment}
\usepackage{ulem}
\usepackage{empheq}
\usepackage{cases}
\ifpdf
  \DeclareGraphicsExtensions{.eps,.pdf,.png,.jpg}
\else
  \DeclareGraphicsExtensions{.eps}
\fi
\usepackage[shortlabels]{enumitem} 

\usepackage{svg}
\usepackage{placeins}
\usepackage{caption}
\usepackage{subcaption}
\makeatletter
\newcommand{\clearsubcaptcounter}{\setcounter{sub\@captype}{0}}
\makeatother
\usepackage[export]{adjustbox} 

\usepackage{algorithm}
\usepackage{algpseudocode}

\newsiamremark{remark}{Remark}
\newsiamremark{hypothesis}{Hypothesis}
\crefname{hypothesis}{Hypothesis}{Hypotheses}
\newsiamthm{claim}{Claim}

\headers{Unified hp-HDG Frameworks for Friedrichs' PDE systems}{Jau-Uei Chen, Shinhoo Kang, Tan Bui-Thanh and John N. Shadid}

\title{Unified hp-HDG Frameworks for Friedrichs' PDE systems\thanks{Submitted to the editors DATE.
\funding{This work was funded by some funds}}}

\author{
Jau-Uei Chen\thanks{Department of Aerospace Engineering and Engineering Mechanics, The University of Texas at Austin, Austin, TX 78712, USA
  (\email{chenju@utexas.edu}).}
\and
Shinhoo Kang\thanks{Mathematics and Computer Science Division, Argonne National Laboratory. Lemont, IL 60439, USA.}
\and
Tan Bui-Thanh\footnotemark[2]\,\,\thanks{Oden Institute for Computational Engineering and Sciences, The University of Texas at Austin, Austin, TX 78712, USA}
\and
John N. Shadid\thanks{Computational Mathematics Department, Sandia National Laboratories, P.O. Box 5800, MS 1321, Albuquerque, NM 87185, US}
}

\usepackage{amsopn}

\DeclareMathOperator*{\infess}{ess\,inf}

\makeatletter
\newcommand*{\addFileDependency}[1]{
  \typeout{(#1)}
  \@addtofilelist{#1}
  \IfFileExists{#1}{}{\typeout{No file #1.}}
}
\makeatother

\newcommand{\beq}{\begin{equation}}
\newcommand{\eeq}{\end{equation}}
\newcommand{\bal}{\begin{align}}
\newcommand{\eal}{\end{align}}
\newcommand{\bsp}{\begin{split}}
\newcommand{\esp}{\end{split}}
\newcommand{\bdm}{\begin{displaymath}}
\newcommand{\edm}{\end{displaymath}}
\newcommand{\bit}{\begin{itemize}}
\newcommand{\eit}{\end{itemize}}
\newcommand{\bde}{\begin{description}}
\newcommand{\ede}{\end{description}}
\newcommand{\bce}{\begin{center}}
\newcommand{\ece}{\end{center}}
\newcommand{\ben}{\begin{enumerate}}
\newcommand{\een}{\end{enumerate}}
\newcommand{\bea}{\begin{eqnarray}}
\newcommand{\eea}{\end{eqnarray}}
\newcommand{\barr}{\begin{array}}
\newcommand{\earr}{\end{array}}
\newcommand{\bean}{\begin{eqnarray*}}
\newcommand{\eean}{\end{eqnarray*}}
\newcommand{\edoc}{\end{document}}

\newtheorem{rema}{remark}

\newcommand{\itmlab}[1]{\label{itm:#1}}
\newcommand{\figlab}[1]{\label{fig:#1}}
\newcommand{\eqnlab}[1]{\label{eq:#1}}
\newcommand{\theolab}[1]{\label{theo:#1}}

\newcommand{\lemlab}[1]{\label{lem:#1}}

\newcommand{\tablab}[1]{\label{tab:#1}}
\newcommand{\algmlab}[1]{\label{alg:#1}}
\newcommand{\seclab}[1]{\label{sect:#1}}

\newcommand{\itmref}[1]{\ref{itm:#1}}
\newcommand{\figref}[1]{\ref{fig:#1}}
\newcommand{\theoref}[1]{\ref{theo:#1}}

\newcommand{\lemref}[1]{\ref{lem:#1}}
\newcommand{\eqnref}[1]{\eqref{eq:#1}}
\newcommand{\algmref}[1]{\ref{alg:#1}}
\newcommand{\secref}[1]{\ref{sect:#1}}
\newcommand{\tabref}[1]{\ref{tab:#1}}

\newcommand{\bs}[1]{\boldsymbol{#1}}
\newcommand{\tit}[1]{\textit{#1}}
\newcommand{\mc}[1]{\mathcal{#1}}
\newcommand{\mb}[1]{\mathbf{#1}}
\newcommand{\mbb}[1]{\mathbb{#1}}
\newcommand{\msc}[1]{\mathscr{#1}}

\newcommand{\LRp}[1]{\left( #1 \right)}
\newcommand{\LRs}[1]{\left[ #1 \right]}
\newcommand{\LRa}[1]{\left< #1 \right>}
\newcommand{\LRc}[1]{\left\{ #1 \right\}}
\newcommand{\half}{\frac{1}{2}}
\newcommand{\pp}[2]{\frac{\partial #1}{\partial #2}}
\newcommand{\ppSq}[2]{\frac{\partial^2 #1}{\partial {#2}^2}}
\newcommand{\Grad}{\ensuremath{\nabla}}

\newcommand{\Div}{\ensuremath{\nabla\cdot}}

\newcommand{\jump}[1]{\llbracket{#1}\rrbracket}

\newcommand{\abs}[1]{\left\lvert#1\right\rvert}
\newcommand{\norm}[1]{\left\lVert#1\right\rVert}

\newcommand{\eval}[2]{#1\vert_{#2}}

\newcommand{\real}{\mbb{R}}

\newcommand{\dom}{\Omega}
\newcommand{\domBnd}{\partial\dom} 
\newcommand{\domBndDir}{\domBnd_D} 
\newcommand{\DirVal}{g_D} 
\newcommand{\domBndNmn}{\domBnd_N} 
\newcommand{\NmnVal}{g_N} 
\newcommand{\domBndRb}{\domBnd_R} 
\newcommand{\RbVal}{g_R} 
\newcommand{\BCVal}{g} 
\newcommand{\InBnd}{\domBnd^-}
\newcommand{\OutBnd}{\domBnd^+}
\newcommand{\ZeroBnd}{\domBnd_0}
\newcommand{\domPart}{\msc{T}_h} 
\newcommand{\domPartH}{\msc{T}_H} 
\newcommand{\domPartBnd}{\partial\domPart} 
\newcommand{\elem}{K}
\newcommand{\elemChild}{\elem'}
\newcommand{\elemL}{h_{\elem}} 
\newcommand{\elemP}{\elem^+}
\newcommand{\elemM}{\elem^-}
\newcommand{\elemBnd}{\partial\elem}
\newcommand{\face}{F}
\newcommand{\faceP}{\face^+}
\newcommand{\faceM}{\face^-}

\newcommand{\mort}{e}
\newcommand{\normal}{\bs{n}}
\newcommand{\normalComp}[1]{n_{#1}}

\newcommand{\normalCon}{\bs{e}} 

\newcommand{\skel}{\mc{E}_h} 
\newcommand{\skelBnd}{\skel^{\partial}}
\newcommand{\skelInt}{\skel^{\circ}}

\newcommand{\Hsp}{\msc{L}} 

\newcommand{\FdsVarLump}{\msc{Z}_h}
\newcommand{\FdsTestVarLump}{\msc{W}_h}

\newcommand{\ApprSpc}[1]{P_{h}^{#1}}
\newcommand{\TrcApprSpc}[1]{\widehat{P}_{h}^{#1}}
\newcommand{\FdsVarApprSpc}[1]{\bs{W}_{h}^{#1}}
\newcommand{\FdsTrcVarApprSpc}[1]{\widehat{\bs{W}}_{h}^{#1}}

\newcommand{\FdsVarSolApprSpc}[1]{\bs{U}_{h}^{#1}}
\newcommand{\FdsTrcVarSolApprSpc}[1]{\widehat{\bs{U}}_{h}^{#1}}

\newcommand{\FdsVarAuxApprSpc}[1]{\bs{\Sigma}_{h}^{#1}}
\newcommand{\FdsTrcVarAuxApprSpc}[1]{\widehat{\bs{\Sigma}}_{h}^{#1}}

\newcommand{\Lsp}[1]{\mc{L}^{#1}} 

\newcommand{\intRef}{I} 

\newcommand{\stabPar}{T}

\newcommand{\poly}{p}
\newcommand{\polyElem}{\poly_{\elem}}
\newcommand{\polyElemP}{\poly_{\elemP}}
\newcommand{\polyElemM}{\poly_{\elemM}}
\newcommand{\polyMort}{\poly_{\mort}}
\newcommand{\polySpc}[1]{\mc{P}^{#1}} 

\newcommand{\proj}{\Pi}

\newcommand{\projOp}[1]{\mathbb{P}\LRp{#1}}

\newcommand{\phyComp}{x}
\newcommand{\phyCompX}{\phyComp_1}
\newcommand{\phyCompY}{\phyComp_2}

\newcommand{\phyVar}{\bs{\phyComp}}

\newcommand{\phySkwCompX}{X}
\newcommand{\phySkwCompY}{Y}

\newcommand{\sol}{u}

\newcommand{\solAppr}{\sol_h} 
\newcommand{\solApprH}{\sol_H} 

\newcommand{\auxComp}{\sigma} 

\newcommand{\aux}{\bs{\auxComp}}

\newcommand{\forcing}{f} 
\newcommand{\diffCoeffComp}{\kappa}
\newcommand{\diffCoeffCompX}{\kappa_{\phySkwCompX}}
\newcommand{\diffCoeffCompY}{\kappa_{\phySkwCompY}}
\newcommand{\diffCoeff}{{\widehat{\boldsymbol{\diffCoeffComp}}}} 
\newcommand{\misalign}{\theta_m}
\newcommand{\anisoRatio}{\mc{A}_{\diffCoeffComp}}
\newcommand{\rtCoeff}{\lambda}

\newcommand{\velComp}[1]{\beta_{#1}}
\newcommand{\vel}{\bs{\beta}}

\newcommand{\upwindFlux}{\bs{F}}
\newcommand{\upwindFluxMort}{\upwindFlux^*}
\newcommand{\upwindFluxTrc}{\widehat{\upwindFlux}}

\newcommand{\id}{{I}} 

\newcommand{\x}{\phyCompX}
\newcommand{\y}{\phyCompY}

\newcommand{\ErrEst}{\msc{E}} 
\newcommand{\RegInd}{\msc{G}} 
\newcommand{\LocErrEstGen}[1]{\ErrEst_{#1,\elem}} 
\newcommand{\LocErrEst}[2]{\ErrEst_{#1,\elem}^{#2}} 
\newcommand{\GlbErrEstGen}[1]{\ErrEst_{#1}} 

\newcommand{\LocRegInd}[1]{\RegInd_{\elem}\LRp{#1}}

\newcommand{\tol}{\omega}

\newcommand{\FdsFlux}{\bs{F}}
\newcommand{\FdsFluxComp}{{F}_{\FdsIndx}}
\newcommand{\FdsVar}{\bs{z}}
\newcommand{\FdsTrcVar}{\widehat{\bs{z}}}
\newcommand{\FdsTest}{\bs{w}}
\newcommand{\FdsSol}{\bs{u}}
\newcommand{\FdsTrcSol}{\widehat{\bs{u}}}
\newcommand{\FdsAux}{\bs{\sigma}}
\newcommand{\FdsVarSol}{\FdsSol}
\newcommand{\FdsVarAux}{\FdsAux}
\newcommand{\Fdsforcing}{\bs{f}} 
\newcommand{\FdsforcingSol}{\Fdsforcing^{\FdsSol}}
\newcommand{\FdsforcingAux}{\Fdsforcing^{\FdsAux}}

\newcommand{\FdsVarAppr}{\FdsVar_h}
\newcommand{\FdsTrcVarAppr}{\widehat{\FdsVar}_h}
\newcommand{\FdsTestAppr}{\FdsTest_h}
\newcommand{\FdsTestSolAppr}{\bs{v}_h}
\newcommand{\FdsTestAuxAppr}{\bs{s}_h}
\newcommand{\FdsTrcTestAppr}{\widehat{\FdsTest}_h}
\newcommand{\FdsTrcTestSolAppr}{\widehat{\bs{v}}_h}
\newcommand{\FdsTrcTestAuxAppr}{\widehat{\bs{s}}_h}

\newcommand{\FdsVarSolAppr}{\FdsSol_h}
\newcommand{\FdsVarAuxAppr}{\FdsAux_h}
\newcommand{\FdsTrcVarSolAppr}{\widehat{\FdsSol}_h}
\newcommand{\FdsTrcVarAuxAppr}{\widehat{\FdsAux}_h}

\newcommand{\FdsIndx}{k}

\newcommand{\FdsABnd}{A}
\newcommand{\FdsAbsA}{\abs{\FdsABnd}}
\newcommand{\FdsAbsAaa}{\mathcal{A}^{\FdsAux\FdsAux}}
\newcommand{\FdsAbsAas}{\mathcal{A}^{\FdsAux\FdsSol}}
\newcommand{\FdsAbsAsa}{\mathcal{A}^{\FdsSol\FdsAux}}
\newcommand{\FdsAbsAss}{\mathcal{A}^{\FdsSol\FdsSol}}
\newcommand{\FdsEigenVal}{{\Lambda}}
\newcommand{\FdsEigenVec}{{R}}

\newcommand{\FdsAk}{\FdsABnd_{\FdsIndx}}
\newcommand{\FdsNVar}{m}
\newcommand{\FdsNVarSol}{\FdsNVar^{\FdsSol}}
\newcommand{\FdsNVarAux}{\FdsNVar^{\FdsAux}}

\newcommand{\FdsSum}{\sum_{\FdsIndx=1}^d}
\newcommand{\FdsPartial}{\partial_{\FdsIndx}}

\newcommand{\FdsAkaa}{\FdsABnd_\FdsIndx^{\FdsAux\FdsAux}}

\newcommand{\FdsBBnd}{B}
\newcommand{\FdsBtBnd}{\FdsBBnd^{T}}
\newcommand{\FdsBk}{\FdsBBnd_\FdsIndx}

\newcommand{\FdsCBnd}{C}
\newcommand{\FdsCk}{\FdsCBnd_\FdsIndx}

\newcommand{\FdsG}{G}
\newcommand{\FdsGaa}{\FdsG^{\FdsAux\FdsAux}}
\newcommand{\FdsGas}{\FdsG^{\FdsAux\FdsSol}}
\newcommand{\FdsGsa}{\FdsG^{\FdsSol\FdsAux}}
\newcommand{\FdsGss}{\FdsG^{\FdsSol\FdsSol}}

\newcommand{\FdsM}{M}
\newcommand{\FdsMBnd}{\FdsM}

\newcommand{\FdsMss}{\FdsMBnd^{\FdsSol\FdsSol}}

\newcommand{\FluxAssumpA}{\Phi}
\newcommand{\FluxAssumpB}{\Psi}

\newcommand{\Null}[1]{\mc{N}\LRp{#1}}

\newcommand{\Output}[1]{\mc{J}\LRs{#1}}
\newcommand{\OutFcnl}{\mc{J}}
\newcommand{\Inject}{\mb{I}^{h}_{H}}
\newcommand{\FdsVarApprH}{\FdsVar_H}

\newcommand{\FdsVarLumpH}{\msc{Z}_H}

\newcommand{\Res}{\mc{R}}

\newcommand{\ResidualhOne}[1]{\mc{R}_h^\text{one}\LRs{#1}}
\newcommand{\ResidualhTwo}[1]{\mc{R}_h^\text{two}\LRs{#1}}

\newcommand{\FdsVarApprInject}{\FdsVar_H^h}

\newcommand{\FdsVarLumpInject}{\msc{Z}_H^h}

\newcommand{\FdsAdVarAppr}{\FdsTest_h}
\newcommand{\FdsAdTrcVarAppr}{\widehat{\FdsTest}_h}
\newcommand{\FdsAdVarSolAppr}{\bs{v}_h}
\newcommand{\FdsAdVarAuxAppr}{\bs{s}_h}
\newcommand{\FdsAdTrcVarSolAppr}{\widehat{\bs{v}}_h}

\newcommand{\FdsAdVarApprSpc}[1]{\bs{W}_{h}^{#1}}
\newcommand{\FdsAdTrcVarApprSpc}[1]{\widehat{\bs{W}}_{h}^{#1}}

\newcommand{\FdsAdTestAppr}{\delta\FdsVar_h}
\newcommand{\FdsAdTestSolAppr}{\delta\FdsSol_h}
\newcommand{\FdsAdTestAuxAppr}{\delta\FdsAux_h}
\newcommand{\FdsAdTrcTestAppr}{\delta\widehat{\FdsVar}_h}
\newcommand{\FdsAdTrcTestSolAppr}{\delta\widehat{\FdsSol}_h}

\newcommand{\FdsAdVarAuxApprSpc}[1]{\bs{\Sigma}_{h}^{#1}}
\newcommand{\FdsAdVarSolApprSpc}[1]{\bs{U}_{h}^{#1}}
\newcommand{\FdsAdTrcVarSolApprSpc}[1]{\widehat{\bs{U}}_{h}^{#1}}

\newcommand{\TranMap}{\chi}

\newcommand{\f}{f}

 \usepackage[colorinlistoftodos,prependcaption,textsize=tiny]{todonotes}

\usepackage{soul}

\ifpdf
\hypersetup{
  pdftitle={Unified hp-HDG Frameworks for Friedrichs' PDE systems},
  pdfauthor={Jau-Uei Chen, Shinhoo Kang, Tan Bui-Thanh}
}
\fi

\begin{document}

\maketitle

\begin{center}
    {\em This paper is dedicated to Professor Leszek Demkowicz on the occasion of his $70$th birthday}.
\end{center}

\begin{abstract}
This work proposes a unified $hp$-adaptivity framework for hybridized  discontinuous Galerkin (HDG) method for a large class of partial differential equations (PDEs) of  Friedrichs' 
type.
In particular, we present unified $hp$-HDG formulations for abstract one-field and two-field structures and prove their well-posedness.
In order to handle non-conforming interfaces we simply take advantage of HDG built-in mortar structures. With split-type mortars and the approximation space of trace, a numerical flux can be derived via Godunov approach and be naturally employed without any additional treatment. As a consequence, the proposed formulations are parameter-free.
We perform several numerical experiments for time-independent and linear PDEs   including elliptic, hyperbolic, and mixed-type to verify the proposed unified $hp$-formulations and demonstrate the effectiveness of $hp$-adaptation. Two adaptivity criteria are considered: one is based on a simple and fast error indicator, while the other is rigorous but more expensive using an adjoint-based error estimate. The numerical results show that these two approaches are comparable in terms of convergence rate even for problems with  strong gradients, discontinuities, or singularities.
\end{abstract}

\begin{keywords}
  Hybridized Discontinuous Galerkin; Friedrichs' system; Discontinuous Galerkin; Hybridization; $hp$-adaptation
\end{keywords}

\begin{AMS}
  68Q25, 68R10, 68U05
\end{AMS}

\section{Introduction}
The hybridized discontinuous Galerkin (HDG) methods were first
introduced in \cite{Cockburn2009a} and they inherit many benefits of
discontinuous Galerkin (DG) methods including the applicability to a
wide variety of partial differential equations (PDEs), the capability
of handling complex geometries, and supporting high-order accuracy, to
name a few. In addition, HDG methods improve computational efficiency
\cite{Cockburn2016a} by condensing out the local unknowns, and  the linear
system to be solved for the trace unknowns on the mesh skeleton is
smaller than DG counterparts. With these favorable advantages, HDG
methods indeed have great success solving various kinds of PDEs such
as Poisson equation \cite{Cockburn2016a,Cockburn2010,kirby2012},
convection-diffusion equations \cite{Nguyen2009a,Nguyen2009b,Fu2014},
Stokes equations
\cite{Cockburn2009b,Cockburn2011,Nguyen2010,Cockburn2014, kang2019},
Navier-Stokes equations \cite{Nguyen2011a,Cesmelioglu2016}, Maxwell
equation \cite{nguyen2011b,Li2014}, acoustics and elastodynamics
equations \cite{nguyen2011c}, Helmholtz equation
\cite{griesmaier2011,cui2014}, and magneto-hydrodynamic equations
\cite{lee2019}, to mention a few.  A constructive and unified HDG
framework for a large class of physics governed by elliptic,
parabolic, hyperbolic, and mixed-typed PDEs has been developed in
\cite{Tan2015} that not only rediscovers most of the existing HDG
methods but also discovers new ones.

As with any numerical discretization method, standard HDG could be
inefficient in some crucial situations where high gradient,
discontinuous, and/or singular features are present. Unfortunately,
these extreme features are not uncommon in almost all
engineering/physics applications. A cure to this issue is to employ
$hp$-adaptivity. The idea is first proposed in \cite{babuska1981a} and
is systematically studied in \cite{guo1986a,guo1986b}. It consists of
two key findings. The first one is that an exponential convergence
rate can be attained by uniformly increasing the degree of
approximation ($p$-refinement) if the solution is regular enough
\cite{Babuska1987}. The second one is that a low degree of
approximation along with refined mesh ($h$-refinement) is desired if
the solution is non-smooth. In brief, the ideal situation is to
locally execute either $h$- or $p$- refinement according to the local
behavior of the solution.

In fact, the adaptive feature has been routinely applied in the
context of HDG methods either through $h$-adaptivity \cite{carstensen2011,
  nguyen2011d, Egger2013, Dahm2014, woopen2014a, chen2016,
  cockburn2016b, Samii2016, araya2019, fidkowski2020, leng2020,
  muixi2020, sanchez2020, leng2021a, leng2021b, shin2021, bai2022,
  cockburn2022, levy2023}, $p$-adaptivity \cite{giorgiani2013b, Giorgio2014,
  Hoermann2018, sevilla2018, giacomini2019, may2021}, or
$hp$-adaptivity \cite{balan2013, Woopen2014b, woopen2015,
  Balan2016}. To drive the adaptation process, some indicator is
necessary. There are three popular approaches: \textit{a posterior}
error estimator, adjoint-based error estimate, and heuristic
indicator. Although the reliability and efficiency of an \textit{a
  posterior} error estimator sometimes can be guaranteed
\cite{Egger2013, chen2016, araya2019}, the derivation is
problem-dependent and is typically non-trivial, especially for nonlinear problems,
\cite{carstensen2011, Egger2013, chen2016, araya2019, leng2020,
  leng2021a, leng2021b, shin2021}. In addition, a post-process may be
required in this type of error estimator \cite{Egger2013, araya2019,
  shin2021, giorgiani2013b, Giorgio2014}. On the other hand, an
adjoint-based error estimate is popular in engineering applications
\cite{balan2013, Dahm2014, woopen2014a, woopen2015, Balan2016,
  fidkowski2020, may2021, cockburn2022} since, in this scenario, one
is usually more interested in some specific quantities instead
of the  solution itself.
An
adaptation process driven by an adjoint-based error estimate has been
developed for computing accurate values for such 
quantities of interest. Finally, some heuristic indicator can also be employed in
driving adaptation
\cite{Huerta1999} and is typically inexpensive to be computed. For example, in
\cite{Samii2016, Hoermann2018}, a measure of jump of flux is used as an error indicator. However, an error indicator is not necessarily associated
with ``error". For instance, authors in \cite{nguyen2011d,
  bai2022} take advantage of artificial viscosity as an indicator and
the work \cite{muixi2020} uses damage-field as an indicator. In this
paper, we study both adjoint-based error estimate and error
indicator. The first option is easier to be derived compared to
\textit{a posterior} error estimator but it still possesses certain
robustness \cite{balan2013, Dahm2014, woopen2014a, woopen2015,
  Balan2016}. The second option is more ad-hoc and lacks robustness. The discussion of both approaches will be
addressed in Section \secref{strategy_hp} and the numerical comparison
will be made in Section \secref{numerics}.

The adaptation procedure involves either $h$-refinement or
$p$-refinement or both for the mesh under consideration. The local
mesh refinement can be achieved without having any hanging node at the
element boundaries for simplex meshes. In this case, no special
treatment is required. The classic algorithms without re-meshing are
bisection \cite{rivara1984a,rivara1984b,stevenson2007} and red-green
procedures \cite{bank1983}. Another approach is to simply re-generate
the mesh where the layout of the small and large elements depends on
some metric \cite{woopen2015, fidkowski2020, leng2020, may2021,
  bai2022, levy2023}. This mechanism is usually more computational
expensive but the resulting mesh is more economical. We would like to
mention that, however, $h$-nonconforming interfaces are typically
involved in local $h$-refinement for quadrilateral and hexahedral
meshes. Thanks to natural built-in mortars in HDG methods, the
relevant techniques can be easily utilized to treat nonconforming
interfaces. In addition, the issue of $p$-nonconforming interfaces due
to local enrichment of approximation space can also be addressed by
the mortar techniques. Nonetheless, special attention is needed for
curved boundaries \cite{sevilla2018, giacomini2019}. In this work,
though our $hp$-HDG approaches are valid for
triangular/tetrahedral/quadrilateral/hexahedral elements with straight
edges/faces in both 2D and 3D, our numerical results are only for
two-dimensional problems with triangle elements.

So far, we have reviewed various HDG schemes for solving different physical problems. Since each physics has a unique characteristic, it is natural to develop different numerical scheme for different problem.
However, the PDEs of Friedrichs' type \cite{friedrichs1958} embraces a large class of PDEs  with similar mathematical structure and this  provides an opportunity of developing a single unified framework. This idea is first adopted in a series of papers \cite{ern2006a,ern2006b,ern2008} in the analysis of DG methods. Friedrichs' system is also the basis to unify various discontinuous Petrov-Galerkin methods \cite{Tan2013}. In the work \cite{Tan2015}, the author uses Friedrichs' system to propose a unified and constructive framework for HDG schemes via a Godunov approach, with the assumption that the interfaces are conforming. 

This paper extends the work in \cite{Tan2015} in two important directions. First, our extension now provides a unified HDG framework for  PDEs with two-field structure (to be defined). Second we develop two unified $hp$-HDG frameworks: one for one-field PDE structure and another for two-field PDE structure. In particular, we consider Friedrichs' systems with more general assumptions that cover one- and two-field structures. For two-field structures, both full and partial coercivities are examined. The resulting system thus covers a wide range PDEs including hyperbolic, elliptic, or mixed-type PDEs. We propose two $hp$-HDG formulations: one for one-field PDEs and the other for two-field PDEs. The derivation heavily relies on the Godunov approach.

For the two-field formulation, we further exploit its intrinsic structure to obtain the corresponding reduced trace system. A few assumptions are identified to guarantee the existence of the numerical flux, and this is also a key to prove the well-posedness. Several numerical experiments are carried out to verify the effectiveness of the abstract $hp$-HDG formulations when applied to specific PDEs. In order to drive the adaptivity, an adhoc error indicator and an adjoint-based error estimation are implemented and their performance are compared. As shall be shown, using either of these criteria, numerically polluted areas induced by high gradient/discontinuity/singular can indeed decrease through the $hp$-adaptation process, and acceptable convergence rates can be attained in some cases.   

The paper is organized as follows. Section \secref{FdsSys}
briefly reviews Friedrichs' systems and outline important assumptions
that will be used in the well-posedness analysis. In Section \secref{hpHDGform}, key concepts about mortar techniques
are discussed, and HDG numerical fluxes, and the corresponding
$hp$-HDG formulations, for one-field and two-field Friedrichs' systems are derived. In addition, the well-posedness of these formulations is proved. The $hp$-adaptation strategy with adaptive criteria based on ad-hoc and adjoint-based error indicator is presented in Section \secref{strategy_hp}.  Several numerical examples for elliptic PDEs (with corner singularity, anisotropic diffusion with discontinuous boundary condition, heterogeneous anisotropic with discontinuous diffusivity field), linear hyperbolic PDE (with variable speed and discontinuous boundary condition), and convection-diffusion PDE (with boundary layer and discontinuous boundary condition) are presented in Section \secref{numerics}. Section \secref{conclusion} concludes the paper with future work.

\section{Linear PDEs of Friedrichs' type} \seclab{FdsSys}
The main
idea of Friedrichs' unification of PDEs \cite{friedrichs1958} is to cast wide classes of PDEs into the first order systems which share the same mathematical structure.
In this section, we outline 
one-field and two-field PDE of Friedrichs' type. The following notations are used in the paper.    Boldface lowercases are reserved for (column) vectors, uppercase letters are for matrices, and boldface uppercase letters are for third order tensors. 
Considering the following general system of linear PDEs  defined in a Lipschitz domain $\dom\subset\real^{d}$, where $d$ refers to the spatial dimension:
\beq\eqnlab{general_PDE}
    \FdsSum\FdsPartial\FdsFluxComp\LRp{\FdsVar} + \FdsG\FdsVar
    :=\FdsSum\FdsPartial\LRp{\FdsAk\FdsVar} + \FdsG\FdsVar = \Fdsforcing\text{ in }\dom,
\eeq
where ${\FdsFluxComp}\LRp{\FdsVar}:=\FdsAk \FdsVar$ is the $\FdsIndx$-th component of the flux tensor $\FdsFlux\LRp{\FdsVar}$, $\FdsVar$ the unknown solution with values in $\real^{\FdsNVar}$,  and $\Fdsforcing \in \real^{\FdsNVar}$ the forcing term. 
Here $\FdsPartial$ stands for the $\FdsIndx$-th (component-wise) partial derivative. Different types of constraints imposed on $\FdsAk$ and $\FdsG$ will result in different types of Friedrichs' systems, and we shall discuss each case separately in the following sub-sections.

\subsection{One-field Friedrichs' systems}
One-field Friedrichs' systems come with the following standard assumptions \cite{friedrichs1958,Jensen2005,ern2006a}:
\ben[label=(A.\arabic*)]
    \item \itmlab{A1}$\FdsG\in[\Lsp{\infty}(\dom)]^{\FdsNVar,\FdsNVar}$
    \item \itmlab{A2}$\forall\FdsIndx\in\LRc{1,\dots,d},\, \FdsAk\in\LRs{\Lsp{\infty}(\dom)}^{\FdsNVar,\FdsNVar}$ and $\FdsSum\FdsPartial\FdsAk\in\LRs{\Lsp{\infty}(\dom)}^{\FdsNVar,\FdsNVar}$.
    \item \itmlab{A3}$\forall\FdsIndx\in\LRc{1,\dots,d},\, \FdsAk=(\FdsAk)^T$ a.e. in $\dom$.
    \item \itmlab{A4}$\exists\mu_0>0,\,{\FdsG+ \FdsG^{T} + \sum_{\FdsIndx=1}^d\FdsPartial\FdsAk}\geq 2\mu_0\id_{\FdsNVar}$ a.e. in $\dom$.
\een
where $\id_{\FdsNVar}$ is a $\FdsNVar\times\FdsNVar$ identity matrix. In this paper, the inequality of the type \itmref{A4} stands for the semi-positive definiteness of the difference between the matrices on the left-hand side and  the right-hand side. 
Inequality \itmref{A4} is also  known as  full coercivity \cite{ern2006a}.  Here, $[\Lsp{\infty}(\dom)]^{\FdsNVar,\FdsNVar}$ denotes the space of $\FdsNVar\times \FdsNVar$ matrix-valued essentially bounded functions on $\dom$. 
It turns out that any symmetric and strictly hyperbolic PDE system is an example of one-field Friedrichs' system. For example, an advection-reaction equation falls into this category.

\subsection{Two-field Friedrichs' systems}
Let  $\FdsNVarAux$ and $\FdsNVarSol$ be two positive integers such that $\FdsNVar=\FdsNVarAux + \FdsNVarSol$. Denote $\Hsp_{\FdsVarAux}:=\LRs{\Lsp{2}\LRp{\dom}}^{\FdsNVarAux}$, $\Hsp_{\FdsVarSol}:=\LRs{\Lsp{2}\LRp{\dom}}^{\FdsNVarSol}$, and $\Hsp:=\Hsp_{\FdsVarAux}\times\Hsp_{\FdsVarSol}$, where $\Lsp{2}\LRp{\dom}$ is the space of square-integrable functions on $\dom$. Suppose we have the decomposition $\FdsVar=\LRp{\FdsVarAux,\FdsVarSol}$ for all $\FdsVar\in\Hsp$, and
\beq\eqnlab{Fds_block_mat}
\FdsG =
\begin{bmatrix}
\FdsGaa & \FdsGas\\
\FdsGsa & \FdsGss
\end{bmatrix},\quad\quad
\FdsAk =
\begin{bmatrix}
\FdsAkaa & \FdsBk\\
\FdsBk^T & \FdsCk
\end{bmatrix}, \quad \FdsIndx\in\LRc{1,\dots,d}.
\eeq
Two additional key assumptions on which the two-field theory is based are \cite{ern2006b}:
\ben[resume,label=(A.\arabic*)]
    \item \itmlab{A5} $\FdsAkaa=0$, $\forall\FdsIndx\in\LRc{1,\dots,d}$, 
    \item \itmlab{A6} $\FdsGaa\geq k_0 \id_{\FdsNVarAux}$ for some $k_0>0$, 
\een
where $\id_{\FdsNVarAux}$ is the identity matrix in $\real^{\FdsNVarAux,\FdsNVarAux}$. Assumptions \itmref{A5}-\itmref{A6} allow us to eliminate the $\FdsVarAux$-component of $\FdsVar$ in the PDE system and the resulting differential equation is an elliptic-like PDE for the $\FdsVarSol$-component. 
The two-field Friedrichs' systems that satisfy assumptions \itmref{A1}-\itmref{A6} cover wide variety of PDEs including convection-diffusion-reaction equation, compressible linear continuum mechanics with reaction term, and simplified MHD \cite{ern2006b}. 

We note that the positivity condition \itmref{A4} can be further relaxed to account for systems that have two-field structures with partial coercivity. This class includes convection-diffusion, anisotropic diffusion, and typical compressible linear continuum mechanics (e.g., linearized compressible elasticity or linearized compressible Navier-Stokes) equations, to name a few. This can be accomplished (see \cite{ern2008}) 
by replacing assumption \itmref{A4} with the following:\\ 
\ben[label=(A4.\alph*)]
    \item \itmlab{A4a}$\exists\mu_0>0,\,{\FdsG + \FdsG^{T} + \FdsSum\FdsAk}  \geq 2\mu_0\id^{o}_{\FdsNVarAux}$ a.e. in $\dom$, where $\id^{o}_{\FdsNVarAux}$ is an $\FdsNVar\times\FdsNVar$ matrix defined as $\id^{o}_{\FdsNVarAux}:=\begin{bmatrix}
    \id_{\FdsNVarAux} & 0\\
    0 & 0
    \end{bmatrix}$.
    \item \itmlab{A4b}$\FdsGas=(\FdsGsa)^T=0$ and $\FdsBk$ are constant over $\dom$.
\een


\begin{rema}
Here, we omit one additional inequality ((A3B") in \cite{ern2008}) required for two-field Friedrichs' systems with partial coercivity since it will not be used in our analysis of HDG. However, the inequality is critical in the proof of well-poseness of the continuous PDE stated in \eqnref{general_PDE}. Such an inequality can be viewed as the generalized form of Friedrichs–Poincaré \cite{brenner2003a} or Korn's \cite{brenner2003b} inequality, and the discrete version of it can actually be used in our analysis which, however, will lead to a mesh-dependent HDG scheme.
\end{rema}

\subsection{Boundary conditions}
Though the numerical results in Section \secref{numerics} use non-homogeneous boundary conditions, it is sufficient to show the well-posedness of the one-field setting \itmref{A1}-\itmref{A4} and the two-field setting \itmref{A1}-\itmref{A6} (or \itmref{A1}-\itmref{A3}, \itmref{A4a}-\itmref{A4b}, \itmref{A5}-\itmref{A6}) with homogeneous boundary condition.   Similar to \cite{ern2006a,ern2006b,ern2008}, we consider a general homogeneous boundary condition in the following form
\begin{subequations}\eqnlab{Abstrct_BC}
\begin{align}
\begin{split}\eqnlab{General_BC}
    \LRp{\FdsABnd - \FdsM}\FdsVar=\bs{0}
\end{split}\\
\intertext{where $\FdsM\in\real^{\FdsNVar,\FdsNVar}$ and $\FdsABnd:=\FdsSum\normalComp{\FdsIndx}\FdsAk$ with $\normal=\LRp{\normalComp{1},\dots,\normalComp{d}}^T$ being unit outward vector of $\domBnd$. In addition, we assume that}
\begin{split}\eqnlab{Abstrct_M1}
    \FdsM \geq 0,
\end{split}\\
\begin{split}\eqnlab{Abstrct_M2}
    \Null{\FdsABnd-\FdsM} + \Null{\FdsABnd+\FdsM} = \real^{\FdsNVar}
\end{split}
\end{align}
\end{subequations}
with $\Null{\cdot}$ denoting the nullspace of its argument. It should be noted that there is no unique definition of $\FdsM$ and different choices will lead to different boundary conditions.

\section{\texorpdfstring{$hp$}{hp}-HDG Formulations}\seclab{hpHDGform}
In this section, we are going to derive the $hp$-HDG formulations
for  Friedrichs'
systems outlined in  Section \secref{FdsSys}.
Toward formulating an $hp$-HDG scheme, it is essential to derive
a numerical flux due to discontinuous approximation space(s) used for
the volume unknown(s). 
The well-known Godunov approach, which involves solving the Riemann problem either exactly or approximately, is one of the most popular methods to construct numerical fluxes.
The key to realize is that the Godunov flux can be hybridized
\cite{Tan2015}. In the other words, the Godunov flux\footnote{It
should be noted that the Godunov flux is simply an upwind flux if the
problem of interest is linear and thus we may use these two terms
interchangeably in this paper without confusion.} can be defined
implicitly along with trace unknown(s) and thus can be employed as an
HDG numerical flux. In addition, such an approach is desirable since
it can lead to a parameter-free scheme. On the other hand, the key
ingredient to handle $hp$-nonconforming interfaces is to construct
such flux directly on the mortars which are naturally built-in HDG
methods. As we shall show, this can be achieved with the specific choice of the
configuration of the mortars and the approximation space(s) of trace
unknown(s).

\subsection{Nomenclatures}
This section collects notations and conventions for the rest of the
paper.  A partition $\domPart$ of the domain $\dom\subset\real^{d}$ is
a finite collection of disjoint elements $\elem$ such that
$\cup_{\elem \in \domPart} \overline{\elem}= \overline{\dom}$ where
the mesh size $h$ is defined as $\max_{\elem\in\domPart}
diam(\elem)$. For the simplicity of the exposition, we will use
two-dimensional simplex elements to convey our idea, though our
approach is valid for three-dimensional settings as well.
The set of elemental boundaries 
is denoted by $\domPartBnd=\LRc{\elemBnd\mid\elem\in\domPart}$ 
each of which comes with unit outward normal vector $\normal^{\elem}$. We conventionally identify $\normal$ as the normal vector on the boundary $\elemBnd$ of element $\elem$ (also denoted as $\elemM$) and $\normal^+$ = $-\normal^-$ as the normal vector of the boundary of a neighboring element (also denoted as $\elemP$).

An element $\elem^+$ is said to be a neighbor of the element $\elem^-$ when $\elemBnd^+\cap\elemBnd^-$ has a positive $d-1$ Lebesgue measure.
For an element $K$ of the partition $\domPart$, we define a face of the element $\elem\in\domPart$ by $\face\in\elemBnd$. 
For an interior interface (nonconforming or not), we introduce a  \tit{mortar} $\mort$ 
as $\mort=\elemBnd^+\cap\elemBnd^-$, and  $\mort=\elemBnd\cap\domBnd$ on the boundary of $\dom$. 
For any conforming interface the mortar $\mort$ is clearly defined and for any nonconforming interface the mortar is defined in Section \secref{hNoncon}. The collection of mortars, called \tit{mesh skeleton}, is denoted by $\skel$,
$\skel=\skelInt\bigcup\skelBnd$ with 
$\skelBnd = \LRc{\mort\in\skel \mid \mort\subset\domBnd}$ and $\skelInt = \skel \setminus \skelBnd$. 
The derivation of an HDG scheme is centered on HDG numerical flux which typically comes with the newly introduced unknowns residing on the skeleton. Such unknowns are usually termed trace unknowns while the usual unknowns defined within elements, such as the ones in the DG methods, are termed as volume unknowns.

For the quantity $f$ that is  possibly double-valued  on the mesh skeleton, we define the jump of $f$ on $\mort\in\skel$ as:
\begin{align*}
    &\jump{f} = f^- + f^+,\quad\text{for }\forall\mort\in\skelInt,
    &&\jump{f} = f,\quad\text{for }\forall\mort\in\skelBnd,
\end{align*}
where $f^{\pm}(x)=\lim_{\substack{y \rightarrow x \\ y\in \face^{\pm}}}f(y)$ and $\face^{\pm}\in\elemBnd^{\pm}$. 

We define $\polySpc{p}\LRp{D}$ as the space of polynomials of degree at most $p$ on a domain $D$ and $\Lsp{2}\LRp{D}$ as the space of square integrable functions on $D$. In particular, we denote the degree of polynomials in an element $\elem$ by $\polyElem$ and on a mortar $\mort$ by $\polyMort$. Next, we introduce discontinuous piecewise polynomial spaces
\begin{align*}
    &\FdsVarApprSpc{}=\LRs{\ApprSpc{}}^{\FdsNVar},
    &&\ApprSpc{}:=\LRc{z_h \in \Lsp{2}(\domPart) :  \eval{z_h}{\elem}\in\polySpc{\polyElem}\LRp{\elem},\ \forall \elem\in\domPart},\\
    &\FdsTrcVarApprSpc{}=\LRs{\TrcApprSpc{}}^{\FdsNVar},
    &&\TrcApprSpc{}:=\LRc{\widehat{z}_h\in \Lsp{2}(\skel)  :  \eval{\widehat{z}_h}{\mort}\in\polySpc{\polyMort}\LRp{\mort},\ \forall\mort\in\skel}.
\end{align*}
To account for various boundary conditions, we denote $\domBndDir$ as Dirichlet type of boundary, $\domBndNmn$ as Neumann type of boundary, and $\domBndRb$ as Robin type boundary. The boundary now can be decomposed as $\domBnd=\domBndDir\cup\domBndNmn\cup\domBndRb$ where the intersections of any two types of boundaries are empty set. 
To facilitate the discussion of the two-field Friedrichs' system later, we further introduce some additional approximation spaces:
\begin{align*}
    &\FdsVarAuxApprSpc{}=\LRs{\ApprSpc{}}^{\FdsNVarAux},
    &&\FdsVarSolApprSpc{}=\LRs{\ApprSpc{}}^{\FdsNVarSol},\\
    &\FdsTrcVarAuxApprSpc{}=\LRs{\TrcApprSpc{}}^{\FdsNVarAux},
    &&\FdsTrcVarSolApprSpc{}=\LRs{\TrcApprSpc{}}^{\FdsNVarSol}.
\end{align*}
 Finally, we define the inner product for the aforementioned finite element spaces. $\LRp{\cdot,\cdot}_{D}$ is defined as the $\Lsp{2}$-inner product on a domain $D\in\real^d$ and  $\LRa{\cdot,\cdot}_{D}$ as the $\Lsp{2}$-inner product  on a domain $D$ if $D\in\real^{d-1}$. To make our presentation more concise, we introduce the following definitions:
\begin{equation*}\small
    \begin{cases}
    \LRp{\cdot,\cdot}_{\domPart}:=\sum_{\elem\in\domPart}\LRp{\cdot,\cdot}_{\elem},\\
    \LRa{\cdot,\cdot}_{\domPartBnd}:=\sum_{\elemBnd\in\domPartBnd}\sum_{\face\in\elemBnd}\sum_{\mort\subseteq\face}\LRa{\cdot,\cdot}_{\mort},\\
    \LRa{\cdot,\cdot}_{\elemBnd}:=\sum_{\face\in\elemBnd}\sum_{\mort\subseteq\face}\LRa{\cdot,\cdot}_{\mort},\\
    \LRa{\cdot,\cdot}_{\skel}:=\sum_{\mort\in\skel}\LRa{\cdot,\cdot}_{\mort},\\
    \LRa{\cdot,\cdot}_{\skelBnd}:=\sum_{\mort\in\skelBnd}\LRa{\cdot,\cdot}_{\mort}.
    \end{cases}\,
\end{equation*}

\subsection{Mortar-based technique}\seclab{mortar}
A mortar technique is characterized by the introduction of mortars,
finite element spaces on the mortars, and the method that uses mortars
to patch the subdomains/elements.  Our mortar approach is built upon
four mortar approaches, all of which share the aforementioned three
steps, and the key difference is the way they compute the mortar
unknowns. The first approach is due to
\cite{Maday1988,anagnostou1990,bernardi1989,Bernardi1993,bernardi2005,barth2005},
originally developed for elliptic PDEs, that uses the mortar unknowns
to weakly maintain the continuity of the solution across the
mortars. In this case, the mortar unknowns are solved together with
the volume unknowns on subdomains or elements.  The second approach
was developed in \cite{Kopriva1996,Kopriva2002, karniadakis2005,
  Tan2012,friedrich2018} for hyperbolic PDEs in the context of
spectral element and DG approaches. The upwind states, which can be
considered as the mortar unknowns \cite{Tan2015}, are computed on the
mortars to construct the numerical fluxes to ensure flux conservation
across the mortars. The third approach by
\cite{belgacem1999,wheeler2000,arbogast2000,Pencheva2003}, originally
developed in the context of mixed finite element methods for
elliptic-type PDEs, calls the mortar unknowns as Lagrange multipliers
and, similar to the first approach, they are solved together with the
volume unknowns. The key difference is that the weak continuity of the
flux is enforced instead of the weak continuity of the
solution. Finally, the HDG approach \cite{Cockburn2009a,
  Cockburn2016a,Cockburn2010,kirby2012,
  Nguyen2009a,Nguyen2009b,Fu2014,
  Cockburn2009b,Cockburn2011,Nguyen2010,Cockburn2014, kang2019,
  lee2019, Tan2015}, in which the mortar unknowns are called trace
unknowns, uses mortar unknowns to enforce the weak continuity of the
flux similar to the third approach. The mortar unknowns are also
solved together with the volume unknowns.
 
In this paper, we extend the HDG built-in mortars to fully account for $hp$-nonconforming interfaces. To that end, two ingredients are required. They include the appropriate choice of mortar configuration and the finite element space defined over the mortars. We will show that our choice, without any additional interpolation or projection, can lead to a setting where the Riemann problem is well-defined and the numerical flux can be derived via the Godunov approach.

\subsubsection{\texorpdfstring{$h$}{h}-nonconforming interfaces}\seclab{hNoncon}
In $h$-nonconforming interfaces (see Figure \figref{noncon_interface}), $\faceP$ is not necessarily equal to $\faceM$. We hence need to carefully consider the definition of mortar. There are two options for constructing a mortar as shown in Figure \figref{mortar}. One is a set of \tit{split-sided} mortars (i.e. Figure \figref{split_mortar}) which conforms to the smaller sides of the adjacent elements, while the other is \tit{full-sided} mortar (i.e. Figure \figref{full_mortar}) which conforms to the larger side of the adjacent element. In the context of HDG methods, the first option is used in \cite{Chen2012, Egger2013, Chen2014} and the second one is used \cite{Samii2016, Dahm2014, muixi2020}. Although the usage of the full-sided mortar is less computationally intensive, the first type of mortar is chosen in this work to facilitate the implementation of the Godunov approach. 
\begin{figure}[!htb]
\centering
\subfloat[A nonconforming interface\label{fig:noncon_interface}]{
	\includegraphics[width=50mm]{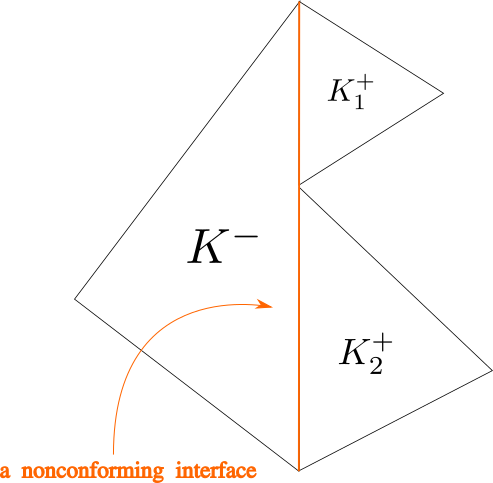}
}
\subfloat[Split-sided mortars\label{fig:split_mortar}]{
	\includegraphics[width=50mm]{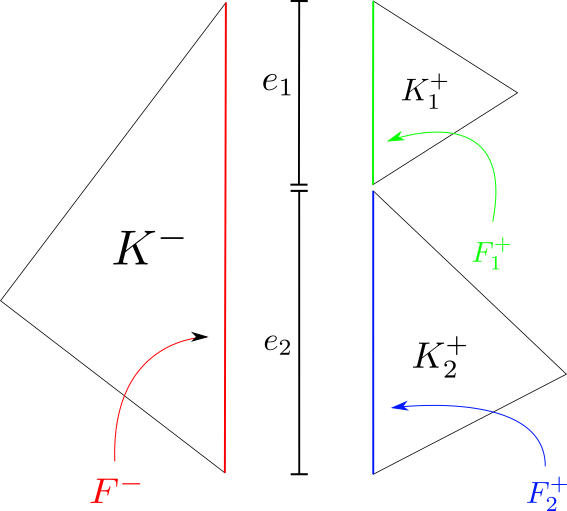}
}
\subfloat[A full-sided mortar\label{fig:full_mortar}]{
	\includegraphics[width=50mm]{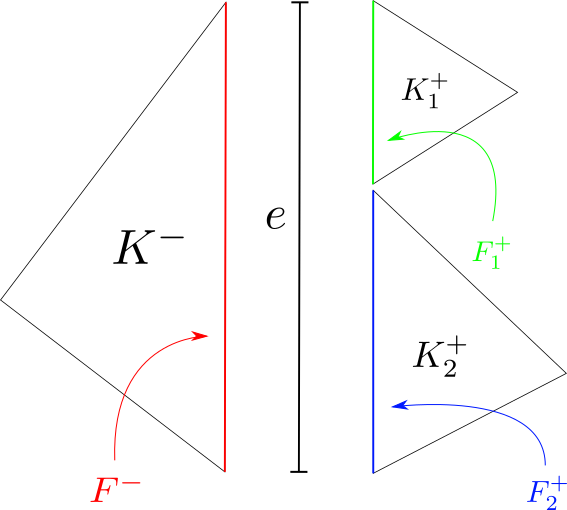}
}
\caption{Different options of mortars on a nonconforming interface.}  
\figlab{mortar}
\end{figure}

\begin{rema}
It should be noted that in \cite{Kozdon2018} the authors showed
theoretically that either full- or split-sided mortar can lead to a
stable DG scheme on a discrete level for a time-dependent linear
elasticity problem. Additionally, they also showed numerically that
both type of mortars can lead to conservative schemes. However,
split-sided mortar is still suggested in \cite{Kozdon2018} in the
sense that \ben
    \item it is the most natural approach for DG methods, 
    \item full-sided mortar has a spectral radius more than twice as large as the split-side mortar for the problem considered (hence, more restrictive time step size for explicit methods).
\een
\end{rema}

\subsubsection{\texorpdfstring{$p$}{p}-nonconforming interfaces}\seclab{pNoncon}
For  $p$-nonconforming interfaces, $\polyElemP = \polyElemM$ does not hold in general. Furthermore, the degree of approximation of trace unknowns denoted by $\polyMort$ could differ from $\polyElemP$ or $\polyElemM$. In this work, we choose:
\begin{equation}\eqnlab{polyMort}
\begin{cases}
    \polyMort = \max\LRc{\polyElemP,\polyElemM},
    \text{ for }\forall\mort\in\skelInt,\\
    \polyMort = \polyElem,
    \text{ for }\forall\mort\in\elemBnd\cap\skelBnd,
\end{cases}
\end{equation}
to facilitate the Godnuov approach and stability. This choice is also suggested in \cite{Cockburn2009a}.  

\subsubsection{\texorpdfstring{$hp$}{hp}-nonconforming interfaces}\seclab{hpNoncon}
By combining the setting presented in Section \secref{hNoncon} and Section \secref{pNoncon}, we can now handle $hp$-nonconforming interfaces and construct an $hp$-HDG scheme for the Friedrichs' system using the Godnuov approach. In addition, neither projection nor interpolation is required thanks to our specific selection of the configuration of mortars and of the degree of approximation of the trace unknowns. 

To illustrate the idea, we consider the nonconforming interface shown
in Figure. \figref{split_mortar} and focus on the segment
$\mort_1$. It is not clear how to implement the Godnuov approach since
the Riemann problem is not well-defined. The left state and right
state are defined on the domains that do not conform to each other
(i.e $\faceM\neq\faceP_1$). To resolve this issue, we can either
project and interpolate the states onto the mortar $\mort_1$. Then,
the Godnuov approach can be applied by solving the Riemann problem
that is properly defined by these intermediate states. This
methodology is already proven to be successful in the context of DG
methods \cite{Kopriva1996,Kopriva2002, karniadakis2005,
  Tan2012,friedrich2018}. {\em Throughout  the paper we assume that all the edges/faces are straight, that is, the meshes are affine. Curved elements are more delicate to treat and this will be part of our future work.} Owing to the natural built-in mortar in HDG
methods and the way we handle the nonconforming interfaces, both
projection, and interpolation are actually implicitly implied. To see
it, we consider the following piecewise polynomial functions
$\FdsVarAppr\in\FdsVarApprSpc{}$,
$(\FdsVarAppr^*)^{\pm}\in\FdsTrcVarApprSpc{}$ and
$\FdsTrcTestAppr\in\FdsTrcVarApprSpc{}$. Moreover, we define a
projection operator $\projOp{\cdot}$ that is the $\Lsp{2}$-projection
into the space $\FdsTrcVarApprSpc{}$. The projection from left state
defined on $\faceM$ to the left intermediate state $(\FdsVarAppr^*)^-$
can be stated as
$\projOp{\eval{\FdsVarAppr}{\faceM\cap\mort_1}}=\eval{(\FdsVarAppr^*)^-}{\mort_1}$
and the equality holds in the sense that \beq
\LRa{\FdsVarAppr,\FdsTrcTestAppr}_{\mort_1} =
\LRa{(\FdsVarAppr^*)^-,\FdsTrcTestAppr}_{\mort_1}\quad\forall\FdsTrcTestAppr\in\FdsTrcVarApprSpc{}.
\eeq Since we use split-sided mortars and choose degree approximation of
the trace test space $\FdsTrcVarApprSpc{}$ by Eq. \eqnref{polyMort},
it is obvious that for any polynomial function
$\f\in\polySpc{\polyElemM}\LRp{\elemM}$ it has to satisfy that
$\eval{\f}{\faceM\cap\mort_1}\subseteq\polySpc{\polyMort}\LRp{\mort_1}$. Due
to the unique representation of polynomials, the projection actually
does nothing here, and hence
$\projOp{\eval{\FdsVarAppr}{\faceM\cap\mort_1}}=\eval{\FdsVarAppr}{\faceM\cap\mort_1}$. As
consequence, the left intermediate state is nothing but just the
restriction of the left state:
$\eval{(\FdsVarAppr^*)^-}{\mort_1}=\eval{\FdsVarAppr}{\faceM\cap\mort_1}$. The
same argument can also be made in terms of interpolation. Similarly,
we have the right intermediate state
$\eval{(\FdsVarAppr^*)^+}{\mort_1}=\eval{\FdsVarAppr}{\faceP\cap\mort_1}$. Now
the upwind numerical flux can be constructed by solving the Riemann
problem locally along the normal $\normal$ of the segment
$\faceM\cap\mort_1$. Given that being along in a single direction is
one-dimension in nature, the normal vector $\normal$ can be
parameterized by some scalar $n$ where $n=0$ will correspond to the
location of the mortar $\mort_1$. By extending the definition of the
coefficient matrix $\FdsABnd:=\FdsSum\normalComp{\FdsIndx}\FdsAk$ with
$\normal=\LRp{\normalComp{1},\dots,\normalComp{d}}^T$ being unit
outward vector of $\elemBnd$ for $\forall\elem\in\domPart$, the
statement of the Riemann problem \cite{toro1999} reads: find
$\FdsVarAppr\LRp{n,t}$ such that \beq\eqnlab{local_Riemann}
\pp{\FdsVarAppr}{t}+\pp{\LRp{\FdsABnd\FdsVarAppr}}{n}=\bs{0}, \eeq
with initial condition $\FdsVarAppr\LRp{n,0}=(\FdsVarAppr^*)^-$ for
$n<0$, $\FdsVarAppr\LRp{n,0}=(\FdsVarAppr^*)^+$ for $n>0$. Here,
(artificial) time $t$ is introduced to help understand the Godunov
flux via the Riemann problem, but it is otherwise not necessary in the
derivation. Figure \figref{riemann} illustrates the idea of how the
Riemann problem is defined in direction $\normal$ that is parametrized
by $n$. With the well-defined problem \eqnref{local_Riemann}, we are
now in the position to derive upwinding HDG flux by following the
procedure outlined in the work \cite{Tan2015}. In this paper, the coefficient matrix $\FdsABnd$ will be assumed to be continuous across the
mesh skeleton\footnote{This condition can be relaxed, but we will use
it to keep the presentation concise.}. In particular, $\FdsABnd$ is
symmetric according to \itmref{A3} and hence its eigen-decomposition
is guaranteed to exist. We thus can also defined
$\FdsAbsA:=\FdsEigenVec\abs{\FdsEigenVal}\FdsEigenVec^{-1}$ where
$\FdsEigenVal:=diag(\lambda_1,\dots,\lambda_{\FdsNVar})$ and
$\lambda_i$ are eigenvalues of $\FdsABnd$, and $\FdsEigenVec$ is the
matrix composed by the corresponding eigenvectors.
\begin{figure}[!htb]
\centering
\includegraphics[width=100mm]{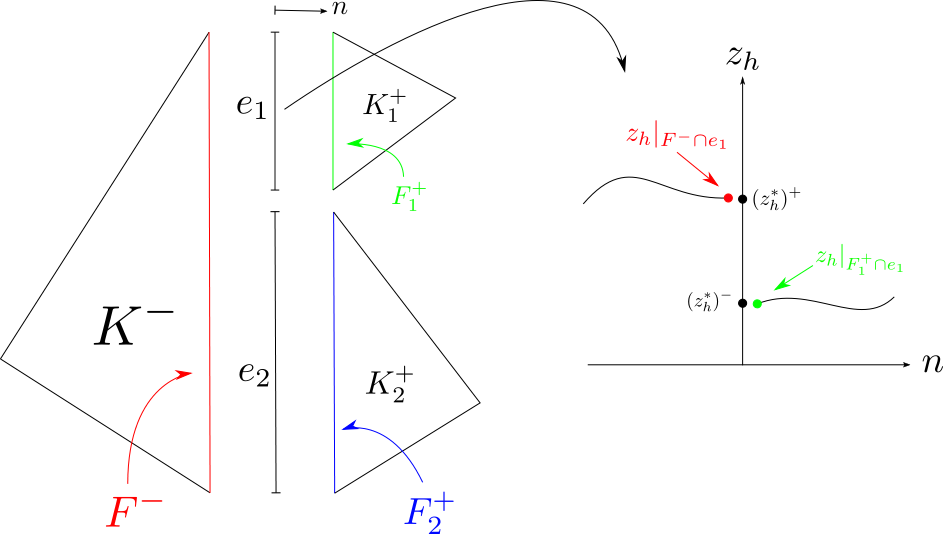}
\caption{Illstration of how the Riemann problem is defined along a normal direction $\normal$.}  
\figlab{riemann}
\end{figure}
\subsection{A constructive derivation of an $hp$-HDG formulation}\seclab{hp_hdg}
\subsubsection{Friedrichs' system with one-field structure}\seclab{Fds_one}
In this section, we derive $hp$-HDG formulation for linear PDE in Eq. \eqnref{general_PDE} that satisfies one-field Friedrichs' system assumptions \itmref{A1}-\itmref{A4}.
To begin with, we apply Galerkin approximation to Eq. \eqnref{general_PDE} on an element $\elem\in\domPart$ together with integration by parts. The resulting local problem reads: seek $\FdsVarAppr\in\FdsVarApprSpc{}$ such that
\beq\eqnlab{Fds_one_hp_local_undefined}\small
\begin{aligned}
    -\FdsSum\LRp{{\FdsAk}\FdsVarAppr,{\FdsPartial}\FdsTestAppr}_{\elem} 
    + \LRp{{\FdsG}\FdsVarAppr,\FdsTestAppr}_{\elem} 
    + \LRa{\upwindFlux\LRp{\FdsVarAppr}\normal,{\FdsTestAppr}}_{\elemBnd}
    = \LRp{{\bs{\forcing}},\FdsTestAppr}_{\elem},
    \quad\forall\elem\in\domPart,
\end{aligned}
\eeq
 for all $\FdsTestAppr\in\FdsVarApprSpc{}$ and the flux $\upwindFlux\LRp{\FdsVarAppr}$ is a tensor in which each component is a $\FdsNVar\times d$ matrix. As result, $\upwindFlux\LRp{\FdsVarAppr}\normal$ is also an $\FdsNVar\times d$ matrix. By treating nonconforming interfaces in the fashion presented in Sec. \secref{hpNoncon}, the normal flux $\upwindFlux\LRp{\FdsVarAppr}\normal$ on $\mort\in\skel$ is still not well-defined since the traces of both $\FdsVarAppr^-$ of element $\elemM$ and $\FdsVarAppr^+$ of element $\elemP$ co-exist on $\mort$. However, it can be resolved by Godunov-type methods \cite{godunov1959} through first solving, either exactly or approximately, the Riemann problem \eqnref{local_Riemann} for the upwind sate at the mortar $\mort$ and then introducing the upwind numerical flux $\upwindFluxMort\LRp{\FdsVarAppr^-,\FdsVarAppr^+}\normal$. Furthermore, as reported in \cite{Tan2015}, such flux is hybridizable. The upwind-based HDG flux can then be constructed by replacing the upwind state with the designated trace unknown. Following this procedure, the upwind HDG flux 
 reads:
\beq\eqnlab{Fds_one_hdg_flux}
\upwindFluxTrc\LRp{\FdsVarAppr,\FdsTrcVarAppr}\normal := \FdsABnd\FdsVarAppr + \FdsAbsA\LRp{\FdsVarAppr - \FdsTrcVarAppr}
\eeq 
where the right hand side (RHS) of \eqnref{Fds_one_hdg_flux} is nothing but the upwind flux in its one-sided form when $\FdsTrcVarAppr$ is the Riemann solution. Note that $\FdsAbsA$ can also be replaced by some other stability parameter matrix $\stabPar$ and that will result in different numerical fluxes (also see the discussion in \cite{Tan2015}). By replacing $\upwindFlux\LRp{\FdsVarAppr}\normal$ by $\upwindFluxTrc\LRp{\FdsVarAppr,\FdsTrcVarAppr}\normal$, we arrive at the so-called \tit{local equations}: seek $\LRp{\FdsVarAppr,\FdsTrcVarAppr}\in\FdsVarApprSpc{}\times\FdsTrcVarApprSpc{}$ such that
\beq\eqnlab{Fds_one_hp_local}\small
\begin{aligned}
    -\FdsSum\LRp{{\FdsAk}\FdsVarAppr,{\FdsPartial}\FdsTestAppr}_{\elem} 
    + \LRp{{\FdsG}\FdsVarAppr,\FdsTestAppr}_{\elem} 
    + \LRa{{\FdsABnd}{\FdsVarAppr}+{\FdsAbsA}\LRp{{\FdsVarAppr}-\FdsTrcVarAppr},{\FdsTestAppr}}_{\elemBnd}
    = \LRp{{\bs{\forcing}},\FdsTestAppr}_{\elem},
    \quad\forall\elem\in\domPart,
\end{aligned}
\eeq
for all $\FdsTestAppr\in\FdsVarApprSpc{}$. To close the system, we still require one more constraint. This can be achieved by weakly enforcing the continuity of the normal numerical flux \eqnref{Fds_one_hdg_flux} across the mortars: for $\LRp{\FdsVarAppr,\FdsTrcVarAppr}\in\FdsVarApprSpc{}\times\FdsTrcVarApprSpc{}$,
\begin{equation}\eqnlab{Fds_one_hp_global}
    \LRa{\jump{{\FdsABnd}{\FdsVarAppr}+{\FdsAbsA}\LRp{{\FdsVarAppr}-\FdsTrcVarAppr}},\FdsTrcTestAppr}_{\mort} 
    = 0, \quad\forall\mort\in\skelInt,
\end{equation}
is enforced for $\forall\FdsTrcTestAppr\in\FdsTrcVarApprSpc{}$. The equation \eqnref{Fds_one_hp_global} is called as \tit{conservativity condition} \cite{Cockburn2009a} since it will guarantee that the scheme is locally conservative. In addition, it couples all volume unknowns and hence is referred to as a \tit{global equation}. On the boundary, it is clearer to enforce non-homogeneous version of boundary conditions \eqnref{General_BC} 
directly through the trace unknown which is already defined on $\skelBnd$:
\beq\eqnlab{Fds_one_hdg_bc1}
    \LRa{\LRp{{\FdsABnd}-{\FdsMBnd}}\FdsTrcVarAppr,\FdsTrcTestAppr}_{\mort}
    = \LRa{ \LRp{{\FdsABnd}-{\FdsMBnd}}\bs{g}, \FdsTrcTestAppr}_{\mort},
    \quad\forall\FdsTrcTestAppr\in\FdsTrcVarApprSpc{}\,\text{and}\,\forall\mort\in\skelBnd,   
\eeq
where we set $\FdsMBnd:=\FdsAbsA$ and the function $\bs{g}:\domBnd\rightarrow\real^{\FdsNVar}$ is defined as
\begin{equation}\eqnlab{Fds_one_Dirichlet}
    \bs{g} := \begin{cases}
    \bs{\DirVal}\quad&\text{if}\,\LRp{\FdsABnd-\FdsMBnd}\neq0,\\
    \bs{0}&\text{if}\,\LRp{\FdsABnd-\FdsMBnd}=0,
    \end{cases}
\end{equation}
where $\bs{\DirVal}$ is the Dirichlet data and is set to be zero for the homogeneous boundary condition. It should be noted that the equation \eqnref{Fds_one_hdg_bc1} corresponds to the "inflow" boundary condition \cite{Tan2015}. In addition, \eqnref{Fds_one_hdg_bc1} is simply analogous formula to its  continuous version stated in \eqnref{General_BC}. Since \eqnref{Fds_one_hdg_bc1} only specify inflow condition, it is clear that $\FdsTrcVarAppr$ cannot be uniquely determined on the outflow. Thus, we further require that   
\beq\eqnlab{Fds_one_hdg_bc2}
        \LRa{{\FdsABnd}{\FdsVarAppr}+{\FdsAbsA}\LRp{{\FdsVarAppr}-\FdsTrcVarAppr},\FdsTrcTestAppr}_{\mort} 
        = \LRa{{\FdsABnd}\FdsTrcVarAppr,\FdsTrcTestAppr}_{\mort},
        \quad\forall\FdsTrcTestAppr\in\FdsTrcVarApprSpc{}\,\text{and}\,\forall\mort\in\skelBnd.
\eeq
Equation \eqnref{Fds_one_hdg_bc2} is resulted from maintaining consistency of the numerical flux, and corresponds to outflow conditions. In fact, \eqnref{Fds_one_hdg_bc1} and \eqnref{Fds_one_hdg_bc2} can be incorporated into a single equation as:
\beq\eqnlab{Fds_one_hdg_bc}
    \LRa{\jump{{\FdsABnd}{\FdsVarAppr}+{\FdsAbsA}\LRp{{\FdsVarAppr}-\FdsTrcVarAppr}},\FdsTrcTestAppr}_{\skelBnd}
    = -\LRa{ \half\LRp{{\FdsABnd}-{\FdsMBnd}}\bs{g}, \FdsTrcTestAppr}_{\skelBnd} 
    + \LRa{\half\LRp{{\FdsABnd}+{\FdsMBnd}}\FdsTrcVarAppr, \FdsTrcTestAppr}_{\skelBnd}.
\eeq
In this paper, we will work with the general form of boundary condition \eqnref{Fds_one_hdg_bc} for one-field Friedrichs' system.
The complete $hp$-HDG formulation for the one-field Friedrichs' system is composed by Eq. \eqnref{Fds_one_hp_local}, Eq. \eqnref{Fds_one_hp_global}, and Eq. \eqnref{Fds_one_hdg_bc} together: seek $\LRp{\FdsVarAppr,\FdsTrcVarAppr}\in\FdsVarApprSpc{}\times\FdsTrcVarApprSpc{}$ such that \footnote{Although $\bs{g}$ is set to be zero, we still keep it in the right-hand side of the global equation so that reader can easily observe that inflow and outflow boundaries are switched in the adjoint $hp$-HDG formulation.}
\begin{subequations}\eqnlab{Fds_one_hp_hdg}\small
\begin{align}
\begin{split}\eqnlab{Fds_one_hp_hdg_local}
    -\FdsSum\LRp{{\FdsAk}\FdsVarAppr,{\FdsPartial}\FdsTestAppr}_{\domPart} 
    + \LRp{{\FdsG}\FdsVarAppr,{\FdsTestAppr}}_{\domPart} 
    + \LRa{{\FdsABnd}{\FdsVarAppr}+{\FdsAbsA}\LRp{{\FdsVarAppr}-\FdsTrcVarAppr},{\FdsTestAppr}}_{\domPartBnd} 
    = \LRp{{\bs{\forcing}},\FdsTestAppr}_{\domPart},&
\end{split}\\
\begin{split}\eqnlab{Fds_one_hp_hdg_global}
    \LRa{\jump{{\FdsABnd}{\FdsVarAppr}+{\FdsAbsA}\LRp{{\FdsVarAppr}-\FdsTrcVarAppr}},\FdsTrcTestAppr}_{\skel}
    = -\LRa{ \half\LRp{{\FdsABnd}-{\FdsMBnd}}\bs{g}, \FdsTrcTestAppr}_{\skelBnd} 
    + \LRa{\half\LRp{{\FdsABnd}+{\FdsMBnd}}\FdsTrcVarAppr, \FdsTrcTestAppr}_{\skelBnd},&
\end{split}
\end{align}
\end{subequations}
for all $\LRp{\FdsTestAppr,\FdsTrcTestAppr}\in\FdsVarApprSpc{}\times\FdsTrcVarApprSpc{}$. We now show that the numerical scheme in \eqnref{Fds_one_hp_hdg} is trivially locally and globally conservative, and furthermore well-posed.

\begin{lemma}[Local conservation]\lemlab{Fds_one_local_consrv}
The $hp$-HDG scheme in \eqnref{Fds_one_hp_hdg} is locally conservative.
\end{lemma}
\begin{proof}
Taking $\FdsTestAppr=\bs{1}$ in the local equations \eqnref{Fds_one_hp_local}, we obtain
\beq\small
\begin{aligned} 
    \LRp{{\FdsG}\FdsVarAppr,\bs{1}}_{\elem} 
    + \LRa{{\FdsABnd}{\FdsVarAppr}+{\FdsAbsA}\LRp{{\FdsVarAppr}-\FdsTrcVarAppr},\bs{1}}_{\elemBnd}
    = \LRp{{\bs{\forcing}},\bs{1}}_{\elem},
    \quad\forall\elem\in\domPart,
\end{aligned} 
\eeq
and thus
\beq\small
\begin{aligned} 
    \LRp{{\FdsG}\FdsVarAppr,\bs{1}}_{\elem} 
    + \sum_{\face\in\elemBnd}\LRa{{\FdsABnd}{\FdsVarAppr}+{\FdsAbsA}\LRp{{\FdsVarAppr}-\FdsTrcVarAppr},\bs{1}}_{\face}
    = \LRp{{\bs{\forcing}},\bs{1}}_{\elem},
    \quad\forall\elem\in\domPart,
\end{aligned} 
\eeq
which indicates the scheme is locally conservative. In particular, the amount of flux entering an element $\elem$ is equal to the amount of flux leaving the element if both the reaction term and forcing term vanish (i.e. $\FdsG={0}$ and $\bs{\forcing}=\bs{0}$).
\end{proof}

As we will show later, the locally conservative property can also be easily proven for Friedrichs' system with two-field structure. A similar result is presented in \cite{Egger2013} as well for an $hp$-HDG method used to solve the problem of Stokes flow.

\begin{lemma}[Global conservation]\lemlab{Fds_one_global_consrv}
The $hp$-HDG scheme in \eqnref{Fds_one_hp_hdg} is globally conservative.
\end{lemma}
\begin{proof}
Taking $(\FdsTestAppr,\FdsTrcTestAppr)=(\bs{1},\bs{1})$ in  Eq. \eqnref{Fds_one_hp_hdg}, we obtain
\begin{subequations}\small
\begin{align}
\begin{split}\eqnlab{global_conv_pf_eq1}
    \LRp{{\FdsG}\FdsVarAppr,\bs{1}}_{\domPart} 
    + \LRa{\jump{{\FdsABnd}{\FdsVarAppr}+{\FdsAbsA}\LRp{{\FdsVarAppr}-\FdsTrcVarAppr}},\bs{1}}_{\skel} 
    = \LRp{{\bs{\forcing}},\bs{1}}_{\domPart},&
\end{split}\\
\begin{split}\eqnlab{global_conv_pf_eq2}
    \LRa{\jump{{\FdsABnd}{\FdsVarAppr}+{\FdsAbsA}\LRp{{\FdsVarAppr}-\FdsTrcVarAppr}},\bs{1}}_{\skel}
    = -\LRa{ \half\LRp{{\FdsABnd}-{\FdsMBnd}}\bs{g}, \bs{1}}_{\skelBnd} 
    + \LRa{\half\LRp{{\FdsABnd}+{\FdsMBnd}}\FdsTrcVarAppr, \bs{1}}_{\skelBnd}.&
\end{split}
\end{align}
\end{subequations}
Substitute Eq. \eqnref{global_conv_pf_eq2} into Eq. \eqnref{global_conv_pf_eq1}, we arrive at
\beq
    \LRp{{\FdsG}\FdsVarAppr,\bs{1}}_{\domPart} 
    -\LRa{ \half\LRp{{\FdsABnd}-{\FdsMBnd}}\bs{g}, \bs{1}}_{\skelBnd} 
    + \LRa{\half\LRp{{\FdsABnd}+{\FdsMBnd}}\FdsTrcVarAppr, \bs{1}}_{\skelBnd}
    = \LRp{{\bs{\forcing}},\bs{1}}_{\domPart},   
\eeq
which implies the scheme is globally conservative.
\end{proof}
\begin{lemma}[Well-posedness of the local equation]\lemlab{Fds_one_hp_localCon} Suppose that the assumptions \itmref{A1}-\itmref{A4} hold, the local solver \eqnref{Fds_one_hp_hdg_local} is well-posed.
\end{lemma}
 
\begin{proof}
Thanks to the $hp$-nonconforming treatment, the proof is the same as the proof of \cite[Lemma 6.1]{Tan2015}, and hence omitted.
\end{proof}

\begin{theorem}[Well-posedness of the $hp$-HDG formulation]\theolab{Fds_one_hp_globalCon} Suppose that
\ben
\item the assumptions \itmref{A1}-\itmref{A4} and  \eqnref{Abstrct_M1} hold. 
\item $\Null{\FdsABnd}=\LRc{\bs{0}}$\footnote{This condition actually implies the condition \eqnref{Abstrct_M2} which is a key to make the exact solution unique (see also Remark 6.3 in \cite{Tan2015}).}.
\een
 There exists a unique solution $\FdsTrcVarAppr\in\FdsTrcVarApprSpc{}$ for the $hp$-HDG method defined by \eqnref{Fds_one_hp_hdg}.
\end{theorem}
\begin{proof}
Following the discussion presented in theorem 6.2 in \cite{Tan2015}, we first take $(\FdsTestAppr,\FdsTrcTestAppr)=(\FdsVarAppr,\FdsTrcVarAppr)$ and assume $\forcing=0$. We then perform integration by part to \eqnref{Fds_one_hp_hdg_local}, subtract \eqnref{Fds_one_hp_hdg_global} from the resulting equations, and substitute the key equality
\beq\small
\begin{aligned}
    &\LRa{{\FdsABnd}\LRp{\FdsTrcVarAppr-{\FdsVarAppr}},\FdsTrcVarAppr}_{\mort}
    =\half\LRs{
    \LRa{{\FdsABnd}\FdsTrcVarAppr,\FdsTrcVarAppr}_{\mort}
    +\LRa{{\FdsABnd}\LRp{{\FdsVarAppr}-\FdsTrcVarAppr},\LRp{{\FdsVarAppr}-\FdsTrcVarAppr}}_{\mort}
    -\LRa{{\FdsABnd}{\FdsVarAppr},{\FdsVarAppr}}_{\mort}
    },
\end{aligned}
\eeq
 which is valid for $\forall\mort\in\skel$, 
the following equation can be obtained along with homogeneous Dirichlet boundary condition:
\beq\eqnlab{Fds_one_hp_globalCon1}\small
\begin{aligned}
    &\half\LRp{\LRs{{\FdsG}+{\FdsG}^T+\FdsSum{\FdsPartial}{\FdsAk}}\FdsVarAppr,\FdsVarAppr}_{\domPart}
    +\half\LRa{{\FdsMBnd}\FdsTrcVarAppr,\FdsTrcVarAppr}_{\skelBnd}\\
    &+\LRa{\LRp{\half{\FdsABnd}+{\FdsAbsA}}\LRp{{\FdsVarAppr}-\FdsTrcVarAppr},\LRp{{\FdsVarAppr}-\FdsTrcVarAppr}}_{\domPartBnd} 
    =0
\end{aligned}
\eeq
whose left-hand side is non-negative owing to the coercivity condition \itmref{A4}, semi-positiveness of boundary operator \eqnref{Abstrct_M1} and semi-positiveness of $\half{\FdsABnd}+{\FdsAbsA}\geq0$. Therefore, we conclude that $\FdsVarAppr=\bs{0}$ in $\elem$ for $\forall\elem\in\domPart$. Furthermore, the last two terms in \eqnref{Fds_one_hp_globalCon1} is equal to zero since our assumption $\Null{\FdsABnd}=\LRc{\bs{0}}$ implies that $\Null{\abs{{\FdsABnd}}}=\LRc{\bs{0}}$. We thus can conclude that $\FdsTrcVarAppr=\bs{0}$ as well.
\end{proof}
\subsubsection{Friedrichs' system with two-field structure}\seclab{Fds_two}
In this section, we derive $hp$-HDG formulation of linear PDE in Eq. \eqnref{general_PDE} that satisfies two-field Friedrichs' system where the coefficient matrix $\FdsG$ and $\FdsAk$ can be decomposed into block matrices as presented in Eq. \eqnref{Fds_block_mat}. In addition, a set of assumptions \itmref{A1}-\itmref{A6} are assumed to hold. However, the strong coercivity is not necessarily required and can be further weakened by replacing \itmref{A4} with \itmref{A4a}-\itmref{A4b}. 

Through the Galerkin approximation along with integration by part, we will again obtain Eq. \eqnref{Fds_one_hp_local_undefined}. Moreover, the two-field structure can be further exploited by taking advantage of the decomposition where we can introduce $\FdsVarAppr=\LRp{\FdsVarAuxAppr,\FdsVarSolAppr}$, $\FdsTestAppr=\LRp{\FdsTestAuxAppr,\FdsTestSolAppr}$,  $\upwindFlux=\LRp{\upwindFlux^{\FdsVarAux}, \upwindFlux^{\FdsVarSol}}$, and $\Fdsforcing=\LRp{\FdsforcingAux, \FdsforcingSol}$.
Then, with the aid of Eq.\eqnref{Fds_block_mat}, the local equation can be rewritten as: seek $\LRp{\FdsVarAuxAppr,\FdsVarSolAppr}\in\FdsVarAuxApprSpc{}\times\FdsVarSolApprSpc{}$ such that
\begin{subequations}\eqnlab{Fds_two_hp_local_undefined}\small
\begin{align}
\begin{split}\eqnlab{Fds_two_hp_local1_undefined}
    -\FdsSum\LRp{{\FdsBk}\FdsVarSolAppr,{\FdsPartial}\FdsTestAuxAppr}_{\elem}
    +\LRp{{\FdsGaa}\FdsVarAuxAppr+{\FdsGas}\FdsVarSolAppr,\FdsTestAuxAppr}_{\elem} 
    +\LRa{\upwindFlux^{\FdsVarAux}\LRp{\FdsVarAuxAppr,\FdsVarSolAppr}\normal,{\FdsTestAuxAppr}}_{\elemBnd}
    = \LRp{{\FdsforcingAux},\FdsTestAuxAppr}_{\elem},
    \quad\forall\elem\in\domPart,&
\end{split}\\
\begin{split}\eqnlab{Fds_two_hp_local2_undefined}
    -\FdsSum\LRp{\FdsBk^T\FdsVarAuxAppr+{\FdsCk}\FdsVarSolAppr,{\FdsPartial}\FdsTestSolAppr}_{\elem}
    +\LRp{{\FdsGsa}\FdsVarAuxAppr + {\FdsGss}\FdsVarSolAppr,\FdsTestSolAppr}_{\elem}
    + \LRa{\upwindFlux^{\FdsVarSol}\LRp{\FdsVarAuxAppr,\FdsVarSolAppr}\normal,{\FdsTestSolAppr}}_{\elemBnd}
    = \LRp{{\FdsforcingSol},\FdsTestSolAppr}_{\elem},
    \quad\forall\elem\in\domPart,&
\end{split}
\end{align}
\end{subequations}
for all $\LRp{\FdsTestAuxAppr,\FdsTestSolAppr}\in\FdsVarAuxApprSpc{}\times\FdsVarSolApprSpc{}$. Now the upwind flux $\upwindFluxMort\LRp{\FdsVarAuxAppr^-,\FdsVarSolAppr^-,\FdsVarAuxAppr^+,\FdsVarSolAppr^+}\normal$ for the two-field system can also be derived by solving the Riemann problem stated in Eq. \eqnref{local_Riemann}. To that end, it is required to compute the eigen-decomposition of the coefficient matrix $\FdsABnd$. The two-field structure can be exploited again by decomposing $\FdsAbsA$ as follows:
\beq
    \FdsAbsA =     
    \begin{bmatrix}
    \FdsAbsAaa & \FdsAbsAas\\
    \FdsAbsAsa & \FdsAbsAss\\
    \end{bmatrix}, 
\eeq 
where $\FdsAbsAaa$, $\FdsAbsAas$, $\FdsAbsAsa$, and $\FdsAbsAss$ are $\FdsNVarAux\times\FdsNVarAux$, $\FdsNVarAux\times\FdsNVarSol$, $\FdsNVarSol\times\FdsNVarAux$, and $\FdsNVarSol\times\FdsNVarSol$ sub-block matrices of $\FdsAbsA$. In addition, the $\FdsNVar\times d$ matrix of upwind flux $\upwindFluxMort\normal$ can also be decomposed into $\FdsNVarAux\times d$ and $\FdsNVarSol\times d$ block matrices $(\upwindFlux^{\FdsVarAux})^*\normal$ and $(\upwindFlux^{\FdsVarSol})^*\normal$, respectively. By introducing the upwind states $\FdsVarAuxAppr^*$ and $\FdsVarSolAppr^*$, the numerical flux can be expressed in one-sided form:
\beq\eqnlab{Fds_two_upwind_flux}\small    \upwindFluxMort\LRp{\FdsVarAuxAppr^-,\FdsVarSolAppr^-,\FdsVarAuxAppr^*,\FdsVarSolAppr^*}\normal
    =
    \begin{bmatrix}
(\upwindFlux^{\FdsVarAux})^*\LRp{\FdsVarAuxAppr^-,\FdsVarSolAppr^-,\FdsVarAuxAppr^*,\FdsVarSolAppr^*}\normal\\
(\upwindFlux^{\FdsVarSol})^*\LRp{\FdsVarAuxAppr,\FdsVarSolAppr,\FdsVarAuxAppr^*,\FdsVarSolAppr^*}\normal
    \end{bmatrix}
    = 
    \FdsSum\normalComp{\FdsIndx}
    \begin{bmatrix}
\FdsAkaa & \FdsBk\\
\FdsBk^T & \FdsCk
    \end{bmatrix}
    \begin{bmatrix}
    \FdsVarAuxAppr\\
    \FdsVarSolAppr
    \end{bmatrix}
    +
    \begin{bmatrix}
    \FdsAbsAaa & \FdsAbsAas\\
    \FdsAbsAsa & \FdsAbsAss\\
    \end{bmatrix}
    \begin{bmatrix}
    \FdsVarAuxAppr - \FdsVarAuxAppr^*\\
    \FdsVarSolAppr - \FdsVarSolAppr^*
    \end{bmatrix},
\eeq
where $\FdsAkaa=0$ by the assumption \itmref{A5}. At this point, we could replace the upwind states $\LRp{\FdsVarAuxAppr^*,\FdsVarSolAppr^*}$ with the trace unknowns $\LRp{\FdsTrcVarAuxAppr,\FdsTrcVarSolAppr}$ and obtain an upwind-based HDG flux. However, one of the upwind states (and hence one of the trace unknowns) can be eliminated, and it is desirable since the system becomes even cheaper to solve. Such reduction can be achieved since $\FdsVarSolAppr^*$ and $\FdsVarAuxAppr^*$ are linearly dependent. The relationship is revealed by recognizing that in the Godunov approach \cite{godunov1959} the numerical flux $\upwindFluxMort\normal$ is the physical flux $\upwindFlux\normal$ evaluated by the approximation of the states defined on the mesh skeleton $\mort\in\skel$. That is,
\beq\eqnlab{Godunov_flux_eq}  \upwindFluxMort\LRp{\FdsVarAuxAppr^-,\FdsVarSolAppr^-,\FdsVarAuxAppr^*,\FdsVarSolAppr^*}\normal
= \upwindFlux\LRp{\FdsVarAuxAppr^*,\FdsVarSolAppr^*}\normal,
\eeq
where
\beq\eqnlab{Godunov_flux}
 \upwindFlux\LRp{\FdsVarAuxAppr^*,\FdsVarSolAppr^*}\normal
 =
     \FdsSum\normalComp{\FdsIndx}
     \begin{bmatrix}
\FdsAkaa & \FdsBk\\
\FdsBk^T & \FdsCk
    \end{bmatrix}
    \begin{bmatrix}
    \FdsVarAuxAppr^*\\
    \FdsVarSolAppr^*
    \end{bmatrix}\text{ and $\FdsAkaa=0$ by the assumption \itmref{A5}.}
\eeq
By equating \eqnref{Fds_two_upwind_flux} and \eqnref{Godunov_flux}, we can remove either $\FdsVarAuxAppr^*$ or $\FdsVarSolAppr^*$. In this work, we will consider the case with $\FdsVarSolAppr^*$. For the other possibilities, one can consult with \cite{stephen2018}  where the formulation only with $\FdsVarAuxAppr^*$ is discussed and can be easily deduced. 
\begin{lemma}[The reduced upwind flux]\lemlab{flux}
    The upwind numerical flux can be expressed as a function of $\FdsVarSolAppr^*$ only:
\beq\eqnlab{Fds_two_upwind_flux_reduced}
    \upwindFluxMort\LRp{\FdsVarAuxAppr, \FdsVarSolAppr, \FdsVarSolAppr^*}\normal =
    \begin{bmatrix}
    \FdsBBnd\FdsVarSolAppr^*\\
    \FdsBtBnd\FdsVarAuxAppr + \FdsCBnd\FdsVarSolAppr + \stabPar\LRp{\FdsVarSolAppr - \FdsVarSolAppr^*}
    \end{bmatrix}
\eeq 
where $\FdsBBnd:=\FdsSum\normalComp{\FdsIndx}\FdsBk$ and $\FdsCBnd:=\FdsSum\normalComp{\FdsIndx}\FdsCk$ are with $\normal=\LRp{\normalComp{1},\dots,\normalComp{d}}^T$ being unit outward vector of $\elemBnd$ for $\forall\elem\in\domPart$. The value of $\stabPar$ depends on either of the following assumptions:
\ben[label=(F.\arabic*)]
    \item \itmlab{F1}\,
    \ben[label=(F1.\alph*)]
        \item \itmlab{F1a} There exists an {invertible matrix $\FluxAssumpA\in\real^{\FdsNVarSol,\FdsNVarSol}$} such that {$\FdsBBnd\FluxAssumpA\FdsAbsAsa=\FdsAbsAaa$}, and
        \item \itmlab{F1b} there exists {an matrix $\FluxAssumpB\in\real^{\FdsNVarSol,\FdsNVarSol}$} such that {$\FdsAbsAas=\FdsBBnd\FluxAssumpB$}, and
        \item \itmlab{F1c} {$\bigcap_{\FdsIndx=1}^d\text{Range}\LRp{\FdsBk}=\emptyset$ and $\Null{\FdsBk}=\LRc{\bs{0}}$ for $\forall\FdsIndx=1,\dots,d$}.
    \een
    \item \itmlab{F2} {$\FdsAbsAas=0$}. Note that $\FdsAbsAas=(\FdsAbsAsa)^T$ owing to the assumption \itmref{A3}.
\een
In particular, if assumption \itmref{F1} holds then:
\beq\eqnlab{stabPar_F1}
    \stabPar := -\LRp{\FluxAssumpA^T\FdsBtBnd\FdsBBnd\FluxAssumpA}^{-1}\FluxAssumpA^T\FdsBtBnd\FdsBBnd\LRp{\FluxAssumpB+\id_{\FdsNVarSol}} + \FdsAbsAss,
\eeq
where $\id_{\FdsNVarSol}$ is a $\FdsNVarSol\times\FdsNVarSol$ identity matrix. On the other hand, if assumption \itmref{F2} holds then,
\beq\eqnlab{stabPar_F2}
    \stabPar := \FdsAbsAss.
\eeq
\end{lemma}
\begin{proof}
It can be proved by simple algebraic manipulation along with the corresponding assumption. See Appendix \secref{proof_reduced_upwind} for detailed proof.
\end{proof}
We would like to mention that the assumption \itmref{F1} and \itmref{F2} are mutually exclusive from each other. For example, the derivation of HDG numerical flux for a convection-diffusion equation can only rely on \itmref{F1} while for an elliptic equation it can only rely on \itmref{F2} (See Section \secref{numerics} for more detail).   
According to the lemma \lemref{flux}, we can construct the upwind-based HDG flux by replacing  $\FdsVarSolAppr^*$ with $\FdsTrcVarSolAppr$:
\beq\eqnlab{Fds_two_hdg_flux}    \upwindFluxTrc\LRp{\FdsVarAuxAppr,\FdsVarSolAppr,\FdsTrcVarSolAppr}\normal :=
    \begin{bmatrix}
    \FdsBBnd\FdsTrcVarSolAppr\\
    \FdsBtBnd\FdsVarAuxAppr + \FdsCBnd\FdsVarSolAppr + \stabPar\LRp{\FdsVarSolAppr - \FdsTrcVarSolAppr}
    \end{bmatrix}, 
\eeq
where the definitions of $\FdsBBnd$ and $\FdsCBnd$ follow ones introduced in Lemma \lemref{flux} (these notations will be used in the rest of the paper as well). In fact, the numerical flux \eqnref{Fds_two_hdg_flux} can represent the larger class of HDG family other than just upwind-based HDG. It is possible to obtain different HDG schemes by setting the stability matrix $\stabPar$ to be different from \eqnref{stabPar_F1} and \eqnref{stabPar_F2}. Such an exploration is also studied in \cite{Tan2015}.

Now the local equation of an $hp$-HDG scheme for the two-field Friedrichs system can be constructed by substituting the upwind-based HDG flux \eqnref{Fds_two_hdg_flux} back into \eqnref{Fds_two_hp_local_undefined}: seek $\LRp{\FdsVarAuxAppr,\FdsVarSolAppr, \FdsTrcVarSolAppr}\in\FdsVarAuxApprSpc{}\times\FdsVarSolApprSpc{}\times\FdsTrcVarSolApprSpc{}$ such that
\begin{subequations}\eqnlab{Fds_two_hp_local}\small
\begin{align}
\begin{split}\eqnlab{Fds_two_hp_local1}
    -\FdsSum\LRp{{\FdsBk}\FdsVarSolAppr,{\FdsPartial}\FdsTestAuxAppr}_{\elem}
    +\LRp{{\FdsGaa}\FdsVarAuxAppr+{\FdsGas}\FdsVarSolAppr,\FdsTestAuxAppr}_{\elem} 
    +\LRa{\FdsBBnd\FdsTrcVarSolAppr,{\FdsTestAuxAppr}}_{\elemBnd}
    = \LRp{{\FdsforcingAux},\FdsTestAuxAppr}_{\elem},
    \quad\forall\elem\in\domPart,&
\end{split}\\
\begin{split}\eqnlab{Fds_two_hp_local2}
    -\FdsSum\LRp{\FdsBk^T\FdsVarAuxAppr+{\FdsCk}\FdsVarSolAppr,{\FdsPartial}\FdsTestSolAppr}_{\elem}
    +\LRp{{\FdsGsa}\FdsVarAuxAppr + {\FdsGss}\FdsVarSolAppr,\FdsTestSolAppr}_{\elem}&\\
    + \LRa{\FdsBtBnd\FdsVarAuxAppr + \FdsCBnd\FdsVarSolAppr + \stabPar\LRp{\FdsVarSolAppr - \FdsTrcVarSolAppr},{\FdsTestSolAppr}}_{\elemBnd}
    = \LRp{{\FdsforcingSol},\FdsTestSolAppr}_{\elem},
    \quad\forall\elem\in\domPart,&
\end{split}
\end{align}
\end{subequations}
for all $\LRp{\FdsTestAuxAppr,\FdsTestSolAppr}\in\FdsVarAuxApprSpc{}\times\FdsVarSolApprSpc{}$. Again, we close the system with the conservative condition. Since the first component in $\upwindFluxTrc\LRp{\FdsVarAuxAppr,\FdsVarSolAppr,\FdsTrcVarSolAppr}\normal$ is already uniquely defined, we weakly enforce the continuity in the second component. Specifically, for $\LRp{\FdsVarAuxAppr,\FdsVarSolAppr, \FdsTrcVarSolAppr}\in\FdsVarAuxApprSpc{}\times\FdsVarSolApprSpc{}\times\FdsTrcVarSolApprSpc{}$, 
\begin{equation}\eqnlab{Fds_two_hp_global}\small
\begin{aligned}
    \LRa{\jump{{\FdsBBnd}^{T}{\FdsVarAuxAppr} + {\FdsCBnd}{\FdsVarSolAppr} +{\stabPar}\LRp{{\FdsVarSolAppr}-\FdsTrcVarSolAppr}},\FdsTrcTestSolAppr}_{\mort}
    = \bs{0},\quad\forall\mort\in\skelInt,
\end{aligned}
\end{equation}
is enforced for $\forall\FdsTrcTestSolAppr\in\FdsTrcVarSolApprSpc{}$. Finally, the boundary conditions are specified in a similar way as in Eq. \eqnref{Fds_one_hdg_bc1}. The difference is that instead of taking $\FdsMBnd:=\stabPar$, we make use of the characteristic of the two-field structure and choose $\FdsMBnd$ in a special way
\beq\eqnlab{FdsMBnd}
    \FdsMBnd:=
    \begin{bmatrix}
    0 & -\alpha\FdsBBnd\\
    \alpha\FdsBtBnd & \FdsMss
    \end{bmatrix},
\eeq
where $\FdsMss\in\real^{\FdsNVarSol,\FdsNVarSol}$ is non-negative and $\alpha\in\LRc{-1,+1}$. With this specific setting, the boundary can then be enforced through Eq. \eqnref{Fds_one_hdg_bc1} with the boundary data $\bs{\BCVal}:\domBnd\to\real^{\FdsNVarAux\times\FdsNVarSol},\,\bs{\BCVal}=\LRp{\bs{\BCVal}^{\FdsVarAux},\bs{\BCVal}^{\FdsVarSol}}$ in which $\bs{\BCVal}^{\FdsVarAux}:\domBnd\to\real^{\FdsNVarAux}$ and $\bs{\BCVal}^{\FdsVarSol}:\domBnd\to\real^{\FdsNVarSol}$. Again, for clarity and for the numerical results, we use nonhomogeneous boundary conditions, but in the well-posedness analysis
it is sufficient to consider homogeneous boundary conditions. 
We further set $\FdsMss=2{\varrho}\id_{\FdsNVarSol}+{\FdsCBnd}$, where $\id_{\FdsNVarSol}$ is a $\FdsNVarSol\times\FdsNVarSol$ identity matrix and $\varrho$ is chosen on the case-by-case basis (see Section \secref{numerics} ). Again, the boundary operator $\FdsMBnd$ is not unique but must satisfy assumptions \eqnref{Abstrct_M1} and \eqnref{Abstrct_M2}.
In particular, \eqnref{Abstrct_M1} requires:
\beq\eqnlab{Fds_two_hdg_bc_constraints}
    2\varrho\id_{\FdsNVarSol}+\FdsCBnd \geq 0.
\eeq
For Dirichlet type of boundary we set $\alpha=1$ and $\varrho=\half$, the boundary condition \eqnref{Fds_one_hdg_bc1} now is restated as\footnote{In fact, with the setting $\alpha=1$ and $\varrho=\half$ we should obtain two equations from \eqnref{Fds_one_hdg_bc1}:
\begin{align*}
    & \LRa{{\FdsBBnd}\FdsTrcVarSolAppr,\FdsTrcTestAuxAppr}_{\mort} = 0,   \quad\forall\FdsTrcTestAuxAppr\in\FdsTrcVarAuxApprSpc{},\,\forall\mort\in\skelBnd\cap\domBndDir,\quad\text{and}
    &&\LRa{\FdsTrcVarSolAppr,\FdsTrcTestSolAppr}_{\mort} = 0,   \quad\forall\FdsTrcTestSolAppr\in\FdsTrcVarSolApprSpc{},\,\forall\mort\in\skelBnd\cap\domBndDir.
\end{align*}
These two equations are supposed to be consistent with each other and hence are equivalent. However, only the latter formulation is considered in this paper since:
\ben
\item Given that only the trace unknown $\FdsTrcVarSolAppr$ is introduced, no test function in $\FdsTrcVarAuxApprSpc{}$ should be involved, and
\item The latter option is more economical.
\een
}:
\beq\eqnlab{Fds_two_hp_hdg_bc_Dir}
    \LRa{\FdsTrcVarSolAppr,\FdsTrcTestSolAppr}_{\mort}
    = 0,
    \quad\forall\FdsTrcTestSolAppr\in\FdsTrcVarSolApprSpc{},\,\forall\mort\in\skelBnd\cap\domBndDir.
\eeq
On the other hand, for Neumann or Robin type of boundary condition, we set $\alpha=-1$ and the variable $\varrho$ depends on problems to be solved. In this case, \eqnref{Fds_one_hdg_bc1} now becomes:
\beq\eqnlab{Fds_two_hp_hdg_bc_Nmn_and_Rb}\small
\begin{aligned}
    \LRa{\FdsBBnd^{T}{\FdsVarAuxAppr} + {\FdsCBnd}{\FdsVarSolAppr} +{\stabPar}\LRp{{\FdsVarSolAppr}-\FdsTrcVarSolAppr},\FdsTrcTestSolAppr}_{\mort} 
    = \LRa{\LRp{{\varrho}\id_{\FdsNVarSol}+{\FdsCBnd}}\FdsTrcVarSolAppr,\FdsTrcTestSolAppr}_{\mort},\quad
    \forall\FdsTrcTestSolAppr\in\FdsTrcVarSolApprSpc{},\,\forall\mort\in\skelBnd\cap(\domBndNmn\cup\domBndRb).&
\end{aligned}
\eeq
Again, the stabilization parameter ${\stabPar}$ is set to be \eqnref{stabPar_F1} if the assumption \itmref{F1} holds or to be \eqnref{stabPar_F2} if the assumption \itmref{F2} holds. 
Finally, the consistency condition like \eqnref{Fds_one_hdg_bc2} is not needed here since the equation \eqnref{Fds_one_hdg_bc1} itself along with the set-up \eqnref{FdsMBnd} is sufficient to determine the trace unknown $\FdsTrcVarSolAppr$ on the boundary. Combining \eqnref{Fds_two_hp_local}, \eqnref{Fds_two_hp_global}, \eqnref{Fds_two_hp_hdg_bc_Dir}, and \eqnref{Fds_two_hp_hdg_bc_Nmn_and_Rb}, we can obtain the complete the $hp$-HDG formulation for the two-field Friedrichs' system: seek $\LRp{\FdsVarAuxAppr,\FdsVarSolAppr,\FdsTrcVarSolAppr}\in\FdsVarAuxApprSpc{}\times\FdsVarSolApprSpc{}\times\FdsTrcVarSolApprSpc{}$ such that
\begin{subequations}\eqnlab{Fds_two_hp_hdg}\small
\begin{align}
\begin{split}\eqnlab{Fds_two_hp_hdg_local1}
    -\FdsSum\LRp{{\FdsBk}\FdsVarSolAppr,{\FdsPartial}\FdsTestAuxAppr}_{\domPart}
    +\LRp{{\FdsGaa}\FdsVarAuxAppr+{\FdsGas}\FdsVarSolAppr,\FdsTestAuxAppr}_{\domPart} 
    +\LRa{{\FdsBBnd}\FdsTrcVarSolAppr,{\FdsTestAuxAppr}}_{\domPartBnd}
    = \LRp{{\bs{\FdsforcingAux}},\FdsTestAuxAppr}_{\domPart},&
\end{split}\\
\begin{split}\eqnlab{Fds_two_hp_hdg_local2}
    -\FdsSum\LRp{\FdsBk^T\FdsVarAuxAppr+{\FdsCk}\FdsVarSolAppr,{\FdsPartial}\FdsTestSolAppr}_{\domPart}
    +\LRp{{\FdsGsa}\FdsVarAuxAppr + {\FdsGss}\FdsVarSolAppr,\FdsTestSolAppr}_{\domPart}
    + \LRa{\FdsBBnd^{T}{\FdsVarAuxAppr} + {\FdsCBnd}{\FdsVarSolAppr} +{\stabPar}\LRp{{\FdsVarSolAppr}-\FdsTrcVarSolAppr},{\FdsTestSolAppr}}_{\domPartBnd}&\\
    = \LRp{{\bs{\FdsforcingSol}},\FdsTestSolAppr}_{\domPart},&
\end{split}\\
\begin{split}\eqnlab{Fds_two_hp_hdg_global}
    \LRa{\jump{\FdsBBnd^{T}{\FdsVarAuxAppr} + {\FdsCBnd}{\FdsVarSolAppr} +{\stabPar}\LRp{{\FdsVarSolAppr}-\FdsTrcVarSolAppr}},\FdsTrcTestSolAppr}_{\skel\backslash\domBndDir}
    = \LRa{\LRp{{\varrho}\id_{\FdsNVarSol}+{\FdsCBnd}}\FdsTrcVarSolAppr,\FdsTrcTestSolAppr}_{\skel\cap(\domBndNmn\cup\domBndRb)},&
\end{split}\\
\begin{split}\eqnlab{Fds_two_hp_hdg_Dirichlet}
    \LRa{\FdsTrcVarSolAppr,\FdsTrcTestSolAppr}_{\skelBnd\cap\domBndDir}
    = 0,&
\end{split}
\end{align}
\end{subequations}
for all $\LRp{\FdsTestAuxAppr,\FdsTestSolAppr,\FdsTrcTestSolAppr}\in\FdsVarAuxApprSpc{}\times\FdsVarSolApprSpc{}\times\FdsTrcVarSolApprSpc{}$. We now show that the numerical scheme in \eqnref{Fds_two_hp_hdg} is both locally and globally conservative, and well-posed. For the well-posedness proof, both full and partial coercivity will be discussed. It turns out that a few extra assumptions are required to complete the well-posedness proof and they are different for full and partial coercivity cases. 
\begin{lemma}[Local conservation]\lemlab{Fds_two_local_consrv}
The $hp$-HDG scheme in \eqnref{Fds_two_hp_hdg} is both locally and globally conservative.
\end{lemma}
\begin{proof}
    The proofs are the same as the proof of lemma \lemref{Fds_one_local_consrv} and lemma \lemref{Fds_one_global_consrv},  and hence omitted.
\end{proof}
\begin{lemma}[Well-posedness of the local equation-with full coercivity]\lemlab{Fds_two_full_hp_localCon} The local solver composed by \eqnref{Fds_two_hp_hdg_local1} and \eqnref{Fds_two_hp_hdg_local2} is well-posed provided that:
\ben
    \item the assumptions \itmref{A1}-\itmref{A6} hold, and
    \item $\half\FdsCBnd + \stabPar\geq0$, and
    \item $\FdsBk$ is a constant and is non-zero for $\FdsIndx=1,\dots,d$.
\een
\end{lemma}
\begin{proof}
Since the formulation is linear and $\FdsVarAppr=(\FdsVarAuxAppr,\FdsVarSolAppr)$ is in finite dimensional space $\FdsVarApprSpc{}$, it is sufficient to restrict to a single element $\elem$ and show that the solution $(\FdsVarAuxAppr,\FdsVarSolAppr)=0$ is a unique solution in $\elem$ for $\elem\in\domPart$ provided that $\FdsTrcVarSolAppr$ and $\Fdsforcing$ are zero. Let $\FdsTrcVarSolAppr$ and $\Fdsforcing$ be zero and $(\FdsTestAuxAppr,\FdsTestSolAppr)$ be $(\FdsVarAuxAppr,\FdsVarSolAppr)$ in \eqnref{Fds_two_hp_hdg_local1} and \eqnref{Fds_two_hp_hdg_local2}. Adding the equations yields
\begin{equation}\eqnlab{Fds_two_full_hp_hdg_localproof0}\small
\begin{aligned}
    -\FdsSum\LRp{{\FdsBk}\FdsVarSolAppr,{\FdsPartial}\FdsVarAuxAppr}_{\elem}
    +\LRp{{\FdsGaa}\FdsVarAuxAppr+{\FdsGas}\FdsVarSolAppr,\FdsVarAuxAppr}_{\elem}&\\
    -\FdsSum\LRp{\FdsBk^T\FdsVarAuxAppr+{\FdsCk}\FdsVarSolAppr,{\FdsPartial}\FdsVarSolAppr}_{\elem}
    +\LRp{{\FdsGsa}\FdsVarAuxAppr + {\FdsGss}\FdsVarSolAppr,\FdsVarSolAppr}_{\elem}
    + \LRa{\FdsBBnd^{T}{\FdsVarAuxAppr} + {\FdsCBnd}{\FdsVarSolAppr} +{\stabPar}{\FdsVarSolAppr},{\FdsVarSolAppr}}_{\elemBnd} 
    = 0.&
\end{aligned}
\end{equation}
By invoking the assumption that $\FdsBk$ is a constant for $\FdsIndx=1,\dots,d$, the term $\half\FdsSum\LRp{\LRp{\FdsPartial\FdsBk^T}\FdsVarAuxAppr,\FdsVarSolAppr}_{\elem}$ contribute nothing and can be freely added into \eqnref{Fds_two_full_hp_hdg_localproof0}. It gives
\begin{equation}\eqnlab{Fds_two_full_hp_hdg_localproof1}\small
\begin{aligned}
    -\FdsSum\LRp{{\FdsBk}\FdsVarSolAppr,{\FdsPartial}\FdsVarAuxAppr}_{\elem}
    +\LRp{{\FdsGaa}\FdsVarAuxAppr+{\FdsGas}\FdsVarSolAppr,\FdsVarAuxAppr}_{\elem}
    -\FdsSum\LRp{\FdsBk^T\FdsVarAuxAppr+{\FdsCk}\FdsVarSolAppr,{\FdsPartial}\FdsVarSolAppr}_{\elem}&\\
    +\half\FdsSum\LRp{\LRp{\FdsPartial\FdsBk^T}\FdsVarAuxAppr,\FdsVarSolAppr}_{\elem}
    +\LRp{{\FdsGsa}\FdsVarAuxAppr + {\FdsGss}\FdsVarSolAppr,\FdsVarSolAppr}_{\elem}
    + \LRa{\FdsBBnd^{T}{\FdsVarAuxAppr} + {\FdsCBnd}{\FdsVarSolAppr} +{\stabPar}{\FdsVarSolAppr},{\FdsVarSolAppr}}_{\elemBnd} 
    = 0.&
\end{aligned}
\end{equation}
The first term in \eqnref{Fds_two_full_hp_hdg_localproof1} can be further expanded as:
\begin{equation}\eqnlab{IBP1_Ref}\small
\begin{aligned}
    -\FdsSum\LRp{{\FdsBk}\FdsVarSolAppr,{\FdsPartial}\FdsVarAuxAppr}_{\elem} 
    &= \FdsSum\LRp{{\FdsBk}{\FdsPartial}\FdsVarSolAppr,\FdsVarAuxAppr}_{\elem}
    + \half\FdsSum\LRp{\LRp{{\FdsPartial}{\FdsBk}}\FdsVarSolAppr,\FdsVarAuxAppr}_{\elem}
    -\LRa{{\FdsBBnd}{\FdsVarSolAppr},{\FdsVarAuxAppr}}_{\elemBnd}, 
\end{aligned}
\end{equation}
which is the result of integration by part and of the product rule. The second term on the right-hand side of \eqnref{IBP1_Ref} is zero owing to our assumption and hence can be multiplied by an arbitrary constant. Similarly, it is easy to show that the following identity holds:
\begin{equation}\eqnlab{IBP2_Ref}\small
    -\FdsSum\LRp{{\FdsCk}\FdsVarSolAppr,{\FdsPartial}\FdsVarSolAppr}_{\elem}
    = \half\FdsSum\LRp{({\FdsPartial}{\FdsCk})\FdsVarSolAppr,\FdsVarSolAppr}_{\elem} - \half\LRa{{\FdsCBnd}{\FdsVarSolAppr},{\FdsVarSolAppr}}_{\elemBnd},
\end{equation}
which is also the result of integration by part but with the assumption about symmetry \itmref{A3}. Substituting \eqnref{IBP1_Ref} and \eqnref{IBP2_Ref} back into \eqnref{Fds_two_full_hp_hdg_localproof1}, and combining (undo the decomposition) the volume integrals, we arrive at 
\begin{equation}\eqnlab{Fds_two_full_hp_hdg_localproof2}
\begin{aligned}
    \half\LRp{\LRp{{\FdsG} + \FdsG^{T} + \sum_{\FdsIndx=1}^d{\FdsPartial}{\FdsAk}}\FdsVarAppr,\FdsVarAppr}_{\elem}
    + \LRa{(\half{\FdsCBnd} +{\stabPar}){\FdsVarSolAppr},{\FdsVarSolAppr}}_{\elemBnd} 
    = 0.
\end{aligned}
\end{equation}
With the assumption of full-coercivity \itmref{A4} and the assumption of semi-positiveness $\half{\FdsCBnd} + {\stabPar}\geq0$, we can conclude that $\FdsVarAppr=\LRp{\FdsVarAuxAppr,\FdsVarSolAppr}=\bs{0}$ in $\elem$ for any $\elem\in\domPart$.
\end{proof}

\begin{theorem}[Well-posedness of the $hp$-HDG formulation-with full coercivity]\theolab{Fds_two_full_hp_globalCon} Suppose: 
\ben
    \item the assumptions \itmref{A1}-\itmref{A6} and \eqnref{Abstrct_M1} hold, and
    \item $\half\FdsCBnd + \stabPar\geq0$, and
    \item $\FdsBk$ is constant and is nonzero for $\FdsIndx=1,\dots,d$, and
    \item $\bigcap_{\FdsIndx=1}^d\text{Range}\LRp{\FdsBk}=\emptyset$ and $\Null{\FdsBk}=\LRc{\bs{0}}$ for $\forall\FdsIndx=1,\dots,d$. 
\een
Then, there exists a unique solution $\FdsTrcVarSolAppr\in\FdsTrcVarSolApprSpc{}$ for the $hp$-HDG formulation defined in \eqnref{Fds_two_hp_hdg}.
\end{theorem}
\begin{proof}
Due to the finite dimensional nature and the linearity of the global system, it is sufficient to show that the solution $\FdsTrcVarSolAppr=0$ is the unique solution if $\Fdsforcing=0$ along with homogeneous boundary data $\bs{\BCVal}=\bs{0}$. We first let $\Fdsforcing=\bs{0}$ and $(\FdsTestAuxAppr,\FdsTestSolAppr,\FdsTrcTestSolAppr)=(\FdsVarAuxAppr,\FdsVarSolAppr,\FdsTrcVarSolAppr)$. The boundary condition \eqnref{Fds_two_hp_hdg_Dirichlet} now reads
\beq
\LRa{\FdsTrcVarSolAppr,\FdsTrcVarSolAppr}_{\skelBnd\cap\domBndDir}= 0, 
\eeq
which implies that $\FdsTrcVarSolAppr=\bs{0}$ at $\mort$ for $\forall\mort\in\skelBnd\cap\domBndDir$. Adding \eqnref{Fds_two_hp_hdg_local1} and \eqnref{Fds_two_hp_hdg_local2} together, and then subtracting \eqnref{Fds_two_hp_hdg_global} from the resulting equation, we obtain:
\begin{equation}\eqnlab{Fds_two_full_hp_hdg_globalproof0}\small
\begin{aligned}
    &-\FdsSum\LRp{{\FdsBk}\FdsVarSolAppr,{\FdsPartial}\FdsVarAuxAppr}_{\domPart}
    +\LRp{{\FdsGaa}\FdsVarAuxAppr+{\FdsGas}\FdsVarSolAppr,\FdsVarAuxAppr}_{\domPart}
    +\LRa{{\FdsBBnd}\FdsTrcVarSolAppr,{\FdsVarAuxAppr}}_{\domPartBnd\backslash\domBndDir}\\
    &-\FdsSum\LRp{\FdsBk^T\FdsVarAuxAppr+{\FdsCk}\FdsVarSolAppr,{\FdsPartial}\FdsVarSolAppr}_{\domPart}
    +\LRp{{\FdsGsa}\FdsVarAuxAppr + {\FdsGss}\FdsVarSolAppr,\FdsVarSolAppr}_{\domPart}\\
    &+ \LRa{\FdsBBnd^T{\FdsVarAuxAppr} + {\FdsCBnd}{\FdsVarSolAppr} +{\stabPar}({\FdsVarSolAppr}-\FdsTrcVarSolAppr),{\FdsVarSolAppr}}_{\domPartBnd\backslash\domBndDir}\\ 
    &-\LRa{\FdsBBnd^T{\FdsVarAuxAppr} + {\FdsCBnd}{\FdsVarSolAppr} +{\stabPar}\LRp{{\FdsVarSolAppr}-\FdsTrcVarSolAppr},\FdsTrcVarSolAppr}_{\domPartBnd\backslash\domBndDir}\\
    &+\LRa{\LRp{{\varrho}\id_{\FdsNVarSol}+{\FdsCBnd}}\FdsTrcVarSolAppr,\FdsTrcVarSolAppr}_{\skel\cap(\domBndNmn\cup\domBndRb),\intRef}= 0,
\end{aligned}
\end{equation}
where the result of $\FdsTrcVarSolAppr=\bs{0}$ at $\mort$ for $\forall\mort\in\skelBnd\cap\domBndDir$ is already applied. We can add the additional term  $\half\FdsSum\LRp{\LRp{\FdsPartial\FdsBk^T}\FdsVarAuxAppr,\FdsVarSolAppr}_{\domPart}$ as $\FdsBk$ is constant for $\FdsIndx=1,\dots,d$ to obtain
\begin{equation}\eqnlab{Fds_two_full_hp_hdg_globalproof1}\small
\begin{aligned}
    &-\FdsSum\LRp{{\FdsBk}\FdsVarSolAppr,{\FdsPartial}\FdsVarAuxAppr}_{\domPart}
    +\LRp{{\FdsGaa}\FdsVarAuxAppr+{\FdsGas}\FdsVarSolAppr,\FdsVarAuxAppr}_{\domPart}
    +\LRa{{\FdsBBnd}\FdsTrcVarSolAppr,{\FdsVarAuxAppr}}_{\domPartBnd\backslash\domBndDir}\\
    &-\FdsSum\LRp{\FdsBk^T\FdsVarAuxAppr+{\FdsCk}\FdsVarSolAppr,{\FdsPartial}\FdsVarSolAppr}_{\domPart}
    +\LRp{{\FdsGsa}\FdsVarAuxAppr + {\FdsGss}\FdsVarSolAppr,\FdsVarSolAppr}_{\domPart}
    +\half\FdsSum\LRp{\LRp{\FdsPartial\FdsBk^T}\FdsVarAuxAppr,\FdsVarSolAppr}_{\domPart}\\
    &+ \LRa{\FdsBBnd^T{\FdsVarAuxAppr} + {\FdsCBnd}{\FdsVarSolAppr} +{\stabPar}({\FdsVarSolAppr}-\FdsTrcVarSolAppr),{\FdsVarSolAppr}}_{\domPartBnd\backslash\domBndDir}\\ 
    &-\LRa{\FdsBBnd^T{\FdsVarAuxAppr} + {\FdsCBnd}{\FdsVarSolAppr} +{\stabPar}\LRp{{\FdsVarSolAppr}-\FdsTrcVarSolAppr},\FdsTrcVarSolAppr}_{\domPartBnd\backslash\domBndDir}\\
    &+\LRa{\LRp{{\varrho}\id_{\FdsNVarSol}+{\FdsCBnd}}\FdsTrcVarSolAppr,\FdsTrcVarSolAppr}_{\skel\cap(\domBndNmn\cup\domBndRb),\intRef}= 0.
\end{aligned}
\end{equation}
We have the following identity by inspection:
\beq\eqnlab{id1_Ref}\small
\begin{aligned}
    -\LRa{{\FdsCBnd}{\FdsVarSolAppr},\FdsTrcVarSolAppr}_{\domPartBnd\backslash\domBndDir}
    &=
    \half\LRa{{\FdsCBnd}\LRp{{\FdsVarSolAppr}-\FdsTrcVarSolAppr},\LRp{{\FdsVarSolAppr}-\FdsTrcVarSolAppr}}_{\domPartBnd\backslash\domBndDir}\\
    &- \half\LRa{{\FdsCBnd}{\FdsVarSolAppr},{\FdsVarSolAppr}}_{\domPartBnd\backslash\domBndDir}
    - \half\LRa{{\FdsCBnd}\FdsTrcVarSolAppr,\FdsTrcVarSolAppr}_{\domPartBnd\backslash\domBndDir}.
\end{aligned}
\eeq
Note that $\LRa{{\FdsCBnd}\FdsTrcVarSolAppr,\FdsTrcVarSolAppr}_{\domPartBnd\backslash\domBnd}=0$ since $\FdsCBnd$ is assumed to be continuous across the mesh skeleton. and the trace unknown $\FdsTrcVarSolAppr$ is uniquely defined on the mortar $\mort$ for all $\mort\in\skelInt$. As a consequence, the last term in \eqnref{id1_Ref} can be rewritten as $- \half\LRa{{\FdsCBnd}\FdsTrcVarSolAppr,\FdsTrcVarSolAppr}_{\domPartBnd\cap(\domBndNmn\cup\domBndRb),\intRef}$. Substituting equality \eqnref{IBP1_Ref}, \eqnref{IBP2_Ref}, and \eqnref{id1_Ref} back into \eqnref{Fds_two_full_hp_hdg_globalproof1}, and combining (undo the decomposition) the volume integrals, we arrive at the following:
\begin{equation}\eqnlab{Fds_two_full_hp_hdg_globalproof2}\small
\begin{aligned}
    &\half\LRp{\LRp{{\FdsG} + \FdsG^{T} + \sum_{\FdsIndx=1}^d{\FdsPartial}{\FdsAk}}\FdsVarAppr,\FdsVarAppr}_{\domPart}
    + \LRa{(\half{\FdsCBnd} + {\stabPar})({\FdsVarSolAppr}-\FdsTrcVarSolAppr),({\FdsVarSolAppr}-\FdsTrcVarSolAppr)}_{\domPartBnd\backslash\domBndDir}\\ 
    &+\LRa{(\half{\FdsCBnd} +{\stabPar}){\FdsVarSolAppr},{\FdsVarSolAppr}}_{\domPartBnd\cap\domBndDir}
    +\LRa{\LRp{\half{\FdsCBnd} + {\varrho}\id_{\FdsNVarSol}}\FdsTrcVarSolAppr,\FdsTrcVarSolAppr}_{\domPartBnd\cap\LRp{\domBndNmn\bigcup\domBndRb}}
    = 0.
\end{aligned}
\end{equation}
With the assumptions of full-coercivity \itmref{A4}, of semi-postiviness of the boundary operator \eqnref{Abstrct_M1} (hence inequality \eqnref{Fds_two_hdg_bc_constraints}), and of semi-positiveness $\half{\FdsCBnd} + {\stabPar}\geq0$, we can conclude that $\FdsVarAppr=\LRp{\FdsVarAuxAppr,\FdsVarSolAppr}=\bs{0}$ in $\elem$ for all $\elem\in\domPart$. 

Now substituting $\LRp{\FdsVarAuxAppr,\FdsVarSolAppr}=\bs{0}$ back to the sub-equation \eqnref{Fds_two_hp_hdg_local1} in the local solver along with $\Fdsforcing=\bs{0}$ and $\FdsTrcVarSolAppr=\bs{0}$ at $\mort$ for $\forall\mort\in\skelBnd\cap\domBndDir$, we get:
\beq
\LRa{{\FdsBBnd}\FdsTrcVarSolAppr,{\FdsTestAuxAppr}}_{\domPartBnd\backslash\domBndDir} = 0\quad \forall\FdsTestAuxAppr\in\FdsVarAuxApprSpc{},
\eeq
which implies that ${\FdsBBnd}\FdsTrcVarSolAppr=0$. By invoking our assumption that $\bigcap_{\FdsIndx=1}^d\text{Range}\LRp{\FdsBk}=\emptyset$ and $\Null{\FdsBk}=\LRc{\bs{0}}$ for $\forall\FdsIndx=1,\dots,d$, the condition $\Null{{\FdsBBnd}}=\LRc{\bs{0}}$ can be concluded.
We then conclude $\FdsTrcVarSolAppr=0$ in $\mort$ for all $\mort=\in\skel\backslash\domBndDir$.
\end{proof}

\begin{lemma}[Well-posedness of the local equation-with partial coercivity]\lemlab{Fds_two_partial_hp_localCon} The local solver composed by \eqnref{Fds_two_hp_hdg_local1} and \eqnref{Fds_two_hp_hdg_local2}  is well-posed provided that:
\ben
    \item the assumption \itmref{A1}-\itmref{A3}, \itmref{A4a}-\itmref{A4b} and \itmref{A5}-\itmref{A6} hold, and
    \item $\half\FdsCBnd + \stabPar>0$, and 
    \item $\bigcap_{\FdsIndx=1}^d\text{Range}\LRp{\FdsBk}=\emptyset$ and $\Null{\FdsBk}=\LRc{\bs{0}}$ for $\forall\FdsIndx=1,\dots,d$. 
\een
\end{lemma}
\begin{proof}
Essentially, the $hp$-HDG formulation for the two-field Friedrichs' system with partial coercivity is the same as the one with full coercivity. Hence, we can obtain the equation \eqnref{Fds_two_full_hp_hdg_localproof2} as well following the same arguments discussed in the proof of Lemma \lemref{Fds_two_full_hp_localCon}. By applying the assumption of partial coercivity \itmref{A4a} and of positiveness of $\half{\FdsCBnd} + {\stabPar}>0$, we can conclude that $\FdsVarAuxAppr=\bs{0}$ on $\elem$ for any $\elem\in\domPart$ and ${\FdsVarSolAppr}=\bs{0}$ on $\face\in\elemBnd$ for all $\elem\in\domPart$. By applying integration by part to the first term in \eqnref{Fds_two_hp_hdg_local1}, and substituting the result that we just obtained into it along with $\FdsTrcVarSolAppr=\bs{0}$ and $\Fdsforcing=\bs{0}$, we have:
\beq
\FdsSum\LRp{{\FdsPartial}({\FdsBk}\FdsVarSolAppr),\FdsTestAuxAppr}_{\elem} = 0
\quad \forall\FdsTestAuxAppr\in\FdsVarAuxApprSpc{},
\eeq
which implies that $\FdsSum\FdsPartial(\FdsBk\FdsVarSolAppr)=\bs{0}$ in $\elem$. Furthermore, it can be rewritten as $\FdsSum\FdsBk\FdsPartial(\FdsVarSolAppr)=\bs{0}$  in $\elem$ according to the assumption \itmref{A4b}. Based on our assumption that $\bigcap_{\FdsIndx=1}^d\text{Range}\LRp{{\FdsBk}}=\emptyset$ and $\Null{{\FdsBk}}=\LRc{\bs{0}}$ for $\forall\FdsIndx=1,\dots,d$, we can further conclude that $\FdsVarSolAppr=\bs{0}$ on $\elem$ for any $\elem\in\domPart$.
\end{proof}

\begin{theorem}[Well-posedness of the $hp$-HDG formulation -with partial coercivity]\theolab{Fds_two_partial_hp_globalCon} Suppose: 
\ben
    \item the assumptions \itmref{A1}-\itmref{A3}, \itmref{A4a}-\itmref{A4b}, \itmref{A5}-\itmref{A6}, and \eqnref{Abstrct_M1} hold, and 
    \item $\half\FdsCBnd + \stabPar>0$, 
    \item $\bigcap_{\FdsIndx=1}^d\text{Range}\LRp{\FdsBk}=\emptyset$ and $\Null{\FdsBk}=\LRc{\bs{0}}$ for $\forall\FdsIndx=1,\dots,d$. 
\een
There exists a unique solution $\FdsTrcVarSolAppr\in\FdsTrcVarSolApprSpc{}$ for the $hp$-HDG formulation defined in \eqnref{Fds_two_hp_hdg}.
\end{theorem}
\begin{proof}
Given that the $hp$-HDG formulation for the two-field Friedrichs' system with partial coercivity is the same as the one with full coercivity. We can directly follow the same arguments used in the proof of Theorem \theoref{Fds_two_full_hp_globalCon} and it should lead us to the equation \eqnref{Fds_two_full_hp_hdg_globalproof2} as well. With the assumptions of partial coercivity \itmref{A4a}, of coefficient matrices \itmref{A4b}, of semi-positiveness of the boundary operator \eqnref{Abstrct_M1} (hence inequality \eqnref{Fds_two_hdg_bc_constraints}), and of positiveness of $\half{\FdsCBnd} + {\stabPar}>0$, the following results can then be derived: $\FdsVarAuxAppr=\bs{0}$ in $\elem$ for all $\elem\in\domPart$, and ${\FdsVarSolAppr}=\FdsTrcVarSolAppr$ in $\face\cap\mort$ for all $\face\in\elemBnd$ for $\forall\elemBnd\in\domPartBnd\backslash\domBndDir$ and for all $\mort\in\skel\backslash\domBndDir$. Now we perform integration by part to the first term in \eqnref{Fds_two_hp_hdg_local1}, transfer all integration over $\elemBnd$ to the summation of the integration over $\mort$ where $\mort\in\skel$, and apply the conclusion we just obtained into the resultant equation along with $\Fdsforcing=\bs{0}$, we get:
\beq
\FdsSum\LRp{{\FdsPartial}({\FdsBk}\FdsVarSolAppr),\FdsTestAuxAppr}_{\domPart} = 0\quad \forall\FdsTestAuxAppr\in\FdsVarAuxApprSpc{},
\eeq
which implies that $\FdsSum{\FdsPartial}({\FdsBk}\FdsVarSolAppr)=\bs{0}$ in $\elem$ for all $\elem\in\domPart$. Furthermore, it can be rewritten as $\FdsSum\FdsBk\FdsPartial(\FdsVarSolAppr)=\bs{0}$ according to the assumption \itmref{A4b}. Given that we assume $\bigcap_{\FdsIndx=1}^d$ $\text{Range}\LRp{{\FdsBk}}=\emptyset$ and $\Null{{\FdsBk}}=\LRc{\bs{0}}$ for $\forall\FdsIndx=1,\dots,d$. We can conclude that $\FdsVarSolAppr=\bs{0}$ in $\elem$ for all $\elem\in\domPart$ but ${\FdsVarSolAppr}=\FdsTrcVarSolAppr$ in $\face\cap\mort$ for all $\face\in\elemBnd$ for $\forall\elemBnd\in\domPartBnd\backslash\domBndDir$ for all $\mort\in\skel\backslash\domBndDir$. This leads to $\FdsTrcVarSolAppr=\bs{0}$ in $\skel\backslash\domBndDir$.
\end{proof}

\section{Strategy for \texorpdfstring{$hp$}{hp}-adaptation} \seclab{strategy_hp}
The formulations stated in \eqnref{Fds_one_hp_hdg} and \eqnref{Fds_two_hp_hdg} provide us an HDG scheme that can be carried out on $hp$-nonconforming meshes. As a result, we have a lot of flexibility when constructing finite element spaces. 
It is well-known that a smooth solution can be well resolved using a high degree of approximation even on a coarse mesh, whereas a solution with a sharp gradient can only be properly solved using a low degree of approximation on a fine mesh. Given that these different behaviors may occur locally, it is beneficial to use an adaptation procedure that allows us to improve the numerical results with a reasonable computational cost. This process can be achieved by refining elements locally via either dividing them into smaller ones ($h$-adaptation), or enriching their approximation spaces ($p$-adaptation). To that end, two essential ingredients are needed: a gauge to probe an error for each element and a method to define a new spatial discretization \cite{Huerta1999}. For the first ingredient, two different approaches are adopted in this work. One is to use an error indicator while the other is to use an adjoint-based error estimate. For the second ingredient, the regularity indicator proposed in \cite{Dolejsi2013} is applied. In the following discussion, we will discuss the error indicator obtained by two different approaches and then outline the algorithm for $hp$-adpatation.

\subsection{Doleji's approach}
By denoting $\bs{q}_h$ as an approximate solution, $\abs{\face}$ as a length of a face of an element and $\bs{\DirVal}$ as Dirichlet data, a local error estimator is defined as the following in \cite{Dolejsi2013}:
\begin{equation}\eqnlab{local_error_est_doleji}\small
    \LocErrEst{h}{\text{Doleji}}:=\LRp{ \sum_{\face\in\elemBnd\backslash\domBndDir}\sum_{\mort\in\face}\frac{1}{\abs{\face}}\LRa{\jump{\bs{q}_h},\jump{\bs{q}_h}}_{\mort} 
    + \sum_{\face\in\elemBnd\cap\domBndDir}\sum_{\mort\in\face}\frac{1}{\abs{\face}}\LRa{\bs{q}_h-\bs{\DirVal},\bs{q}_h-\bs{\DirVal}}_{\mort}}^{\half}\quad\forall\elem\in\domPart,
\end{equation}
which, originally, is derived in the context of the interior-penalty DG methods. We simply use it as a local error indicator to probe errors in our work. It is inexpensive since only the computation of the jump between adjacent elements is needed. In particular, $\bs{q}_h:=\FdsVarAppr$ and $\bs{q}_h:=\FdsVarSolAppr$ are picked for the one and two-field Friedrichs system, respectively.

\subsection{Adjoint approach}
The main idea of the adjoint approach is to measure the error in the output functional of interest. Error arises when the output functional is evaluated by a less accurate solution. Based on the pioneering work \cite{eriksson1995}, the dual-weight-residual (DWR) method had been developed for error control and mesh optimization within the context of finite element methods \cite{becker2001}. In this method, an additional linear system is needed to be solved an adjoint of the original governing equations, which then gives an estimate of the error in the target functional. This estimate can be used as a criterion to drive adaptation such that the error in the target functional is reduced. Recently, this method is adopted in solving the elliptic equations by using HDG method in \cite{cockburn2022} along with $h$-adaptivity, and the evaluation of error estimate is completed with the aid of the postprocessing technique \cite{cockburn2012b}.

In our work, we use the discrete adjoint approach where DWR method is still largely followed but the primal problem considered here is already in a discrete form. The counterpart of the adjoint problem is usually referred to as the primal problem. In this paper, the primal problem refers to the $hp$-HDG formulation stated in \eqnref{Fds_one_hp_hdg} or \eqnref{Fds_two_hp_hdg}. For HDG discretization, the discrete adjoint approach had been studied and implemented in \cite{balan2013, Dahm2014, woopen2014a, woopen2015, Balan2016, fidkowski2020, may2021}. Here we derive the adjoint equations directly on the HDG formulation, and thus never need to form or transpose a large matrix. As shall be shown, it is simply integration by parts and algebraic manipulations to arrive at  discrete adjoints that are insightful. To proceed with the discussion, some additional notations are necessary. $\Output{\cdot}$ indicates some output functional $\OutFcnl:\mathcal{X}\mapsto\real$ where $\mathcal{X}$ is a suitable Hilbert space. Examples of such functional could be drag coefficient, lift coefficient, energy across the entire domain, and so on. The subscript $H$ is used to denote the approximation computed at a coarse discretization while $h$ is a finer level. We then define the operator $\Inject$ as the injection from level $H$ to level $h$ and such operation can be done by interpolation. In addition, the interpolated quantity is denoted with a subscript $H$ and a superscript $h$. For example, the quantity $\FdsVarApprInject=\Inject\FdsVarApprH$ is obtained by interpolating the approximate solution $\FdsVarApprH$ that is solved at the coarser level (i.e. a lower degree of approximation or a coarser mesh or the combination.).
We define the following lumped variables for $\LRp{\FdsVarAppr,\FdsTrcVarAppr}\in\FdsVarApprSpc{}\times\FdsTrcVarApprSpc{}$ and $\LRp{\FdsVarAuxAppr,\FdsVarSolAppr,\FdsTrcVarSolAppr}\in\FdsVarAuxApprSpc{}\times\FdsVarSolApprSpc{}\times\FdsTrcVarSolApprSpc{}$ that are solutions to the $hp$-HDG formulations \eqnref{Fds_one_hp_hdg} and \eqnref{Fds_two_hp_hdg} respectively:
\begin{subequations}
\begin{align}
\begin{split}
\FdsVarLump:=
\begin{cases}
    \LRp{\FdsVarAppr,\FdsTrcVarAppr} & \text{for one-field Friedrichs' system},\\
    \LRp{\FdsVarAuxAppr,\FdsVarSolAppr,\FdsTrcVarSolAppr} & \text{for two-field Friedrichs' system},
\end{cases}
\end{split}\\
\begin{split}
\FdsTestVarLump:=
\begin{cases}
    \LRp{\FdsTestAppr,\FdsTrcTestAppr} & \text{for one-field Friedrichs' system},\\
    \LRp{\FdsTestAuxAppr,\FdsTestSolAppr,\FdsTrcTestSolAppr} & \text{for two-field Friedrichs' system},
\end{cases}
\end{split}
\end{align}
\end{subequations}
where $\FdsTestVarLump$ resides in the product Hilbert space $\mathcal{Y}$ constructed from the corrresponding fields that define $\FdsTestVarLump$. The residual $\ResidualhOne{\cdot;\cdot}$ of $hp$-HDG formulation of one-field Friedrichs' system \eqnref{Fds_one_hp_hdg} is a functional $\Res^{\text{one}}_h:\mathcal{X}\times\mathcal{Y}\mapsto\real$ acting on the product of $\mathcal{X}$ and $\mathcal{Y}$ and is 
composed by two functionals $\Res^{\FdsVar}_{h,\elem}:\mathcal{X}\times\mathcal{Y}\mapsto\real$ 
and $\Res^{\FdsTrcVar}_{h,\mort}:\mathcal{X}\times\mathcal{Y}\mapsto\real$. It is then defined as: 
\begin{equation}\eqnlab{res_Fds_one}\small
    \begin{aligned}
    \ResidualhOne{\FdsVarLump;\FdsTestVarLump} 
    &= 
    \sum_{\elem\in\domPart}\Res^{\FdsVar}_{h,\elem}\LRs{\FdsVarLump;\FdsTestVarLump} 
    + \sum_{\mort\in\skel}\Res^{\FdsTrcVar}_{h,\mort}\LRs{\FdsVarLump;\FdsTestVarLump},
    \end{aligned}
\end{equation}
where
\begin{subequations}\small
\begin{align}
\begin{split}
\Res^{\FdsVar}_{h,\elem}\LRs{\FdsVarLump;\FdsTestVarLump}
:=
-\FdsSum\LRp{{\FdsAk}\FdsVarAppr,{\FdsPartial}\FdsTestAppr}_{\elem} 
+ \LRp{{\FdsG}\FdsVarAppr,{\FdsTestAppr}}_{\elem} 
+ \LRa{{\FdsABnd}{\FdsVarAppr}+{\stabPar}\LRp{{\FdsVarAppr}-\FdsTrcVarAppr},{\FdsTestAppr}}_{\elemBnd}
- \LRp{{\bs{\forcing}},\FdsTestAppr}_{\elem},
\end{split}\\
\begin{split}
\Res^{\FdsTrcVar}_{h,\mort}\LRs{\FdsVarLump;\FdsTestVarLump}
:=
\LRa{\jump{{\FdsABnd}{\FdsVarAppr}+{\stabPar}\LRp{{\FdsVarAppr}-\FdsTrcVarAppr}},\FdsTrcTestAppr}_{\mort}
+\LRa{ \half\LRp{{\FdsABnd}-{\FdsMBnd}}\bs{g}, \FdsTrcTestAppr}_{\mort\cap\skelBnd} 
- \LRa{\half\LRp{{\FdsABnd}+{\FdsMBnd}}\FdsTrcVarAppr, \FdsTrcTestAppr}_{\mort\cap\skelBnd},
\end{split}
\end{align}
\end{subequations}
for all $\FdsTestVarLump\in\FdsVarApprSpc{}\times\FdsTrcVarApprSpc{}$. On the other hand, the residual $\ResidualhTwo{\cdot;\cdot}$ of $hp$-HDG formulation for the two-field Friedrichs' system is also a functional $\Res^{\text{two}}_h:\mathcal{X}\times\mathcal{Y}\mapsto\real$ acting on the product of $\mathcal{X}$ and $\mathcal{Y}$ and is 
composed by three functionals $\Res^{\FdsAux}_{h,\elem}:\mathcal{X}\times\mathcal{Y}\mapsto\real$ 
, $\Res^{\FdsSol}_{h,\elem}:\mathcal{X}\times\mathcal{Y}\mapsto\real$ and $\Res^{\FdsTrcSol}_{h,\mort}:\mathcal{X}\times\mathcal{Y}\mapsto\real$. It is defined as:
\begin{equation}\eqnlab{res_Fds_two}\small
    \begin{aligned}
    \ResidualhTwo{\FdsVarLump;\FdsTestVarLump} 
    &= 
    \sum_{\elem\in\domPart}\Res^{\FdsAux}_{h,\elem}\LRs{\FdsVarLump;\FdsTestVarLump} 
    +\sum_{\elem\in\domPart}\Res^{\FdsSol}_{h,\elem}\LRs{\FdsVarLump;\FdsTestVarLump} 
    +\sum_{\mort\in\skel}\Res^{\FdsTrcSol}_{h,\mort}\LRs{\FdsVarLump;\FdsTestVarLump},
    \end{aligned}
\end{equation}
where
\begin{subequations}\small
\begin{align}
\begin{split}
\Res^{\FdsAux}_{h,\elem}\LRs{\FdsVarLump;\FdsTestVarLump}
:=&
-\FdsSum\LRp{{\FdsBk}\FdsVarSolAppr,{\FdsPartial}\FdsTestAuxAppr}_{\elem}
+\LRp{{\FdsGaa}\FdsVarAuxAppr+{\FdsGas}\FdsVarSolAppr,\FdsTestAuxAppr}_{\elem} 
+\LRa{{\FdsBBnd}\FdsTrcVarSolAppr,{\FdsTestAuxAppr}}_{\elemBnd}
-\LRp{{\bs{\FdsforcingAux}},\FdsTestAuxAppr}_{\elem},
\end{split}\\
\begin{split}
\Res^{\FdsSol}_{h,\elem}\LRs{\FdsVarLump;\FdsTestVarLump}
:=&
-\FdsSum\LRp{\FdsBk^T\FdsVarAuxAppr+{\FdsCk}\FdsVarSolAppr,{\FdsPartial}\FdsTestSolAppr}_{\elem}
+\LRp{{\FdsGsa}\FdsVarAuxAppr + {\FdsGss}\FdsVarSolAppr,\FdsTestSolAppr}_{\elem}\\
&+ \LRa{\FdsBBnd^{T}{\FdsVarAuxAppr} + {\FdsCBnd}{\FdsVarSolAppr} +{\stabPar}\LRp{{\FdsVarSolAppr}-\FdsTrcVarSolAppr},{\FdsTestSolAppr}}_{\elemBnd}
- \LRp{{\bs{\FdsforcingSol}},\FdsTestSolAppr}_{\elem},
\end{split}\\
\begin{split}
\Res^{\FdsTrcSol}_{h,\mort}\LRs{\FdsVarLump;\FdsTestVarLump}
:=&
\LRa{\jump{\FdsBBnd^{T}{\FdsVarAuxAppr} + {\FdsCBnd}{\FdsVarSolAppr} +{\stabPar}\LRp{{\FdsVarSolAppr}-\FdsTrcVarSolAppr}},\FdsTrcTestSolAppr}_{\mort\backslash\domBndDir}
- \LRa{\LRp{{\varrho}\id_{\FdsNVarSol}+{\FdsCBnd}}\FdsTrcVarSolAppr,\FdsTrcTestSolAppr}_{\mort\cap(\domBndNmn\cup\domBndRb)}\\
&+\LRa{{\FdsCBnd}\FdsTrcVarSolAppr,\FdsTrcTestSolAppr}_{\mort\cap\domBndDir},
\end{split}
\end{align}
\end{subequations}
for all $\FdsTestVarLump\in\FdsVarAuxApprSpc{}\times\FdsVarSolApprSpc{}\times\FdsTrcVarSolApprSpc{}$. At this point, we can further define a more general residual based on \eqnref{res_Fds_one} and \eqnref{res_Fds_two} that covers both cases:
\beq
    \Res_h\LRs{\FdsVarLump;\FdsTestVarLump}
    :=
    \begin{cases}
        \ResidualhOne{\FdsVarLump;\FdsTestVarLump}\quad\text{if only identified as a one-field Friedrichs' system},\\ 
        \ResidualhTwo{\FdsVarLump;\FdsTestVarLump}\quad\text{if identified as a two-field Friedrichs' system}. 
    \end{cases}
\eeq
Obviously, the residual is always zero if it is evaluated by using the correct solution while it is generally non-zero when using the interpolated solution. That is, $\Res_h\LRs{\FdsVarLump;\FdsTestVarLump}=0$ but  $\Res_h\LRp{\FdsVarLumpInject;\FdsTestVarLump}\neq 0$ for all $\FdsTestVarLump\in\FdsVarApprSpc{}\times\FdsTrcVarApprSpc{}$ with respect to the one-field Friedrichs' system and for all $\FdsTestVarLump\in\FdsVarAuxApprSpc{}\times\FdsVarSolApprSpc{}\times\FdsTrcVarSolApprSpc{}$ with respect to the two -field Friedrichs' system. The error of the output functional $\Output{\cdot}$ can now be approximated as 
\cite{Dahm2014,Balan2016,Woopen2014b}:
\beq\eqnlab{output_error}
    \Output{\FdsVarLumpH} - \Output{\FdsVarLump} \approx -\Res_h\LRs{\FdsVarLumpInject;\FdsTestVarLump},
\eeq
where $\FdsTestVarLump\in\FdsVarApprSpc{}\times\FdsTrcVarApprSpc{}$ and $\FdsTestVarLump\in\FdsVarAuxApprSpc{}\times\FdsVarSolApprSpc{}\times\FdsTrcVarSolApprSpc{}$ for the one-field and two-field Friedrichs' system, respectively. Here, the variable $\FdsTestVarLump$ is an adjoint variable and serves as a detection of the sensitivity of output functional error induced by a less accurate solution. Further, it has to satisfy the adjoint $hp$-HDG formulation that is either \eqnref{Fds_one_ad_hdg} or \eqnref{Fds_two_ad_hdg} and depends on which system we are looking at. The derivation of the adjoint $hp$-HDG formulation and well-posedness analysis are discussed in Appendix \secref{adjoint_hdg}.

From \eqnref{output_error}, it can be seen that two different approximation spaces are required. In this work, we construct the finer space by enriching the degree of approximation without refining the mesh. That is, the meshes used in solving the primal and adjoint $hp$-HDG formulations are the same (i.e. $\domPart=\domPartH$) while the finite element spaces used in the primal and adjoint $hp$-HDG formulations differ by one degree of approximation for each element. The computational cost is relatively cheap in this way and the enrichment can be done without increasing the complexity of the code.  Toward the adaptation, we need to localize the error approximation presented in \eqnref{output_error}. By defining the localized residual ${\Res}_{h,\elem}$ as:
\beq
    {\Res}_{h,\elem}\LRs{\FdsVarLump;\FdsTestVarLump}
    :=
    \begin{cases}
        \Res^{\FdsVar}_{h,\elem}\LRs{\FdsVarLump;\FdsTestVarLump} & \text{for one-field Friedrichs' system},\\ 
        \Res^{\FdsAux}_{h,\elem}\LRs{\FdsVarLump;\FdsTestVarLump}+\Res^{\FdsSol}_{h,\elem}\LRs{\FdsVarLump;\FdsTestVarLump} & \text{for two-field Friedrichs' system}, 
    \end{cases}
\eeq 
and following the works \cite{Dahm2014,Woopen2014b,Balan2016}, the local error indicator based on the adjoint approach can be defined as :
\begin{equation}\eqnlab{local_error_est_adjoint}
    \LocErrEst{H}{\text{adjoint}} := \abs{{\Res}_{h,\elem}\LRs{\FdsVarLumpInject;\FdsTestVarLump}},
\end{equation}
where the adjoint variable $\FdsTestVarLump$ is the solution to the adjoint $hp$-HDG formulation which is either \eqnref{Fds_one_ad_hdg} or \eqnref{Fds_two_ad_hdg}. It should be emphasized that the error indicator \eqnref{local_error_est_adjoint} does not include the contribution from the trace unknowns (i.e. $\Res^{\FdsTrcSol}_{h,\mort}$ and $\Res^{\FdsTrcSol}_{h,\mort}$ are neglected) owing to its insignificant influence \cite{Dahm2014,Woopen2014b}.

\begin{rema}
It should be pointed out that the output error stated in \eqnref{output_error} can directly be computed by evaluating the difference between $\Output{\FdsVarLumpH}$ and $\Output{\FdsVarLump}$, where we have to solve the $hp$-HDG formulation \eqnref{Fds_one_hp_hdg} or \eqnref{Fds_two_hp_hdg} at two different levels of approximation. However, in this work, we stick to the approximation given by the DWR method. This method is more general in the sense that the adjoint problem is always linear and is the only problem needed to be solved at the fine level of approximation. It holds true regardless of whether the primal problem is linear or not.
\end{rema}

\subsection{An adaptation algorithm}
The algorithm used in this work is the simplified version of the strategy proposed in \cite{Dolejsi2013}, which provides all necessary keys for carrying out $hp$-adaption.
Combining the previous discussion on error indicators, we denote a general local and global error indicator as:
\begin{subequations}\eqnlab{error_est}\small
\begin{align}
\begin{split}\eqnlab{local_error_est}
    \LocErrEstGen{H}:=
        \LocErrEst{H}{\text{Doleji}}\text{ or }
        \LocErrEst{H}{\text{adjoint}}&
\end{split},\\
\begin{split}\eqnlab{global_error_est}
    \GlbErrEstGen{H}:=\LRp{ \sum_{\elem\in\domPart}\LocErrEstGen{H}^2}^{\half}&
\end{split}.
\end{align}
\end{subequations}
To drive full $hp$-adpatation, a method to decide how to construct a new spatial discretization is also necessary. In this work, it is desirable that the spatial discretization can be constructed according to the smoothness of the solution. To achieve it, a local regularity indicator is needed and the one proposed in \cite{Dolejsi2013} is applied in this work. By denoting $\abs{\elem}$ as the area of an element, the indicator reads:
\begin{equation}\eqnlab{local_reg_est}
    \LocRegInd{\bs{q}_h}:=\frac{ \sum_{\face\in\elemBnd}\sum_{\mort\in\face\backslash\domBnd}\LRa{\jump{\bs{q}_h},\jump{\bs{q}_h}}_{\mort} }{\abs{\elem}\elemL^{2\polyElem-3}},
\end{equation}
where $\bs{q}_h:=\FdsVarAppr$ and $\bs{q}_h:=\FdsVarSolAppr$ are  one-field and two-field Friedrichs' system, respectively.
Once error and regularity indicators are computed, one or a couple of the following operations will be performed:\\
\bit
    \item $h$-refinement\footnote{We also enforce the number of hanging nodes resulting from local $h$-refinement to be always one in each interface within a mesh.}: to split a given \tit{mother element} $\elem$ into four \tit{daughter elements} $\elemChild$ by connecting centers of its edges.
    \item $p$-refinement: to increase the degree of polynomial approximation for a given element $\elem$, i.e., we set $\polyElem = \polyElem+1$.
    \item $p$-coarsening: to decrease the degree of polynomial approximation for a given element $\elem$, i.e., we set $\polyElem = \polyElem-1$.
\eit

In the original strategy presented in \cite{Dolejsi2013}, there are two additional operations called $h$-coarsening and $hp$-substitution. They merge elements that have arisen in a previous adaptation cycle along with $p$-refinement, $p$-coarsening, or nothing. However, according to our numerical experiments, this action only slightly increased efficiency, and sometimes the performance seems to be degrading. For this reason, we remove these operations from our adaptation strategy. Given the user-defined tolerance $0\leq\tol\leq1$ and the maximum cycle number, the $hp$-adaption procedure can now be performed by following the strategy outlined in Algorithm \algmref{hp_algm}.
\begin{algorithm}
\caption{An $hp$-adaptation algorithm}\algmlab{hp_algm}
\begin{algorithmic}[1]
\State $\LocErrEstGen{H}\gets 1\quad\forall\elem\in\domPartH$\text{ and }$\GlbErrEstGen{H}\gets 0$\Comment{The initialization for starting the adaptation cycle(s).}
\While{$\max_{\elem\in\domPartH}{\LocErrEstGen{H}}\geq\tol\GlbErrEstGen{H}$\text{ or }\#\text{cycle}$\leq$\#\text{max. cycle}}
\State \text{Solve the $hp$-HDG formulation stated in \eqnref{Fds_one_hp_hdg} or \eqnref{Fds_two_hp_hdg} on a coarse (current) level.}
\State \text{(Solve the adjoint $hp$-HDG formulation stated in \eqnref{Fds_one_ad_hdg} or \eqnref{Fds_two_ad_hdg}}
\State \text{on a fine (by enriching degree of approximation) level if the adjoint approach is applied).}
\State \text{Compute and update local and global error indicator presented in \eqnref{error_est}.}
    \For{$\elem\in\domPartH$}
        \If{$\LocErrEstGen{H}\geq\tol\max_{\elem\in\domPartH}{\LocErrEstGen{H}}$}
            \If{$\LocRegInd{\solApprH}\leq\elemL^{-2}$}
            \State \text{Tag the element as $p$-refinement}
            \ElsIf{$\elemL^{-2}<\LocRegInd{\solApprH}\leq\elemL^{-4}$}
            \State \text{Tag the element as $h$-refinement}
            \Else
            \State \text{Tag the element as $h$-refinement along with $p$-coarsening}
            \EndIf
        \EndIf
    \EndFor
    \State \text{Perform adaption and construct the new corresponding finite element space} 
\EndWhile
\end{algorithmic}
\end{algorithm}

It should be noted that in each $hp$-adaption cycle a global refinement will be performed if $\tol=0$ while only a single element will be refined if $\tol=1$. However, we still will use relatively small $\tol$ and it will not result in global refinement which we would like to avoid for the sake of efficiency. It could happen since all problems to be solved in this paper have solutions that have extremely sharp gradients within a portion of the domain or on the boundary. Under this circumstance, the maximum of the local error indicator is relatively large in general (i.e. $\max_{\elem\in\domPartH}{\LocErrEstGen{H}}\gg\min_{\elem\in\domPartH}{\LocErrEstGen{H}}$) and its distribution leans toward to its' minimum value $\min_{\elem\in\domPartH}{\LocErrEstGen{H}}$.

\section{Numerical experiments} \seclab{numerics}
In this section, we are going to present several numerical experiments for different kinds of PDEs. The numerical solution is obtained by solving $hp$-HDG formulations \eqnref{Fds_one_hp_hdg} or \eqnref{Fds_two_hp_hdg} along with the adaptivity strategy discussed in the section \secref{strategy_hp}. The main goal is to demonstrate the validity of the unified $hp$ formulations and examine the performance of our proposed adaptive algorithm. 
As for the performance, we measure 
the convergence rate of the error in the $\Lsp{2}$-norm versus the number of degrees of freedoms (DOFs) resulting from  the statically condensed $hp$-HDG formulations. It should be noted that the required DOFs for the adjoint approach are slightly higher than those for Doleji's method.
This is because we additionally need to solve adjoint HDG formulation presented in \eqnref{Fds_one_ad_hdg} or \eqnref{Fds_two_ad_hdg} to get the adjoint that is used to evaluate the error indicator \eqnref{local_error_est_adjoint}. 

The PDEs under consideration in the experiments can be classified as elliptic, hyperbolic, and mixed equations. In the following subsections, we will briefly discuss each type of PDEs and justify the well-posedness of its HDG formulation by using the lemmas and the theorems discussed in the section \secref{hpHDGform}.
In this section, we simply use subscript $h$ to denote the numerical solution. Readers should not confuse it with the notation used in Section \secref{strategy_hp} where $h$ and $H$ are used to refer to different levels.

\subsection{Elliptic PDEs}
For elliptic PDEs, we consider: 
\ben[label=(E.\arabic*)]
    \item \itmlab{e1}Poisson problem (isotropic diffusion) with a corner singularity,
    \item \itmlab{e2}anisotropic diffusion problem with discontinuous Dirichlet boundary condition, and
    \item \itmlab{e3}heterogeneous anisotropic diffusion problem with  discontinuous field $\diffCoeff$.
\een
We analyze these problems 
by using the two-field Friedrichs' system with partial coercivity. 
The problem reads: find a function $\sol:\dom\to\real$ such that:
\begin{equation}\eqnlab{AniDiff}
\begin{split}
    -\Div{\LRp{\diffCoeff\nabla{\sol}}} = \forcing,\quad&\text{in}\ \dom,\\
    \sol = \BCVal^{\sol},\quad&\text{on}\ \domBndDir,\\
    \diffCoeff\Grad{\sol}\cdot\normal + \rtCoeff\sol = \BCVal^{\sol},\quad&\text{on}\ \domBndNmn\cap\domBndRb\text{ where }\rtCoeff=0\text{ when on the }\domBndNmn,
\end{split}
\end{equation}
where the boundary data $\BCVal^{\sol}:\domBnd\mapsto\real$ is in $\Lsp{2}(\domBnd)$ and is defined as
\beq
\BCVal^{\sol}
=
\begin{cases}
    \DirVal\text{ on }\domBndDir,\\
    \NmnVal\text{ on }\domBndNmn,\\
    \RbVal\text{ on }\domBndRb.
\end{cases}
\eeq
Here, $\forcing\in\Lsp{2}(\dom)$ is a source term; and $\diffCoeff\in\LRs{\Lsp{\infty}(\dom)}^d$ is a diffusivity coefficient that is assumed to be  symmetric positive-definite  and its lowest eigenvalue is uniformly bounded away from zero. To be able to interpret the numerical result later, we briefly review some physical aspects of the PDE stated in \eqnref{AniDiff}. At each location within $\dom$, we have the principal direction of anisotropy denoted by $\phySkwCompX$ and the direction of weak diffusion denoted by $\phySkwCompY$. As shown in Figure \figref{skewedDomain} along with coordinate of physical domain $(\phyCompX,\phyCompY)$, it is possible to align $\phyCompX$ to $\phySkwCompX$ by rotating the system with the angle $\misalign$ so that  the equation \eqnref{AniDiff} becomes:

\begin{equation*}
    \diffCoeffCompX\ppSq{\sol}{\phySkwCompX} + \diffCoeffCompY\ppSq{\sol}{\phySkwCompY} = \forcing\quad\text{in}\ \dom, 
\end{equation*}
where $\diffCoeffCompX$ and $\diffCoeffCompY$ are referred to the diffusivity in $\phySkwCompX$-direction and in $\phySkwCompY$-direction, respectively. 
Since $\diffCoeffCompX$ represents the principal direction of anisotropy, we always have $\diffCoeffCompX \geq \diffCoeffCompY$. At this point, we can define the anisotropy ratio $\anisoRatio:=\diffCoeffCompX/\diffCoeffCompY$ which indicates the strength of the anisotropy. The case $\anisoRatio=1$ corresponds to isotropic diffusion (i.e. a Laplace or Poisson equation). 


\begin{figure}[H]
    \centering
    \includegraphics[width=60mm]{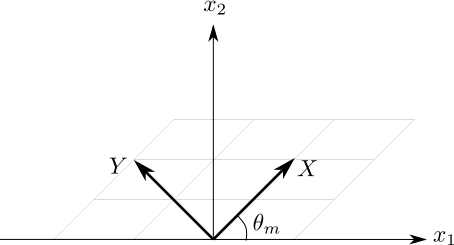}
    \caption{The skewed domain of anisotropic field with $\phySkwCompX$ parallel to the anisotropic principal direction.}
    \figlab{skewedDomain}
\end{figure}

To cast the problems into the Friedrichs' framework, we rewrite the original PDE stated in \eqnref{AniDiff} into the first order form by introducing the \textit{auxiliary variable} $\aux:=-\diffCoeff\nabla{\sol}$:
\begin{subequations}\eqnlab{AniDiffMixed}
\begin{align}
\begin{split}
    \nabla\sol + \diffCoeff^{-1}\aux=0,\quad&\text{in}\ \dom,
\end{split}\\
\begin{split}
    \nabla\cdot\aux = \forcing,\quad&\text{in}\ \dom,
\end{split}\\
\begin{split}
        \sol = \BCVal^{\sol},\quad&\text{on}\ \domBndDir,\\
        -\aux\cdot\normal + \rtCoeff\sol = \BCVal^{\sol},\quad&\text{on}\ \domBndNmn\cup\domBndRb\text{ where }\rtCoeff=0\text{ when on the }\domBndNmn.
\end{split}
\end{align}
\end{subequations}
Thus, the size of the system is determined: $\FdsNVar=d+1$, $\FdsNVarAux=d$, and $\FdsNVarSol=1$. The corresponding two-field Friedrichs' system reads:
\beq\eqnlab{Fds_block_mat_elliptic}
\FdsG =
\begin{bmatrix}
\diffCoeff^{-1} & 0_{d\times 1}\\
0_{1\times d} & 0
\end{bmatrix},\quad\quad
\FdsAk =
\begin{bmatrix}
0_{d\times d} & \normalCon_{\FdsIndx}\\
\normalCon_{\FdsIndx}^T & 0
\end{bmatrix},
\eeq
where $\normalCon_{\FdsIndx}$ stands for the $\FdsIndx$-th canonical basis in $\real^{\FdsNVarAux}$ and $0$ with subscript indicates zero matrix with its dimension. To enforce boundary conditions properly, the boundary operator $\FdsMBnd$ defined in \eqnref{FdsMBnd} can be specified as:
\beq\eqnlab{FdsMBnd_elliptic}
    \begin{cases}
    \alpha=+1,\,\FdsMss=2{\varrho}\id_{\FdsNVarSol},\text{ on }\domBndDir\text{ where }\varrho:=\half,\\
    \alpha=-1,\,\FdsMss=2{\varrho}\id_{\FdsNVarSol},\text{ on }\domBndNmn\cap\domBndRb\text{ where }\varrho:=\rtCoeff,\\
    \end{cases}
\eeq
where we require $\rtCoeff=0$ on $\domBndNmn$ and $\rtCoeff>0$ on $\domBndRb$ in order to for the conditions in \eqnref{Abstrct_M1} and \eqnref{Abstrct_M2} to hold.

\begin{lemma}
The $hp$-HDG formulation for the PDE stated in \eqnref{AniDiffMixed} is well-posed both locally and globally.
\end{lemma}
\begin{proof}
The assumptions \itmref{A1}-\itmref{A3}, \itmref{A4a}-\itmref{A4b} and \itmref{A5}-\itmref{A6} is obviously satisfied by substituting \eqnref{FdsMBnd_elliptic} into each conditions. 

On the other hand, the numerical flux falls into the category \itmref{F2} where we have:
\beq
    \FdsAbsA =     
    \begin{bmatrix}
    \frac{\normal\normal^T}{\norm{\normal}_2} & 0_{d\times1}\\
    0_{1\times d} & \norm{\normal}_2^2\\
    \end{bmatrix}
\eeq
in which $\norm{\cdot}_2$ is a standard Euclidean norm. Thus, the stabilization parameter reads $\stabPar=\norm{\normal}_2^2=1$.  It is evident that:
\bit
    \item $\half\FdsCBnd + \stabPar=\stabPar=1>0$, and
    \item $\bigcap_{\FdsIndx=1}^d\text{Range}\LRp{\normalCon_{\FdsIndx}}=\emptyset$ and $\Null{\normalCon_{\FdsIndx}}=\LRc{\bs{0}}$ for $\forall\FdsIndx=1,\dots,d$. 
\eit
Hence, by Lemma \lemref{Fds_two_partial_hp_localCon} and Theorem \theoref{Fds_two_partial_hp_globalCon} we can conclude that the $hp$-HDG formulation for the elliptic PDE \eqnref{AniDiffMixed} is well-posed locally and globally.
\end{proof}

\subsection{Hyperbolic PDE}
We consider the following hyperbolic PDE:
\ben[label=(HP.\arabic*)]
    \item \itmlab{hp1}steady-state linear advection with  variable speed and discontinuous inflow condition. 
\een
The PDE for steady-state linear advection reads: find a function $\sol:\dom\to\real$ such that:
\begin{equation}\eqnlab{LinAdv}
\begin{split}
    \Div{\LRp{\vel\sol}} = \forcing, \quad&\text{in}\ \dom,\\
    \sol = \DirVal,\quad&\text{on}\ \domBnd^-
\end{split}
\end{equation}
with $\vel\in\LRs{\Lsp{\infty}\LRp{\dom}}^{d}$, $\Div{\vel}\in\Lsp{\infty}\LRp{\dom}$, $\forcing\in\Lsp{2}\LRp{\dom}$ and $\DirVal\in\Lsp{2}\LRp{\domBnd^-}$. Here, we adopt the convention to denote the inflow boundaries as $\domBnd^- = \LRc{x \in \partial\Omega: \vel\cdot\normal<0}$ and they are essentially Dirichlet boundaries in this problem set. It is well-known that singularity (or discontinuity) can be propagated by linear advection. Hence, we can expect that there is a shock within the domain $\dom$ if a discontinuity is specified at the inflow boundary $\domBnd^-$. The problem can be analyzed by the one-field Friedrichs' system. The size of the system is $\FdsNVar=1$ and the corresponding system reads: 
\beq\eqnlab{Fds_block_mat_hyper}
\FdsG = 0
,\quad\quad
\FdsAk =
\velComp{\FdsIndx} 
\text{ for } k=1,\cdots,d.
\eeq
It is easy to show that the assumptions \itmref{A1}-\itmref{A3} are valid. To have coercivity \itmref{A4}, we further assume that
\beq\eqnlab{LinAdv_assumption}
    \infess_{\dom}\half\Div{\vel}\geq0.
\eeq
Finally, we also require the following condition to obtain a well-posed HDG scheme:
\begin{align}\eqnlab{LinAdv_flux_assumption}
    & \vel\cdot\normal\neq0\text{ on }\mort,\quad\forall\mort\in\skel,
    && \jump{\vel\cdot\normal}=0\text{ on }\elemBnd,\quad\forall\elem\in\domPart,
\end{align}
where we assume $\vel\cdot\normal$ is always continuous across element edges and does not vanish at edges (or mortars). It should be noted that the condition for continuity can be relaxed and the resulting numerical flux has the weight-average type of stabilization parameter \cite{Tan2015}. 
\begin{lemma}
The $hp$-HDG formulation for the PDE stated in \eqnref{LinAdv} is well-posed both locally and globally if the assumptions \eqnref{LinAdv_assumption} and \eqnref{LinAdv_flux_assumption} hold.
\end{lemma}
\begin{proof}
Consider the following transformation:
\begin{equation}\eqnlab{LinAd_TranMap}
    \sol=\TranMap\Tilde{\sol},\quad\text{where }\TranMap:=e^{-\gamma\LRp{\phyVar-\phyVar_0}\cdot\vel},
\end{equation}
in which $\gamma\in\real$, $\phyVar_0\in\dom$, and $\Tilde{\sol}:\dom\to\real$. Substituting \eqnref{LinAd_TranMap} back into \eqnref{LinAdv} gives us:
\begin{equation}\eqnlab{LinAdv_Tran}
\begin{split}
    \Div{\LRp{\Tilde{\vel}\Tilde{\sol}}} = \forcing, \quad&\text{in}\ \dom,\\
    \Tilde{\sol} = \Tilde{\DirVal},\quad&\text{on}\ \domBnd^-,
\end{split}
\end{equation}
where $\Tilde{\vel}=\TranMap\vel$ and $\Tilde{\DirVal}=\TranMap^{-1}\DirVal$. Note that $\TranMap>0$ and hence its inverse always exists. The PDE \eqnref{LinAdv_Tran} can also be identified as a one-field Friedrichs' system where:
\beq\eqnlab{Fds_block_mat_Tran_hyper}
\FdsG = 0
,\quad\quad
\FdsAk =
\Tilde{\vel}_{\FdsIndx}.
\eeq
It is obvious that \itmref{A1}-\itmref{A3} are valid and
\beq
    \FdsG+ \FdsG^{T} + \sum_{\FdsIndx=1}^d\FdsPartial\FdsAk
    = \sum_{\FdsIndx=1}^d\FdsPartial\Tilde{\vel}_{\FdsIndx}
    = \TranMap\Grad{\cdot\vel} - \norm{\vel}^2_2\gamma\TranMap
    > 0,
\eeq
where the last inequality will hold by the assumption \eqnref{LinAdv_assumption} and by taking $\gamma<0$. Therefore, condition \itmref{A4} is also satisfied. Finally, 
 $\Null{\FdsABnd}=\Null{\Tilde{\vel}\cdot\normal}=\LRc{0}$ along all surfaces of the elements since the continuity of $\vel\cdot\normal$ is assumed and the mapping $\TranMap$ is diffeomorphism. With the aid of Lemma \lemref{Fds_one_hp_localCon} and Theorem \theoref{Fds_one_hp_globalCon}, we can conclude that the $hp$-HDG formulation for \eqnref{LinAdv_Tran} is well-posed both locally and globally. Given that the mapping $\TranMap$ is bijective, this conclusion is also valid for \eqnref{LinAdv}.
\end{proof}

\subsection{Mixed PDE}
For mixed PDE, we consider:
\ben[label=(HB.\arabic*)]
    \item \itmlab{hb1} steady-state convection-diffusion problem  with discontinuous inflow condition,
\een
\begin{equation}\eqnlab{ConvDiff}
\begin{split}
    &\Div{\LRp{\vel\sol - \diffCoeff\Grad\sol}} = \forcing, \quad\text{ in }\dom,\\
    &\begin{cases}
    \sol = \BCVal^{\sol}, \quad\text{ on }\OutBnd, \\
    \LRp{\vel\sol - \diffCoeff\Grad\sol}\cdot\normal = \BCVal^{\sol}, \quad\text{ on }\InBnd\cup\ZeroBnd,
    \end{cases} 
\end{split}
\end{equation}
where the boundary data $\BCVal^{\sol}:\domBnd\mapsto\real$ is in $\Lsp{2}(\domBnd)$ and is defined as
\beq
\BCVal^{\sol}
=
\begin{cases}
    \DirVal\text{ on }\OutBnd,\\
    \BCVal_{N,R}\text{ on }\ZeroBnd\cup\InBnd.
\end{cases}
\eeq
Here, $\vel\in\LRs{\Lsp{\infty}\LRp{\dom}}^{d}$, $\Div{\vel}\in\Lsp{\infty}\LRp{\dom}$, $\forcing\in\Lsp{2}\LRp{\dom}$, and $\diffCoeff$ whose lowest eigenvalue is uniformly bounded away from zero is a symmetric positive definite matrix-valued defined on $\dom$. In addition, $\InBnd\cup\OutBnd\cup\ZeroBnd=\domBnd$ where $\InBnd=\LRc{x \in \partial\Omega: \vel\cdot\normal<0}$ is an inflow boundary; $\OutBnd\LRc{x \in \partial\Omega: \vel\cdot\normal>0}$ is an outflow boundary; and $\ZeroBnd\LRc{x \in \partial\Omega: \vel\cdot\normal=0}$ is a zero-flow boundary. It is evident that $\OutBnd=\domBndDir$ and $\ZeroBnd\cup\InBnd=\domBndNmn\cup\domBndRb$. The problem can be analyzed by a two-field Friedrichs' system with partial coercivity. The size of the system is $\FdsNVar=d+1$, $\FdsNVarAux=d$, and $\FdsNVarSol=1$. The two-field Friedrichs' system for this problem reads:
\beq\eqnlab{Fds_block_mat_mixed}
\FdsG =
\begin{bmatrix}
\diffCoeff^{-1} & 0_{d\times1}\\
0_{1\times d} & 0
\end{bmatrix}\quad \text{and}\quad 
\FdsAk =
\begin{bmatrix}
0_{d \times d} & \normalCon_{\FdsIndx}\\
\normalCon_{\FdsIndx}^T & \velComp{\FdsIndx}
\end{bmatrix}.
\eeq
We further assume that:
\beq\eqnlab{ConvDiff_assumption}
    \infess_{\dom}\LRp{\half\Div{\vel}}\geq 0,
\eeq
to gain partial coercivity. The boundary conditions can be enforced by specifying the boundary operator $\FdsMBnd$ as
\beq\eqnlab{FdsMBnd_mixed}
    \begin{cases}
    \alpha=+1,\,\FdsMss=2{\varrho}\id_{\FdsNVarSol} + \vel\cdot\normal, \text{ on }\domBndDir,\\
    \alpha=-1,\,\FdsMss=2{\varrho}\id_{\FdsNVarSol} + \vel\cdot\normal,\text{ on }\domBndNmn\cap\domBndRb.\\
    \end{cases}
\eeq
In addition, we set $\varrho=\half$ on $\domBndDir$, $\varrho=0$ on $\domBndNmn$ and $\varrho=-\vel\cdot\normal$ on $\domBndRb$ and the conditions \eqnref{Abstrct_M1} and \eqnref{Abstrct_M2} hold. Finally, the following condition is also assumed:
\begin{align}\eqnlab{ConvDiff_flux_assumption}
    & \vel\cdot\normal\neq0\text{ on }\mort,\quad\forall\mort\in\skel,
    && \jump{\vel\cdot\normal}=0\text{ on }\elemBnd,\quad\forall\elem\in\domPart,
\end{align}
where we assume $\vel\cdot\normal$ is always continuous across element edges and does not vanish at edges (or mortars). As mentioned in the previous example, this condition can be relaxed by modifying the derivation of the upwind flux.
\begin{lemma}
The $hp$-HDG formulation for the PDE stated in \eqnref{ConvDiff} is well-posed both locally and globally if the assumptions \eqnref{ConvDiff_assumption} and \eqnref{ConvDiff_flux_assumption} hold.
\end{lemma}
\begin{proof}
The assumptions \itmref{A1}-\itmref{A3}, \itmref{A4b}, and \itmref{A5}-\itmref{A6} hold true and can be easily verified. In addition, \itmref{A4a} is also true if the assumption \eqnref{ConvDiff_assumption} holds. On the other hand, the eigendecomposition of $\FdsABnd$ for the convection-diffusion problem is: 
\beq
    \FdsAbsA =  \frac{1}{\sqrt{\abs{\vel\cdot\normal}^2+4}}   
    \begin{bmatrix}
    2\normal\normal^T & \LRp{\vel\cdot\normal}\normal\\
    \LRp{\vel\cdot\normal}\normal^{T} & \abs{\vel\cdot\normal}^2+2\\
    \end{bmatrix}.
\eeq
Thus, with the following setting of $\FluxAssumpA=\frac{2}{\vel\cdot\normal}$ and 
$\FluxAssumpB=\frac{\vel\cdot\normal}{\sqrt{\abs{\vel\cdot\normal}^2+4}}
$,
the hypothesis \itmref{F1} of the numerical flux will hold since $\vel\cdot\normal\neq0$ across all elements. Then we have the stabilization parameter $\stabPar=\half\LRp{\sqrt{\abs{\vel\cdot\normal}^2+4} - \vel\cdot\normal}$. 
This leads to the following:
\bit
    \item $\half\FdsCBnd + \stabPar=\half\sqrt{\abs{\vel\cdot\normal}^2+4}>0$, and
    \item $\bigcap_{\FdsIndx=1}^d\text{Range}\LRp{\normalCon_{\FdsIndx}}=\emptyset$ and $\Null{\normalCon_{\FdsIndx}}=\LRc{\bs{0}}$ for $\forall\FdsIndx=1,\dots,d$. 
\eit
\,\\
By Lemma \lemref{Fds_two_partial_hp_localCon} and Theorem \theoref{Fds_two_partial_hp_globalCon}, we conclude that the HDG formulation for \eqnref{ConvDiff_flux_assumption} is well-posed both locally and globally.
\end{proof}

\subsection{Numerical settings and results}
For the numerical experiments of all problems stated before, \itmref{e1}-\itmref{e3},\itmref{hp1}, and \itmref{hb1}, they are defined on the square domain $\dom=[0,1]\times[0,1]$ except for \itmref{e3} which is defined on the rectangular domain $\dom=[0,8.4]\times[0,24]$. In addition, they are initially solved on the simple meshes as shown in Figure \figref{init_mesh} with $\polyElem=2$ for $\forall\elem\in\domPart$ at the $0$-th cycle of adaptation. 

\begin{figure}[!htb]
\centering
\subfloat[Initial mesh 1]{
	\includegraphics[width=55mm]{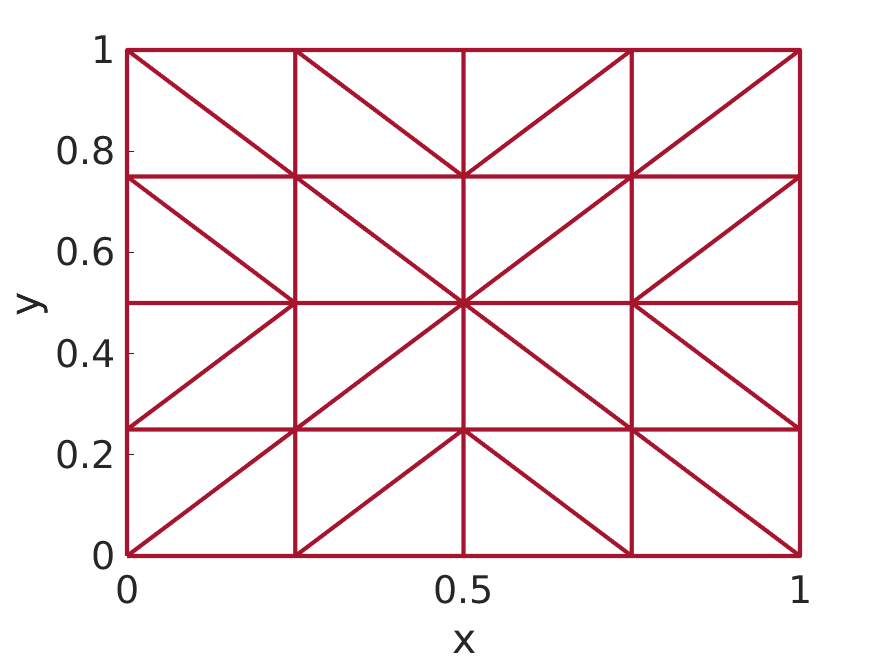}
}
\subfloat[Initial mesh 2]{
	\includegraphics[width=55mm]{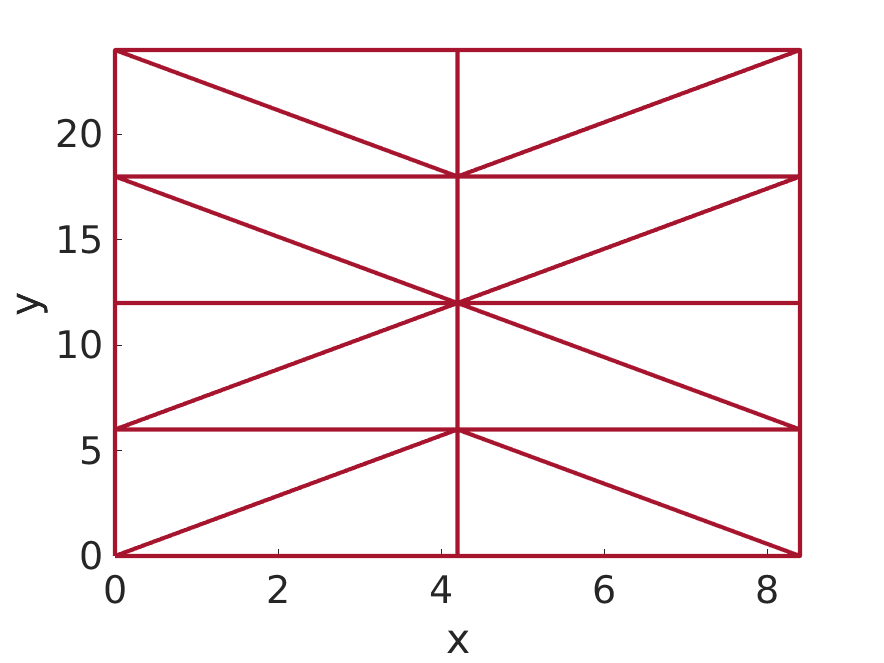}
}
\caption{(a) Initial mesh for all the numerical examples except \itmref{e3}. (b) Initial mesh used only for \itmref{e3}.}  
\figlab{init_mesh}
\end{figure}
\FloatBarrier

The solver developed in this work is built upon a \texttt{MATLAB} code discussed in \cite{Hesthaven2007}. For the numerical evaluation of integrals, the cubature rule is used over elements and the Gauss quadrature over the surfaces of elements. The adaptation is performed using Algorithm \algmref{hp_algm} with different error indicators stated in \eqnref{error_est} for all problems. Convergence histories of $\Lsp{2}$-error norm are also presented if the exact solutions are available. 

\subsection*{\itmref{e1} Poisson problem with a corner singularity}
Consider the Poisson problem stated in \eqnref{AniDiff} where the diffusivity coefficient $\diffCoeff$ is set to be the identity matrix, the forcing term $\forcing$ is set to be zero and the exact solution is given below:
\begin{equation}\eqnlab{e1_sol}
    \sol\LRp{\x,\y}=2\LRp{\x^2 + \y^2}^{-3/4}\x\y\LRp{1 - \x}\LRp{1 - \y}.
\end{equation}
Dirichlet boundary condition 
is applied to all the boundaries  such 
that the solution can satisfy \eqnref{e1_sol}. It can be shown (see \cite{Babuska1990}) that the solution presented in \eqnref{e1_sol} is singular at the origin $\LRp{\x,\y}=\LRp{0,0}$, but is regular in the rest of the domain $\dom$. The problem is also studied in \cite{Dolejsi2013,Dolejsi2015}.\\

We define the output functional for the adjoint formulation \eqnref{Fds_two_ad_hdg} by
\beq\eqnlab{e1_output}
    \Output{\FdsVarLump}:=-\LRa{\FdsVarAuxAppr\cdot\normal+{\FdsVarSolAppr+\FdsTrcVarSolAppr},1}_{\domBnd_{\text{left}}\cup\domBnd_{\text{bottom}}}.
\eeq
It is the integration of the solution over the boundaries where a singularity exists. As consequence, an error induced by the singularity should be detected by the adjoint. Hence, the error indicator will be small if the solution near the singularity is well-resolved.

Figure \figref{e1_result} shows the numerical results at the final cycle of adaptation where two different local error indicators are used to drive the adaptation process with tolerance $\tol=0.01$. The pattern of mesh configurations shown in Figure  \figref{e1_result} exactly meets our expectations where aggressive $h$-refinement takes place near the singularity while intensive $p$-refinement occurs in the other part of the domain. The mesh configurations produced by the two different methods are similar. Around the singularity, numerous small elements with low-order approximation are generated by the adaptation procedure, while away from it, 
a few large elements are generated with high-order approximation. 
However, overall, higher degrees of approximation are generated in the adjoint approach as opposed to the Doleji's approach.

Despite the presence of intense oscillation near the singularity, $hp$-adaptation forces the oscillation zone to shrink.
As demonstrated in Figure \figref{e1_result}, the numerically polluted area is significantly reduced to a tiny region at the final cycle of the adaptation. This improvement can also be seen  in Figure \figref{e1_conv} where convergence histories of $\Lsp{2}$-error norm of $\solAppr$ are plotted. Figure \figref{e1_conv} shows a convergence study with different tolerance values $\tol$ as well. 
For the lower tolerance $\tol=0.01$, both approaches show good convergence behavior, but Doleji's method requires fewer degrees of freedom than the adjoint method at a given error level.
For the higher tolerance $\tol=0.1$, however, the convergence rate for Doleji's approach is flattened out near $10^{3}$ degrees of freedom, whereas the adjoint counterpart still converges to the true solution with increasing degrees of freedom beyond $10^{3}$.

\subsection*{\itmref{e2} Anisotropic diffusion problem with the discontinuous Dirichlet boundary condition}
In this example, we consider a strongly anisotropic diffusion problem stated in \eqnref{AniDiff} where the diffusivity coefficient $\diffCoeff$ is set to be  $\misalign=\pi/4$ and $\anisoRatio=1000$, and the forcing term $\forcing$ is set to be zero. In addition, the following piecewise constant Dirichlet boundary condition is applied to all boundaries:
\begin{equation}\eqnlab{e2_bc}
    \DirVal=
    \begin{cases}
    1\quad\text{when }\x=1\text{ or }\y=0,\\
    0\quad\text{when }\x=0\text{ or }\y=1.
    \end{cases}
\end{equation}
It should be noted that there are discontinuities at the corners $(0,0)$ and $(1,1)$. This problem is also investigated in \cite{Umansky2005,Li2010}. A semi-analytic solution for this test problem can be found by a sequence of geometric transformations which are numerically computed using  \texttt{MATLAB} Schwarz-Christoffel toolbox \cite{Driscoll1996}. Given that the accuracy of the mapping is sufficient enough, we treat this semi-analytic solution as ``exact" to benchmark
against our $hp$-HDG solution.\\

The output functional used in the adjoint formulation for this problem is specified as:
\beq\eqnlab{e2_output}
    \Output{\FdsVarLump}:=-\LRa{\FdsVarAuxAppr\cdot\normal+\FdsVarSolAppr+\FdsTrcVarSolAppr,\cos{2\pi\phyCompX}\cos{2\pi\phyCompY}}_{\domBnd}
\eeq
where it is the integration of weighted solutions over all boundaries which are all Dirichlet boundaries in this testing case. As a result, the adjoint will try to detect the region covered by the sinusoidal function that is across the entire domain $\dom$. Such sinusoidal weighted output over a boundary is also considered in \cite{Woopen2014b}.

With the tolerance $\tol=0.01$, we plot the numerical results of both the approaches in Figure \figref{e2_result} at the adaptation cycle where $\Lsp{2}$-error norm of $\solAppr$ are similar (it is about $O(10^{-3})$). 
Due to the strong anisotropic feature, the solution behaves like convection where the amount of the flux transported in the specific direction is far more favorable than in the other direction. In this example, the dominant direction is 45 degrees from the $\x$-axis. As a result, ``discontinuity"-like behavior occurs within the domain along the diagonal. 
Due to the presence of discontinuous Dirichlet boundary data, Gibbs phenomenon \cite{gottlieb1997} occurs around the corners $(0,0)$ and $(1,1)$.
Similar to the numerical result shown in \itmref{e1}, the numerically polluted area can significantly be reduced though not removed completely via the adaptation process. This observation is also consistent with the convergence histories presented in Figure \figref{e2_conv} where the results obtained by different tolerances with various anisotropic ratios are shown. It can be observed in the convergence histories that a small tolerance value is required in this case to achieve acceptable convergence rates. However, the adaptation process stops due to the criteria $\max_{\elem\in\domPart}{\LocErrEstGen{h}}\geq\tol\GlbErrEstGen{h}$ for Doleji's approach while the process can proceed further for the adjoint approach and stops since  the maximum number of iterations is reached. Finally, we would like to point out that increasing the anisotropic ratio $\anisoRatio$ will make the problem harder to be solved in that the profile of the $\FdsVarSolAppr$ will tend to be even steeper. A shock-like front may form, which causes more $h$-refinement and hence more DOFs. In summary, both approaches are comparable in this case but more robust stopping criteria may be needed. The effort in designing robust stopping criteria may not be trivial and hence we will include it in our future work. 


\subsection*{\itmref{e3} Heterogeneous anisotropic diffusion problem with  discontinuous field $\diffCoeff$}
Here we  consider the problem stated in \eqnref{AniDiff}, but with a piecewise constant diffusivity coefficient $\diffCoeff$ and Neumann/Robin mixed type  boundary conditions
\beq\eqnlab{e3_robin}
    \diffCoeff\Grad{\sol}\cdot\normal + \rtCoeff\sol = \BCVal,\quad\text{on}\ \domBndNmn\cup\domBndRb, 
\eeq
where $\domBnd=\domBndNmn\cup\domBndRb$. Given that $\diffCoeff$ is now a function of the physical domain $\LRp{\x, \y}$, the problem is not only being anisotropic but also heterogeneous. 
The heat conduction in non-homogeneous materials can be modeled using this PDE, 
where $\sol$ is the unknown temperature field. For example, a so-called ``battery problem" \cite{Demkowicz2006}, is this type of model and examined here. The domain is then modeled as a battery composed of five different materials which are indexed as numbers $1$-$5$ and specified in Table \tabref{e3_geometry}. The values of $\diffCoeff$ for different materials and the corresponding forcing term $\forcing$ are summarized in Table \tabref{e3_coeff}. The boundary data is given in Table \tabref{e3_bc}.\\

The output functional used here
is specified as:
\beq\eqnlab{e3_output}
    \Output{\FdsVarLump}:=-\LRa{\FdsVarAuxAppr\cdot\normal+\FdsVarSolAppr+\FdsTrcVarSolAppr,1}_{\domBnd_{\text{left}}},
\eeq
which is the integration of the solution over the left boundary which is actually a Neumann boundary. That is, the disturbance induced by the numerical error within the domain $\dom$ will also be propagated here and hence can be detected by the adjoint.  

Figure \figref{e3_result} shows the numerical results of the two methods at the final cycle of adaptation 
with the tolerance $\tol=0.01$. This problem is challenging in that the coefficient of the PDE is discontinuous across the entire domain. Without aligning the mesh skeleton with these discontinuities, a serious Gibb's phenomenon is easily induced.
This is the situation in this experiment where each material does have a smooth solution, enjoying high order degree approximation, while each material has discontinuous diffusivity across the interfaces between the materials, resulting in severe anisotropy and requiring $h$-refinement. 
If we simply employ an isotropic $h$-refinement, then the numerically polluted area may still spread to some extent if the discontinuous diffusivity is not sufficiently resolved. One remedy for solving this issue is to use anisotropic $h$-refinement \cite{dolejsi1998,Li2010,ceze2013,Balan2016,bartos2019}. However, such refinement requires a more delicate error estimator/indicator and needs to be equipped with a proper algorithm for generating a mesh. This task is left for our future work.

\begin{table}[!htb]
\centering
\begin{tabular}{cc}
\hline
Material & Region \\
\hline
1 & $[0,\,8.4]\times[0,\,0.8)$, $(8,\,8.4]\times[0.8,\,23.2]$, $[0,\,8.4]\times(23.2,\,24]$\\
2 & $[0,\,6.1)\times[1.6,\,3.6)$, $[0,\,6.1)\times[18.8,\,21.2)$\\
3 & $[0,\,6.1)\times[3.6,\,18.8)$\\
4 & $[6.1,\,6.5)\times[0.8,\,21.2)$\\
5 & $[0,\,6.1)\times[0.8,\,1.6)$, $(6.5,\,8)\times[0.8,\,21.2)$, $[0,8)\times[21.2,\,23.2)$\\
\hline
\end{tabular}
\caption{The geometry of materials of the battery modeled by \itmref{e3}.}\tablab{e3_geometry}
\end{table}

\begin{table}[!htb]
\parbox{.5\linewidth}{
\centering
\begin{tabular}{cccccc}
\hline
Material & $\diffCoeffCompX$ & $\diffCoeffCompY$ & $\anisoRatio$ & $\misalign$ & $\forcing$\\
\hline
1 & 25.0& 25.0   & 1.00 & 0.0& 0.0\\
2 & 7.0 & 0.8    & 8.75 & 0.0& 0.0\\
3 & 5.0 & 0.00001& $5.00\times10^5$& 0.0& 1.0\\
4 & 0.2 & 0.2    & 1.00 & 0.0& 1.0\\
5 & 0.05& 0.05   & 1.00 & 0.0& 0.0\\
\hline
\end{tabular}
\caption{Data of diffusivity coefficient $\diffCoeff$ and of forcing term $\forcing$ for \itmref{e3}.}\tablab{e3_coeff}
}
\hfill
\parbox{.45\linewidth}{
\centering
\begin{tabular}{ccc}
\hline
BC data & $\rtCoeff$ & $\BCVal$\\
\hline
Left    & 0.0   & 0.0 \\
Up      & 1.0   & 3.0 \\
Right   & 2.0   & 2.0 \\
Bottom  & 3.0   & 1.0 \\
\hline
\end{tabular}
\caption{Boundary data for \itmref{e3}.}\tablab{e3_bc}
}
\end{table}\tablab{e3_bc_data}


\subsection*{\itmref{hp1} Steady-state linear advection with  variable speed and  discontinuous inflow condition}
In this experiment, we are going to solve the linear advection problem described in \eqnref{LinAdv} along with the advection velocity $\vel=(1+\sin{(\pi \y)},\ 2)$ and the inflow data that is given as:
\beq\eqnlab{hp1_bc}
    \DirVal = \begin{cases}
    1,\quad\text{for }\x = 0,\ 0 \leq \y \leq 1,\\
    \sin^6(2 \pi \x),\quad\text{for } 0 \leq \x \leq 0.5,\ \y = 0,\\
    0,\quad\text{for } 0.5 \leq \x \leq 1,\ \y = 0,
    \end{cases}
\eeq
where there is a discontinuity occurring right at the origin. The problem is also studied in \cite{Tan2015,Muralikrishnan2017} and can be solved exactly by using the method of characteristics.\\

The output functional $\OutFcnl$ for the adjoint formulation is specified as:
\beq\eqnlab{hp1_output}
    \Output{\FdsVarLump}:=\LRp{\half\LRp{-\vel\cdot\normal-\abs{\vel\cdot\normal}}\widehat{\sol}_h,\cos{2\pi\phyCompX}\cos{2\pi\phyCompY}}_{\domBnd},
\eeq
which is the integration of the solution over the outflow boundaries. The design of the output functional \eqnref{hp1_output} is similar to the one we used in \eqnref{e2_output} for the elliptic problem. The wave-carrying nature of the hyperbolic problem implies that $\widehat{\sol}_h$ can be well-solved at outflow boundaries only if every data point is properly resolved when it is traced back from these boundaries to inflow ones. As a result, it can be expected that the error indicator can be minimized if $\widehat{\sol}_h$ is well-resolved at these boundaries. 

In Figure \figref{hp1_result}, 
we see the numerical results of both methods at the final cycle of adaptation. 
Due to the discontinuous inflow boundary data and the nature of hyperbolic PDEs, we have a shock formed within the domain $\dom$. It is very challenging to remove
the oscillation induced by Gibbs' phenomena unless all the discontinuities are well aligned with the skeleton of the mesh along with the first order of approximation. Given that we only consider isotropic $h$-refinement here, there is no way to meet this condition. However, we can still narrow down the region of shock-induced oscillation by the $hp$-adaptation process. As we expect, the aggressive $h$-refinement is performed around the shock. However, adaptation is carried out in a less aggressive way in the near-outflow region. We numerically found that this less aggressive behavior becomes less dominant when using the adjoint approach. It is numerically found that the aggressive $h$-refinement from the inflow to outflow boundary can still be obtained if we use finer initial mesh or  higher initial $\polyElem$.

Figure \figref{hp1_conv} presents the convergence histories of $\Lsp{2}$-error norm of $\solAppr$ using both approaches along with two different tolerance values. It can be observed that, again, the convergence rate can only be improved (not zero anymore) if the tolerance is set to be small enough. Otherwise, the two approaches are comparable in this example as well.

\subsection*{\itmref{hb1} Steady-state convection-diffusion problem with  discontinuous inflow condition}
In this example, we focus on the steady-state convection-diffusion equation \eqnref{ConvDiff} and especially examine the problem first proposed by Eriksson and Johnson in \cite{Eriksson1993}. This problem is also investigated in \cite{Chan2014} using a discontinuous Petrov-Galerkin method. The diffusivity matrix is set to be $\diffCoeff:=\epsilon\id_2$ where $\id_2$ is $2\times2$ identity matrix, and the velocity field is stated as $\vel:=\LRp{0,1}$. The boundaries are given as follows:
\begin{align*}
    \InBnd = \LRc{(\x,\y): \x=0,\, 0\leq \y \leq 1},\\
    \OutBnd = \LRc{(\x,\y): \x=1,\, 0\leq \y \leq 1},\\
    \ZeroBnd = \LRc{(\x,\y): 0\leq \x \leq 1,\, \y=0\text{ or }1}.
\end{align*}
The boundary conditions read:
\beq
\begin{split}
    &\DirVal = 0,\\
    &\BCVal_{N,R} = \begin{cases}
    \LRp{\vel\sol_0+\aux_0}\cdot\normal\text{ on }\InBnd,\\
    0\text{ on }\ZeroBnd,\\
    \end{cases}
\end{split}
\eeq
and $\sol_0:=\sol(0,\y)$, $\aux_0:=\aux(0,\y)$. Further, the function $\sol_0$ is set to be a discontinuous function :
\begin{equation}\eqnlab{hb1_bc}
    \sol_0(\y)=
    \begin{cases}
    (\y-1)^2,\quad \y > 0.5,\\
    -\y^2,\quad \y\leq 0.5.
    \end{cases}
\end{equation}
The Eriksson-Johnson problem can be solved by the separation of variables and the solution is:
\begin{equation}\eqnlab{hb1_sol}
    \sol(\x,\y) = C_0 + \sum_{i=1}^{\infty} C_i\frac{e^{s_2(\x-1)} - e^{s_1(\x-1)}}{e^{-s_2} - e^{-s_1}}\cos{\LRp{i\pi \y}},
\end{equation}
where
\begin{align*}
    &C_i = \int_0^1 2\sol_0\cos{\LRp{i\pi \y}}\,d\y,\\
    &s_{1,2} = \frac{1\pm\sqrt{1+4\epsilon\sigma_i}}{2\epsilon},\\
    &\sigma_i = \epsilon i^2\pi^2.
\end{align*}
In this testing case, we actually have $\InBnd=\domBndDir$, $\OutBnd=\domBndRb$, and $\ZeroBnd=\domBndNmn$. Note that we do not have a closed form of the exact solution. Therefore, for a convergence study, we approximate $\sol_0$ using the first 20 terms of the series in \eqnref{hb1_sol}. Similarly, $\aux_0$ can be approximated in the same way. The problem is tricky because there is not only discontinuous inflow data but also a boundary layer developed around the outflow boundary. In addition, the smaller the diffusivity coefficient $\epsilon$ is, the thinner the boundary layer. To have the layer well-resolved, there must be a mesh with proper resolutions. \\

The output functional $\OutFcnl$ for the adjoint formulation is specified as:
\beq\eqnlab{hb1_output}
    \Output{\FdsVarLump}:=\LRp{\FdsTrcVarSolAppr,1}_{\domBnd_{\text{right}}}
\eeq
where the integration is defined over the outflow boundary $\OutBnd$. The reason why we choose this output functional is two folds. 
First, a boundary layer is developed contingent on the outflow boundary. 
Second, inspired by the design of \eqnref{hp1_output}, there is still a wave-carrying feature in a mixed type of problem though the wave is dissipated while it travels. In brief, we should expect that these two issues can be properly addressed by reducing the value of the error indicator associated with \eqnref{hb1_output}.

Figure \figref{hb1_result} shows the numerical results of the two approaches with diffusivity $\epsilon=10^{-3}$ and tolerance $\tol=0.05$ at the final cycle of adaptation. 
As discussed previously, there is a boundary layer (sharp gradient in solution $\solAppr$) around the outflow boundary $\x=1$. In this example, the resulting $\polyElem$ maps for the two methods are significantly different. For Doleji's approach, 
even if there is another steep gradient at the inflow boundary when discontinuous Dirichlet boundary data is applied, the sharp gradient created in the boundary layer dominates the $hp$-adaptation process.
As a result, less refinement is done in the inflow region. Specifically speaking, 
because the local error indicator value $\LocErrEst{h}{\text{Doleji}}$  surrounding the boundary layer is higher than that of the area around the inflow boundary, less refinement is performed there unless an even smaller tolerance value is provided.
Furthermore, 
the method used a low degree of approximation and failed to detect smooth regions.
On the other hand, for the adjoint approach, sharp gradient features in both inflow and outflow boundaries are captured by $h$-refinement. However, we need more mesh refinement at the outflow boundary to have the boundary layer well-resolved. 
The improperly resolved boundary layer was the cause of the adjoint approach's flattened convergence rate in Figure \figref{hb1_conv}.

Figure \figref{hb1_conv} shows the convergence histories of $\Lsp{2}$-error norm of $\solAppr$ for two different approaches together with three different diffusivities. To capture the thinner boundary layer, we employ a more strict tolerance value when the solution is less diffusive and therefore more advection-dominated.
In this testing case, Doleji's approach slightly outperforms the adjoint approach in terms of accuracy and convergence rate.

\begin{figure}[!htb]
\centering
\subfloat[Tolerance $\tol=0.1$]{
	\includegraphics[width=80mm]{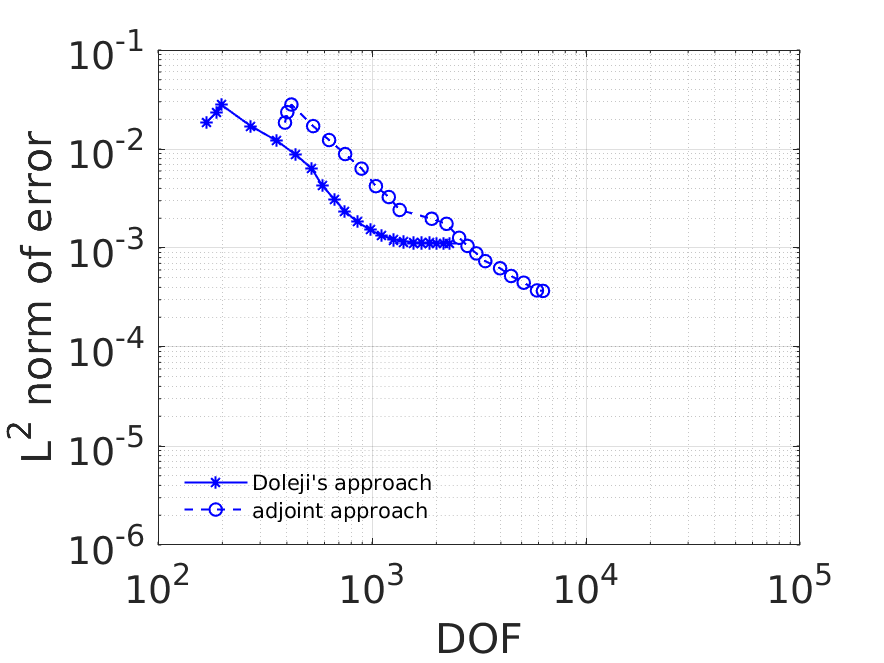}
}
\subfloat[Tolerance $\tol=0.01$]{
	\includegraphics[width=80mm]{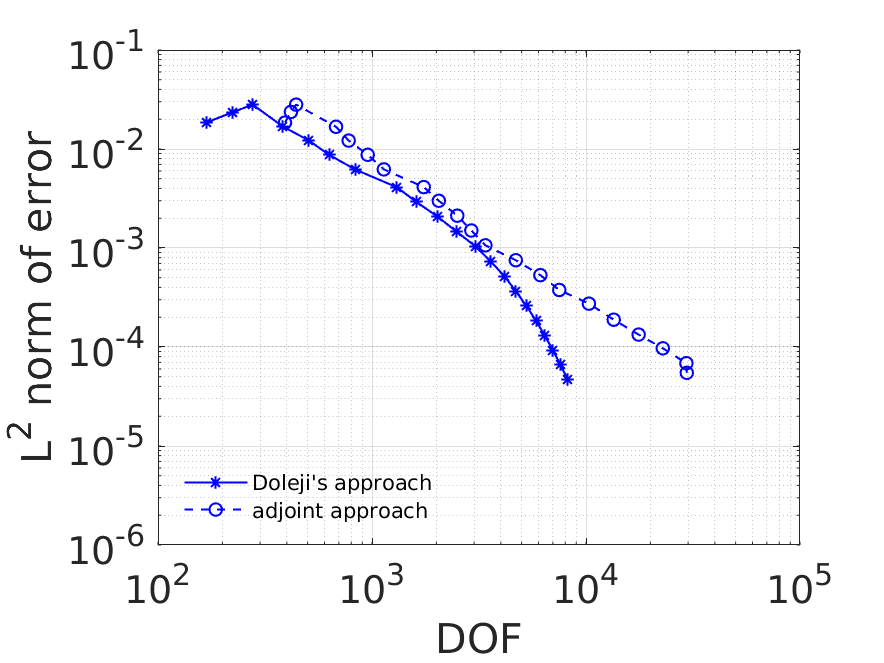}
}
\caption{Convergence histories of $\Lsp{2}$-error norm of $\solAppr$ by numerically solving the elliptic problem \itmref{e1} that admits the exact solution stated in \eqnref{e1_sol} with different tolerance (a) $\omega=0.1$ and (b) $\omega=0.01$. In each plot, the results obtained by Doleji's and adjoint approaches are presented. In particular, the adjoint approach is derived with regard to the output functional defined in \eqnref{e1_output}.}
\figlab{e1_conv}
\end{figure}

\begin{figure}[!htb]
\centering
\subfloat{
	\includegraphics[width=80mm,valign=t]{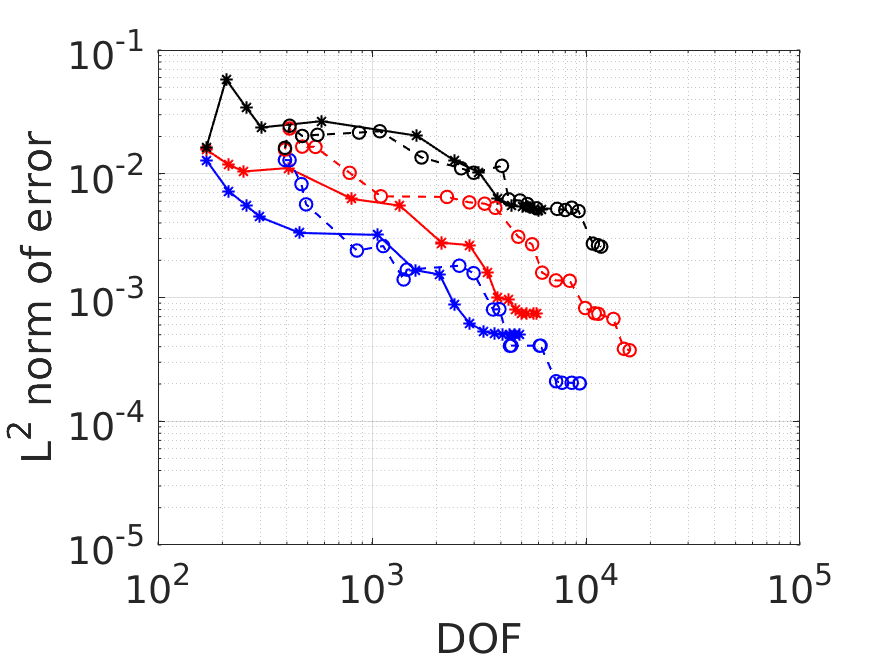}
}
\subfloat{
	\includegraphics[width=80mm,valign=t]{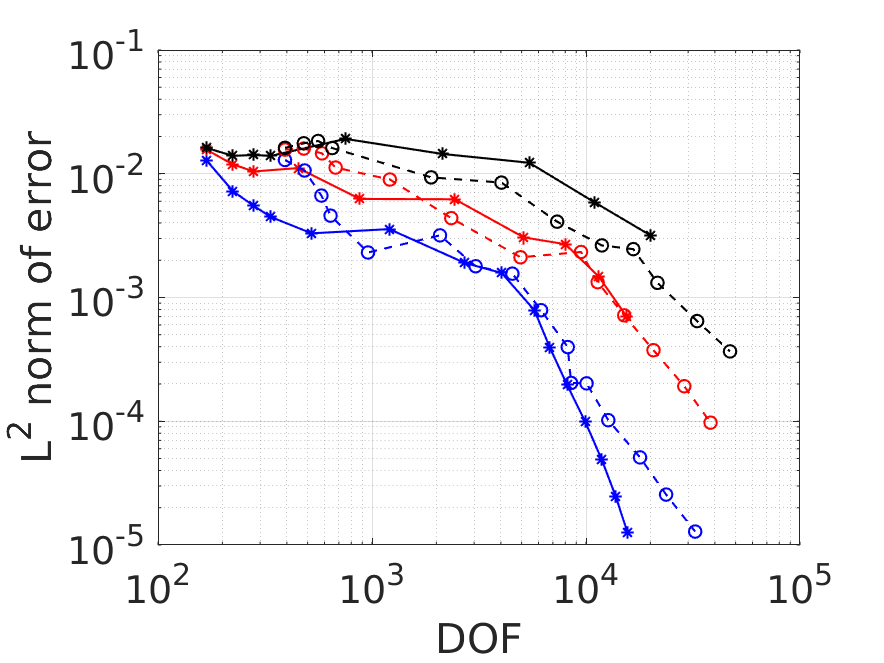}
}\\
\setcounter{subfigure}{0}
\subfloat[Tolerance $\tol=0.1$]{
	\includegraphics[width=50mm]{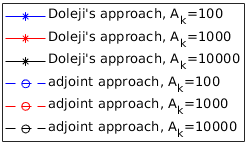}
}
\subfloat[Tolerance $\tol=0.01$]{
	\includegraphics[width=50mm]{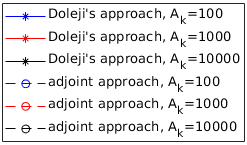}
}
    \caption{Convergence histories of $\Lsp{2}$-error norm of $\solAppr$ by numerically solving the elliptic problem \itmref{e2} where the semi-analytic solution can be obtained with the aid of accurate mappings. Different tolerance values of (a) $\tol=0.1$ and (b) $\tol=0.01$ are used. In each plot, the results with various anisotropy ratios $\anisoRatio$ (denoted by different colors) that are obtained by different approaches (denoted by different marks) are presented. The adjoint approach is derived with regard to the output functional defined in \eqnref{e2_output}.}
\figlab{e2_conv}
\end{figure}

\begin{figure}[!htb]
\centering
\subfloat[Tolerance $\tol=0.1$]{
	\includegraphics[width=80mm]{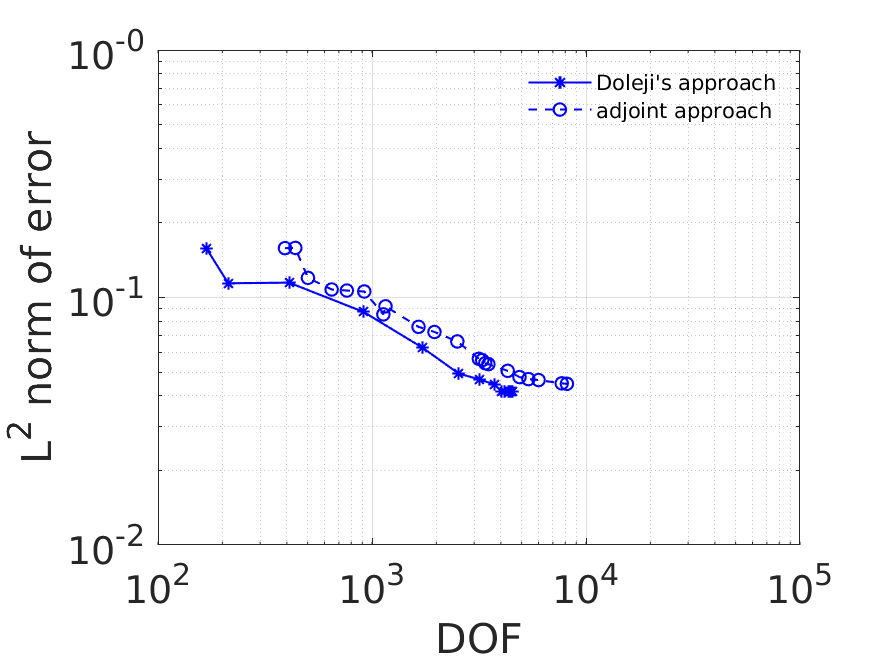}
}
\subfloat[Tolerance $\tol=0.05$]{
	\includegraphics[width=80mm]{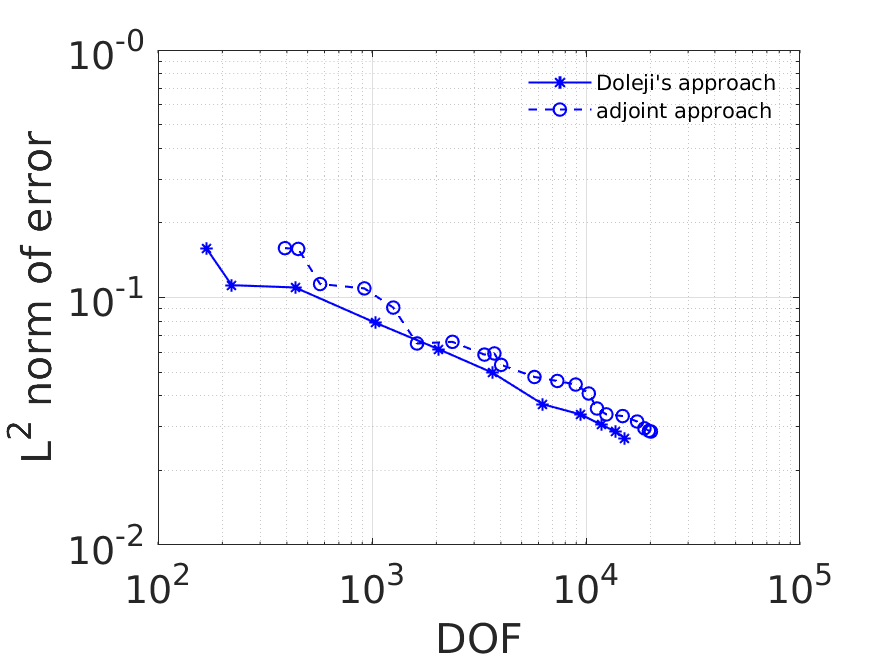}
}
\caption{Convergence histories of $\Lsp{2}$-error norm of $\solAppr$ by numerically solving the hyperbolic problem \itmref{hp1}. The exact solution can be found by the method of characteristics. The results obtained by two different approaches with two different tolerances of (a) $\tol=0.1$ and (b) $\tol=0.05$ are presented. The adjoint approach is derived with regard to the output functional defined in \eqnref{hp1_output}.}
\figlab{hp1_conv}
\end{figure}

\begin{figure}[!htb]
\centering
\subfloat{
	\includegraphics[width=80mm]{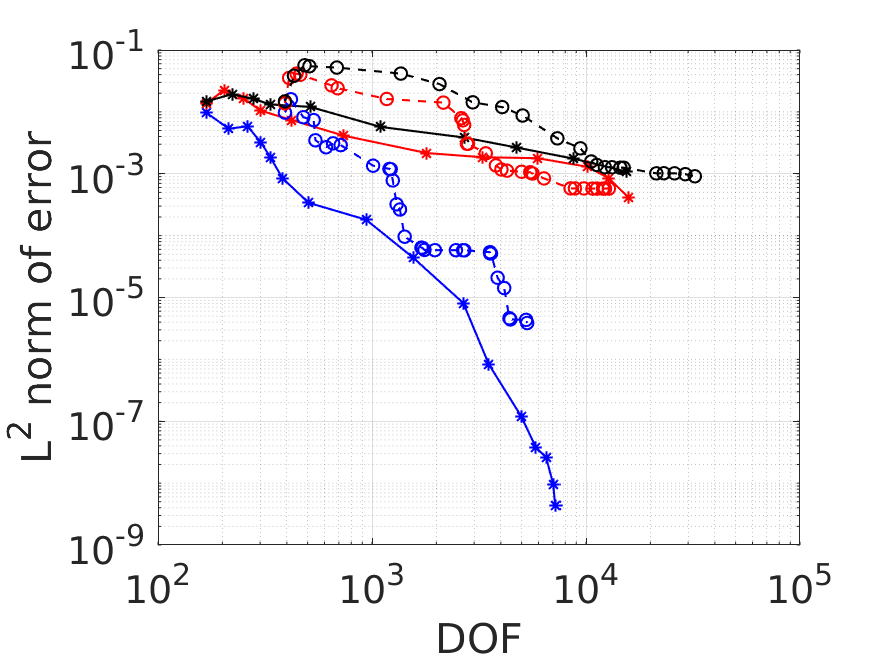}
}
\subfloat{
	\includegraphics[width=50mm]{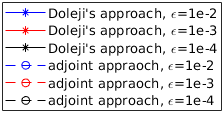}
}
\setcounter{subfigure}{0}
\caption{Convergence histories of $\Lsp{2}$-error norm of $\solAppr$ by numerically solving the mixed problem \itmref{hb1} that admits the exact solution stated in \eqnref{hb1_sol}. The results with various diffusivity values $\epsilon$ (denoted by different colors) that are obtained by different approaches (denoted by different marks) are presented. For different diffusivity values $\epsilon=10^{-2},10^{-3}$ and $10^{-4}$, different tolerance $\tol=0.1,0.05$ and $0.01$ are used respectively. The adjoint approach is derived with regard to the output functional defined in \eqnref{hb1_output}}
\figlab{hb1_conv}
\end{figure}

\begin{figure*}
  \centering
\begin{tabular}{cccc}
& Doleji's approach & Adjoint approach\\
\rotatebox[origin=c]{90}{$\solAppr$ (Surface)} &
\includegraphics[trim=15 0 35 10,clip,width=5cm,align=c]{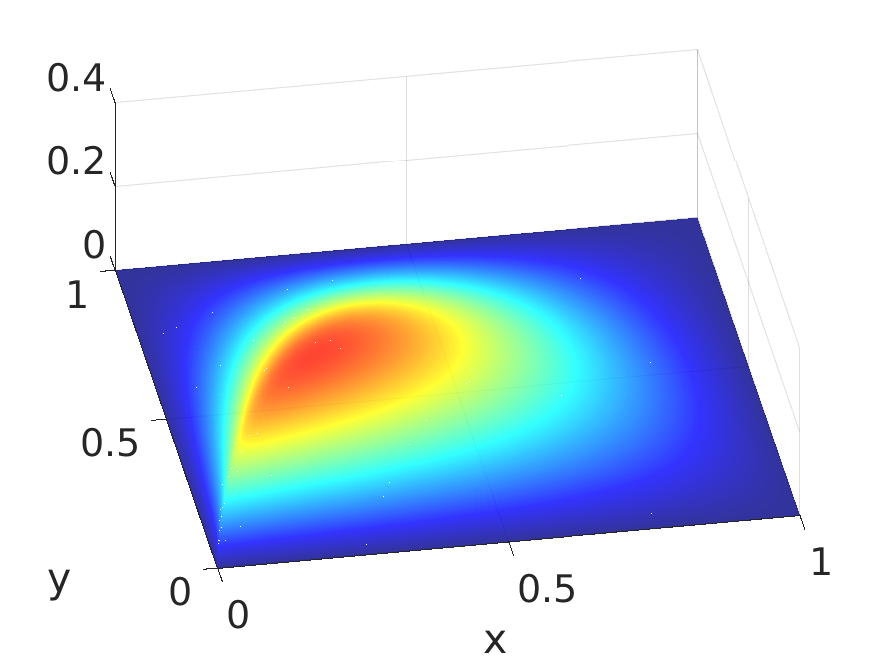}&
\includegraphics[trim=15 0 35 10,clip,width=5cm,align=c]{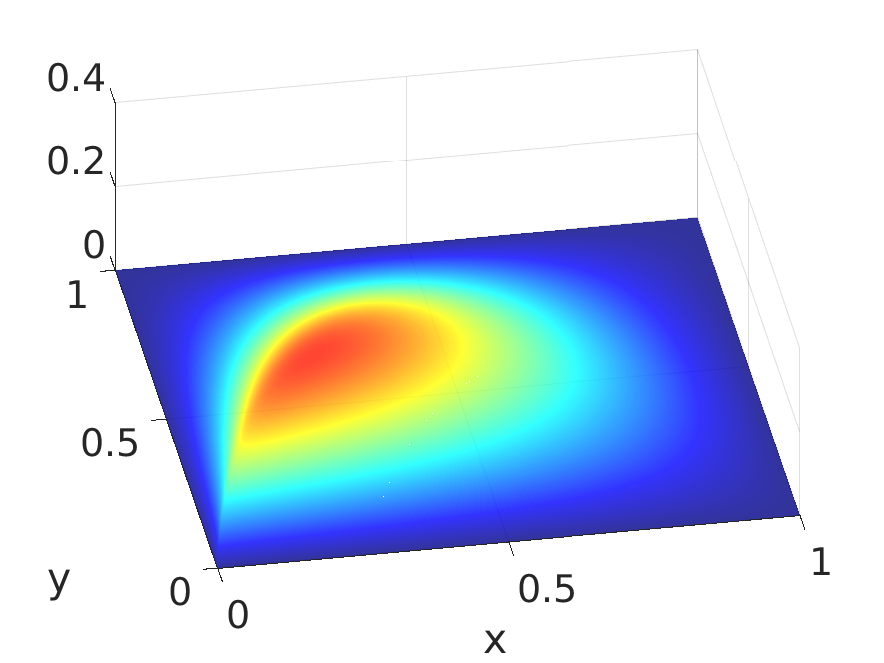}&
\includegraphics[trim=160 0 90 0,clip,width=0.7cm,align=c]{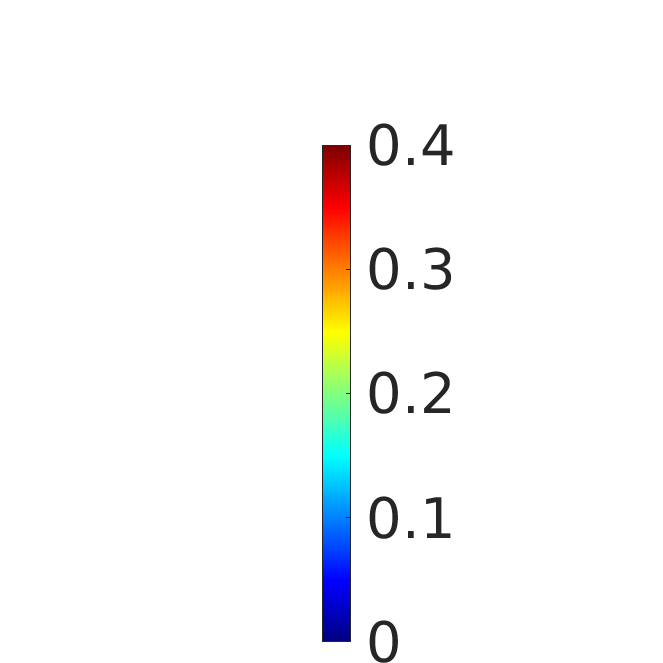}\\
\rotatebox[origin=c]{90}{$\solAppr$ (Contour)}&
\includegraphics[trim=23 15 35 10,clip,width=5cm,align=c]{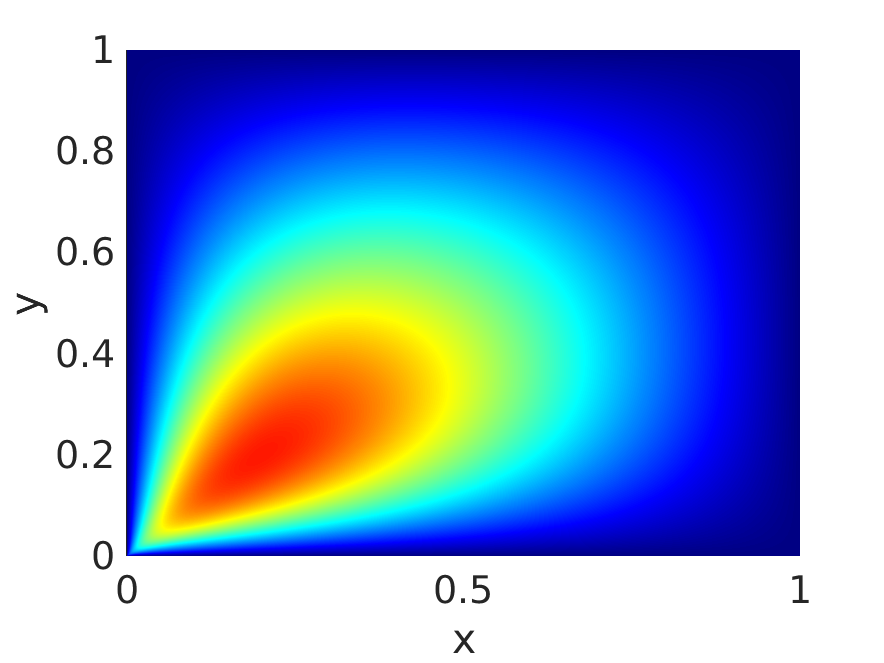}&
\includegraphics[trim=23 15 35 10,clip,width=5cm,align=c]{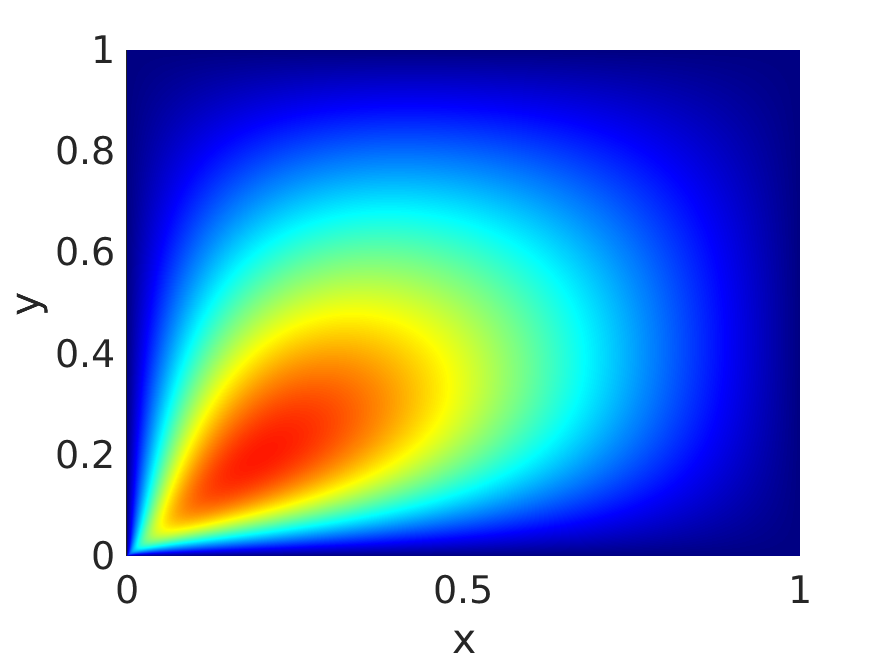}&
\includegraphics[trim=160 0 90 0,clip,width=0.7cm,align=c]{Figures2/Numerical Results/E1/u_colobar_iso_sig.png}\\
\rotatebox[origin=c]{90}{$\polyElem$ map}&
\includegraphics[trim=15 5 35 10,clip,width=5cm,align=c]{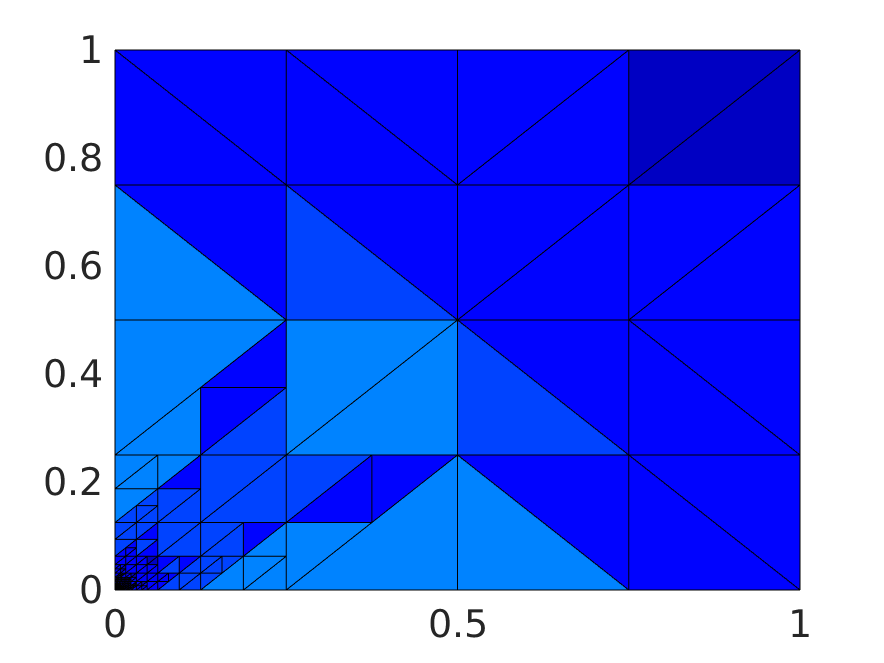}&
\includegraphics[trim=15 5 35 10,clip,width=5cm,align=c]{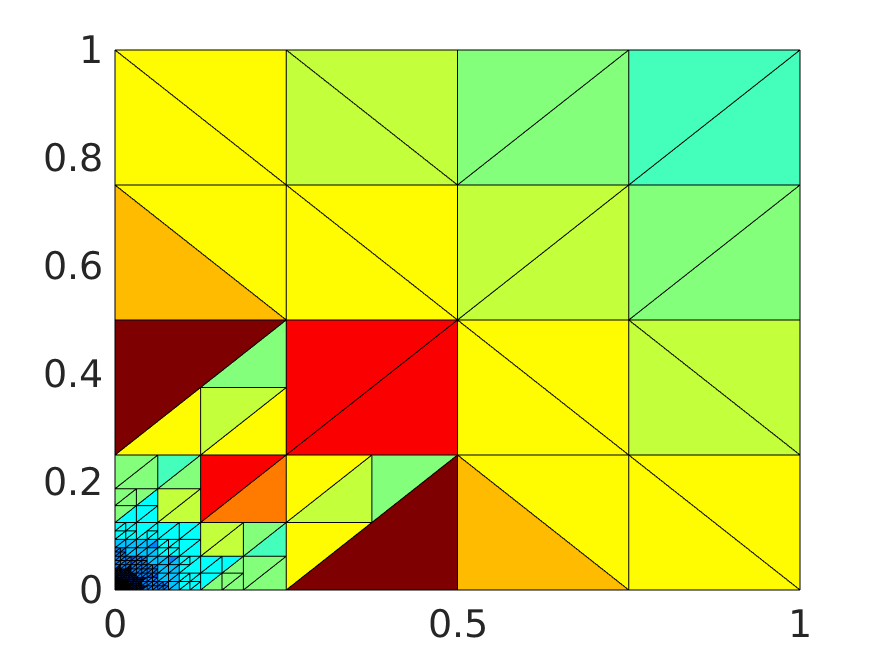}&
\includegraphics[trim=160 0 90 0,clip,width=0.7cm,align=c]{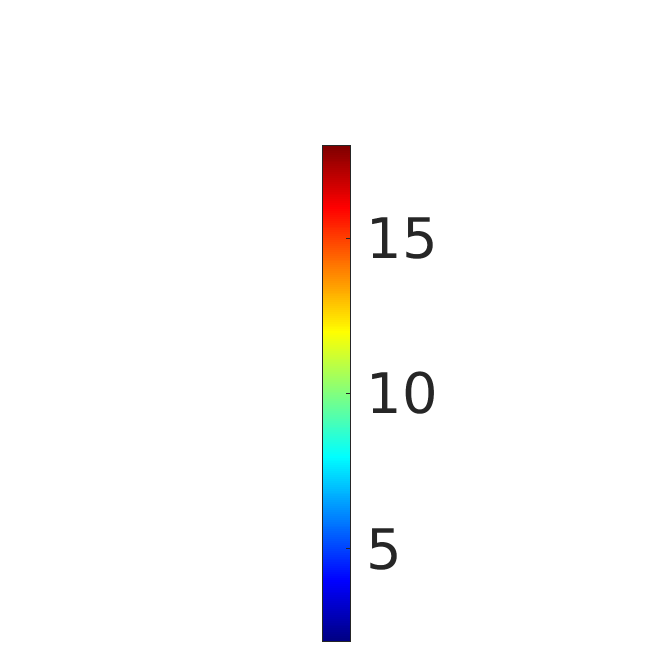}\\
\rotatebox[origin=c]{90}{$\abs{\sol-\solAppr}$}&
\includegraphics[trim=20 15 35 10,clip,width=5cm,align=c]{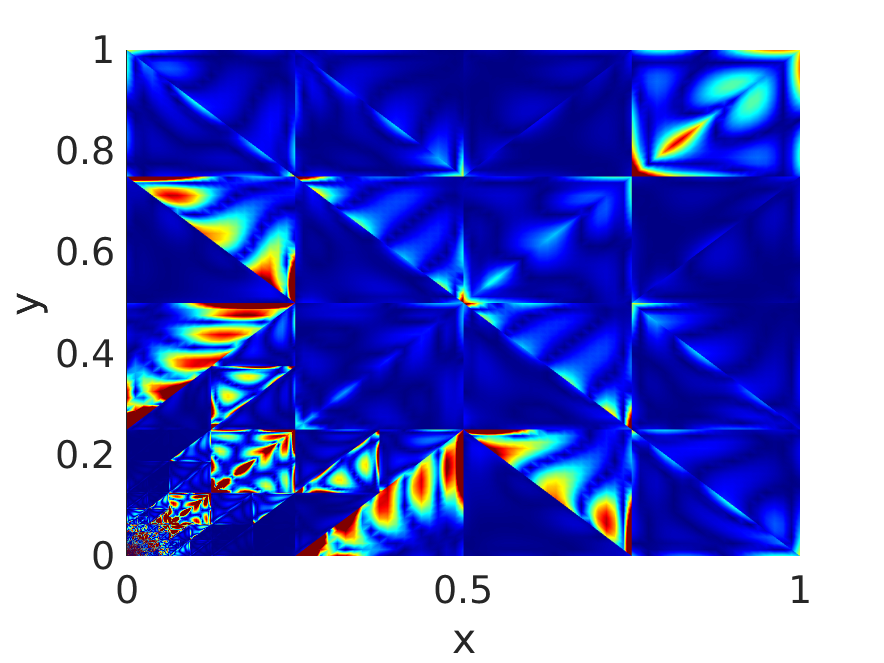}&
\includegraphics[trim=20 15 35 10,clip,width=5cm,align=c]{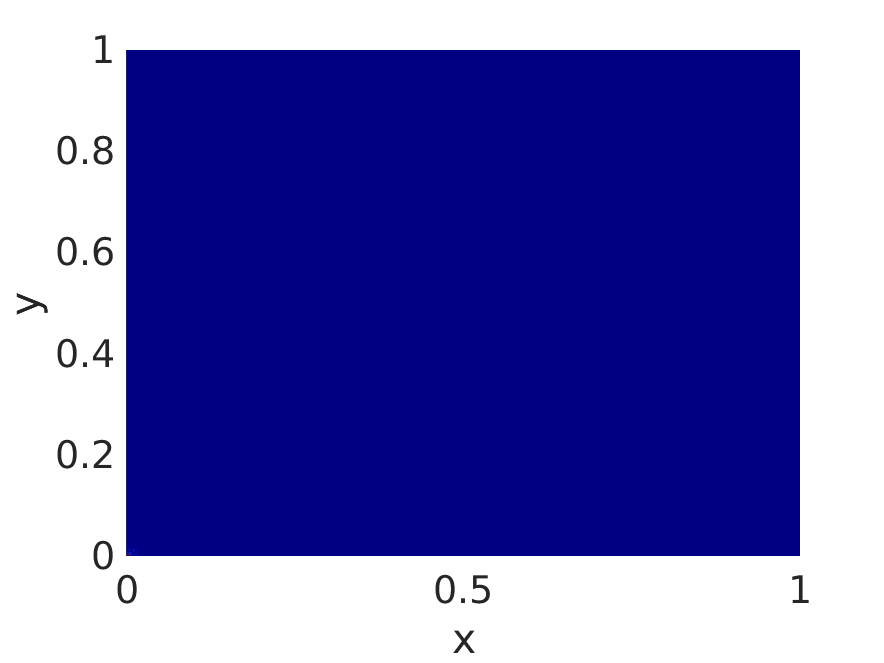}&
\includegraphics[trim=155 0 75 0,clip,width=0.7cm,align=c]{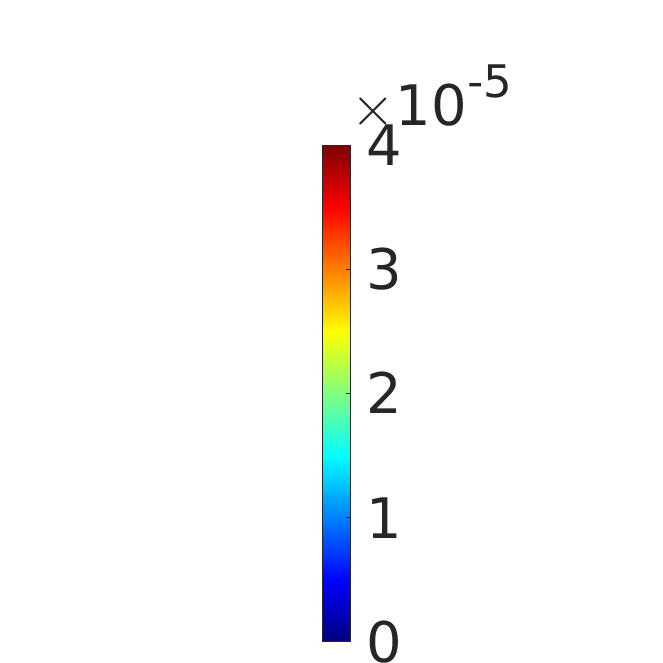}\\
\end{tabular}
    \caption{Numerical results at the final cycle of adaptation for the elliptic problem \itmref{e1} that admits the exact solution stated in \eqnref{e1_sol}. Two different local error indicators/estimations are used to drive the adaptation process with tolerance $\tol=0.01$. The left column uses Doleji's approach \eqnref{local_error_est_doleji} while the right column uses the adjoint approach \eqnref{local_error_est_adjoint} with regard to the output functional defined in \eqnref{e1_output}. Surface and contour plots of the numerical solution are present in the first two rows, the mesh configuration along with the arrangement of the degree of approximation $\polyElem$ is presented in the third row, and the absolute error is presented in the fourth row.}
    \figlab{e1_result}
\end{figure*}

\begin{figure*}
  \centering
\begin{tabular}{cccc}
& Doleji's approach & Adjoint approach\\
\rotatebox[origin=c]{90}{$\solAppr$ (Surface)} &
\includegraphics[trim=15 0 35 10,clip,width=5cm,align=c]{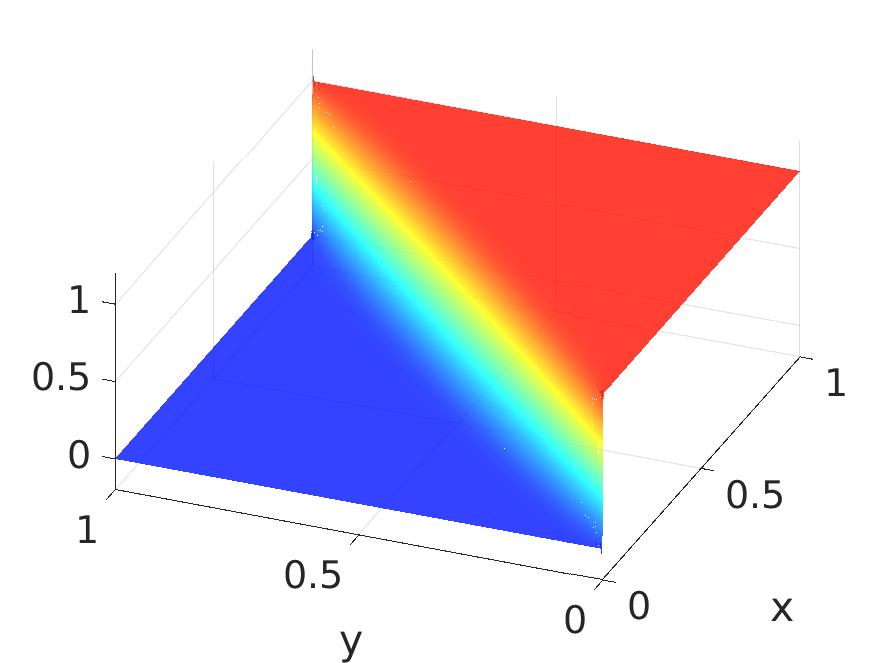}&
\includegraphics[trim=15 0 35 10,clip,width=5cm,align=c]{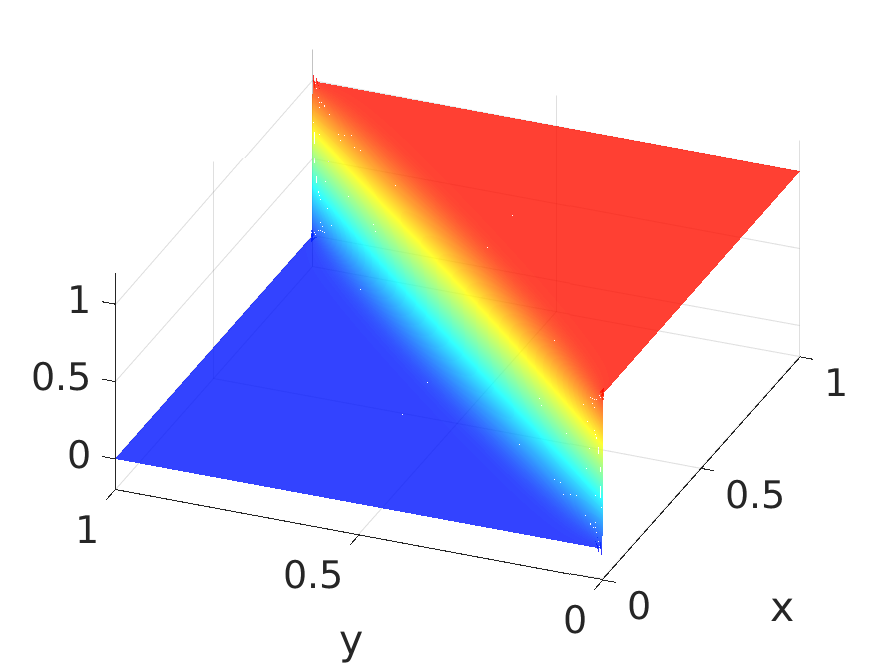}&
\includegraphics[trim=160 0 90 0,clip,width=0.7cm,align=c]{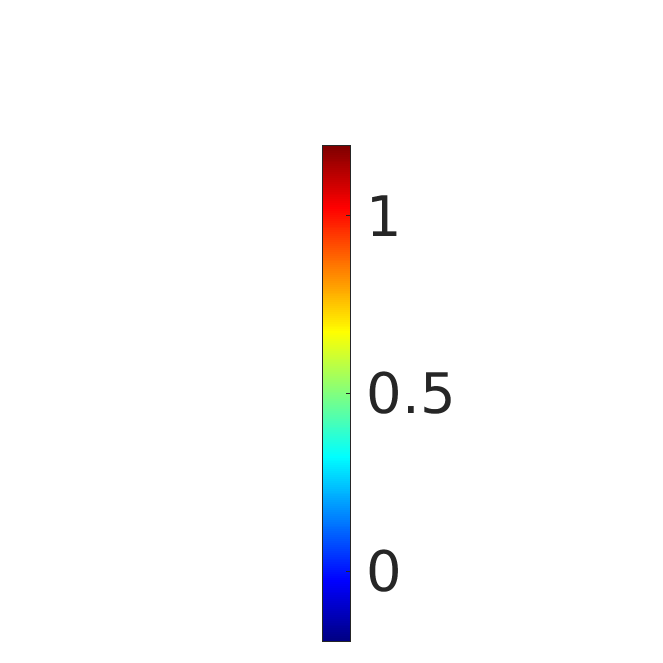}\\
\rotatebox[origin=c]{90}{$\solAppr$ (Contour)}&
\includegraphics[trim=23 15 35 10,clip,width=5cm,align=c]{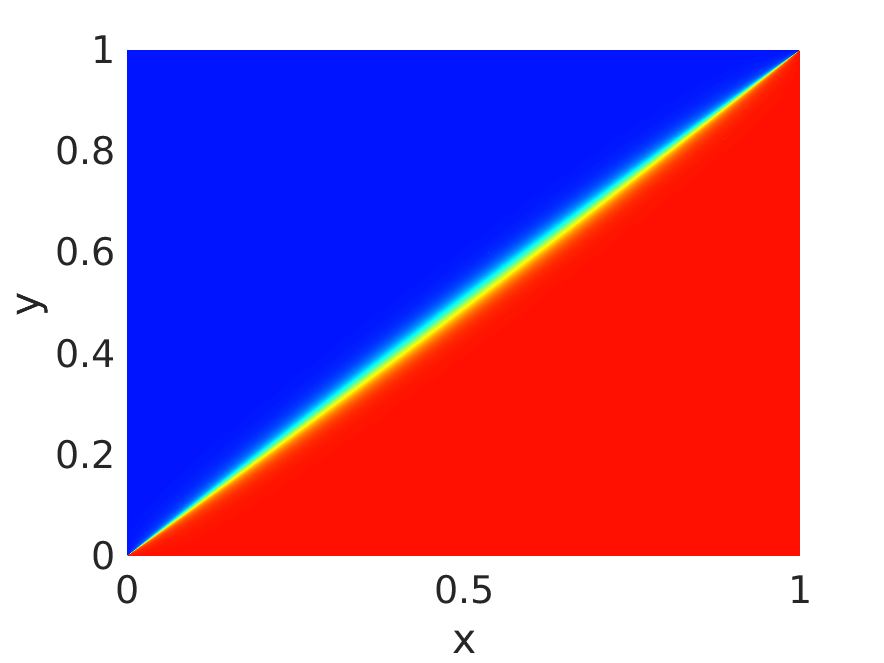}&
\includegraphics[trim=23 15 35 10,clip,width=5cm,align=c]{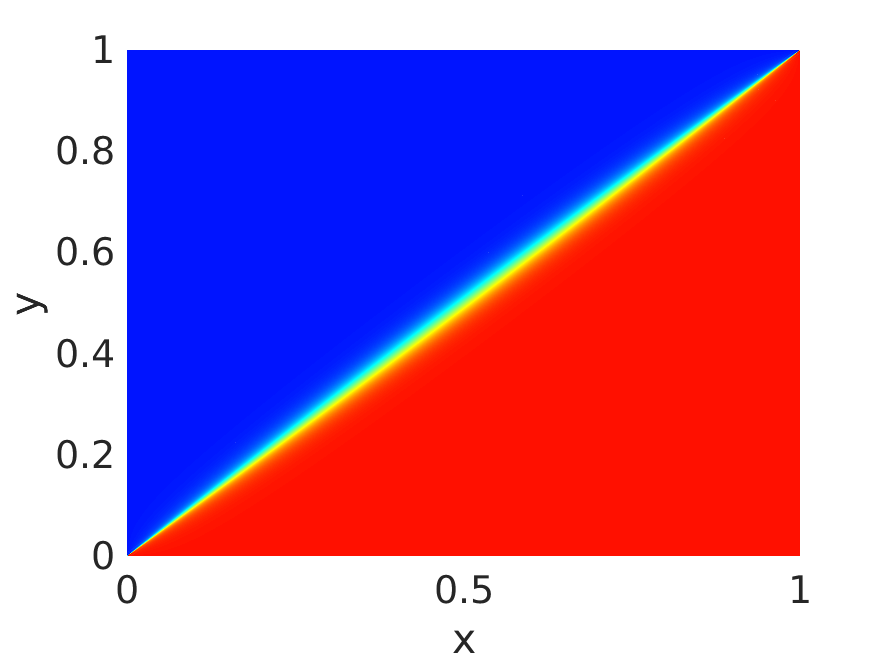}&
\includegraphics[trim=160 0 90 0,clip,width=0.7cm,align=c]{Figures2/Numerical Results/E2/u_colobar_aniso_sig.png}\\
\rotatebox[origin=c]{90}{$\polyElem$ map}&
\includegraphics[trim=15 5 35 10,clip,width=5cm,align=c]{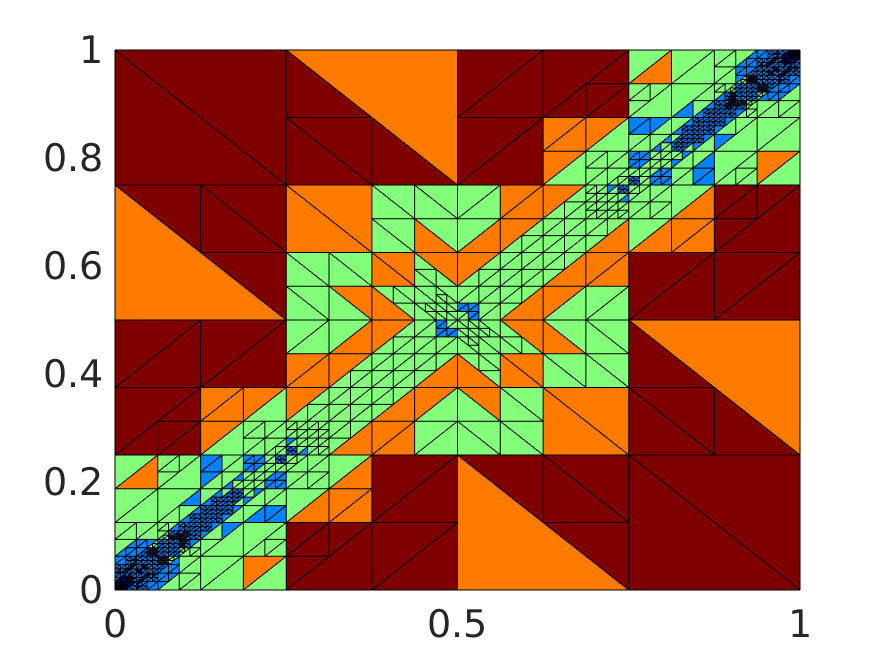}&
\includegraphics[trim=15 5 35 10,clip,width=5cm,align=c]{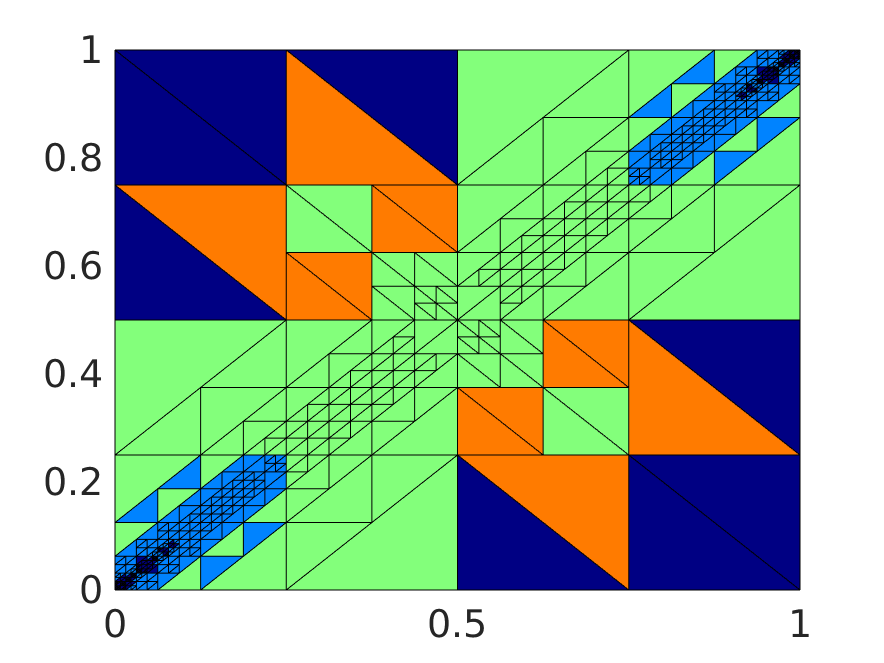}&
\includegraphics[trim=160 0 90 0,clip,width=0.7cm,align=c]{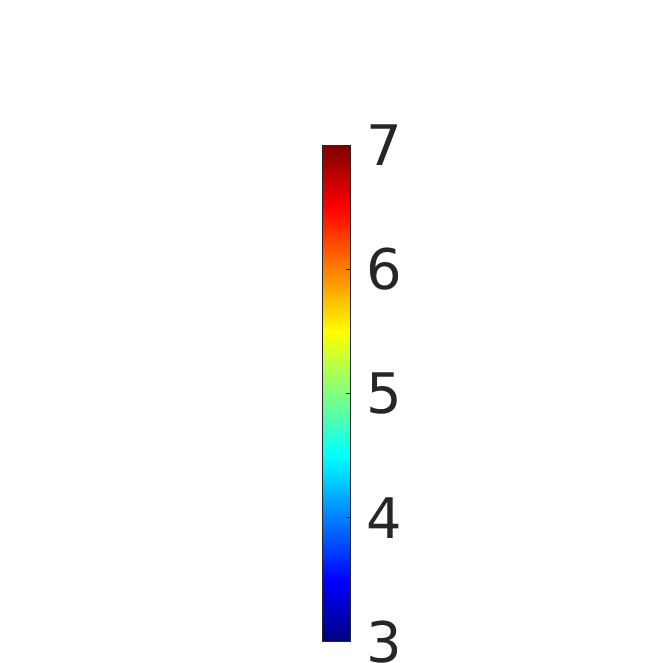}\\
\rotatebox[origin=c]{90}{$\abs{\sol-\solAppr}$}&
\includegraphics[trim=20 15 35 10,clip,width=5cm,align=c]{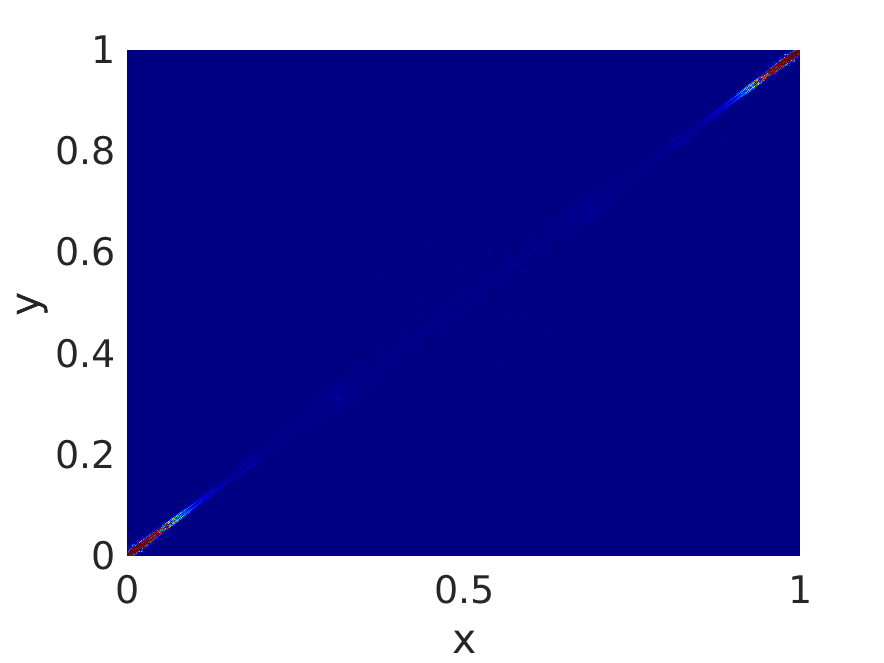}&
\includegraphics[trim=20 15 35 10,clip,width=5cm,align=c]{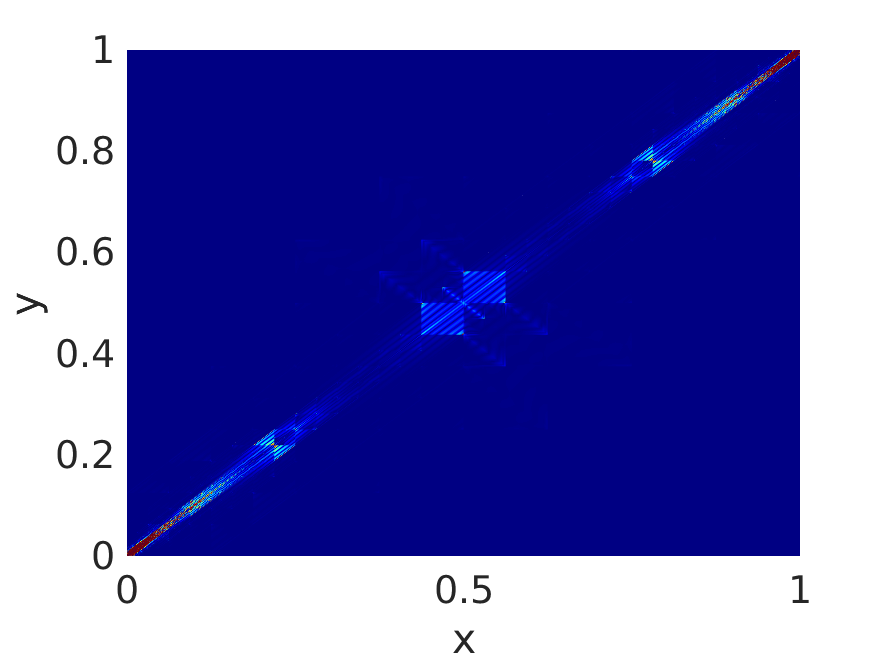}&
\includegraphics[trim=155 0 75 0,clip,width=0.7cm,align=c]{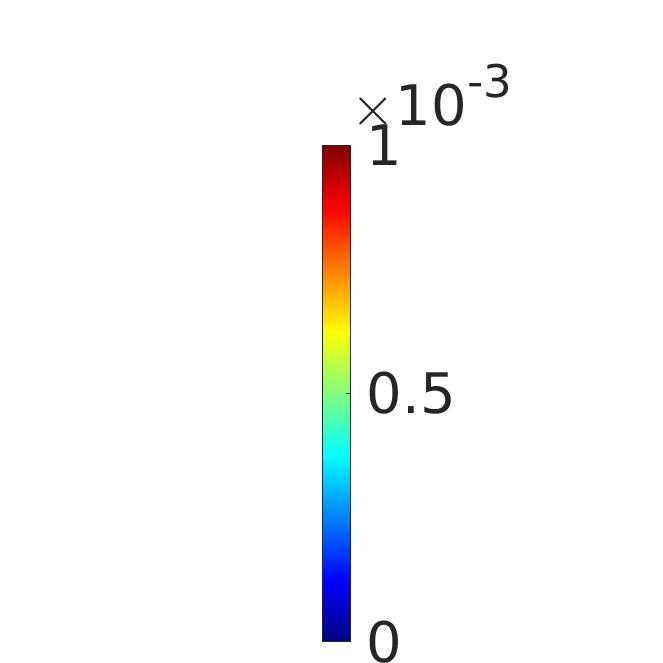}\\
\end{tabular}
    \caption{Numerical results at the adaptation cycle where $\Lsp{2}$-error norm of $\solAppr$ are similar, $O(10^{-3})$, for both approaches by which the adaptation process is driven with tolerance $\tol=0.01$. The results are obtained by solving the elliptic problem \itmref{e2} where anisotropic ratio $\anisoRatio=1000$ and the semi-analytic solution can be obtained with the aid of accurate mappings. The left column uses Doleji's approach \eqnref{local_error_est_doleji} while the right column uses the adjoint approach \eqnref{local_error_est_adjoint} with regard to the output functional defined in \eqnref{e2_output}. Surface and contour plots of the numerical solution are present in the first two rows, the mesh configuration along with the arrangement of the degree of approximation $\polyElem$ is presented in the third row, and the absolute error is presented in the fourth row.}
    \figlab{e2_result}
\end{figure*}

\begin{figure*}
  \centering
\begin{tabular}{cccc}
& Doleji's approach & Adjoint approach\\
\rotatebox[origin=c]{90}{$\solAppr$ (Surface)} &
\includegraphics[trim=15 0 35 10,clip,width=5cm,align=c]{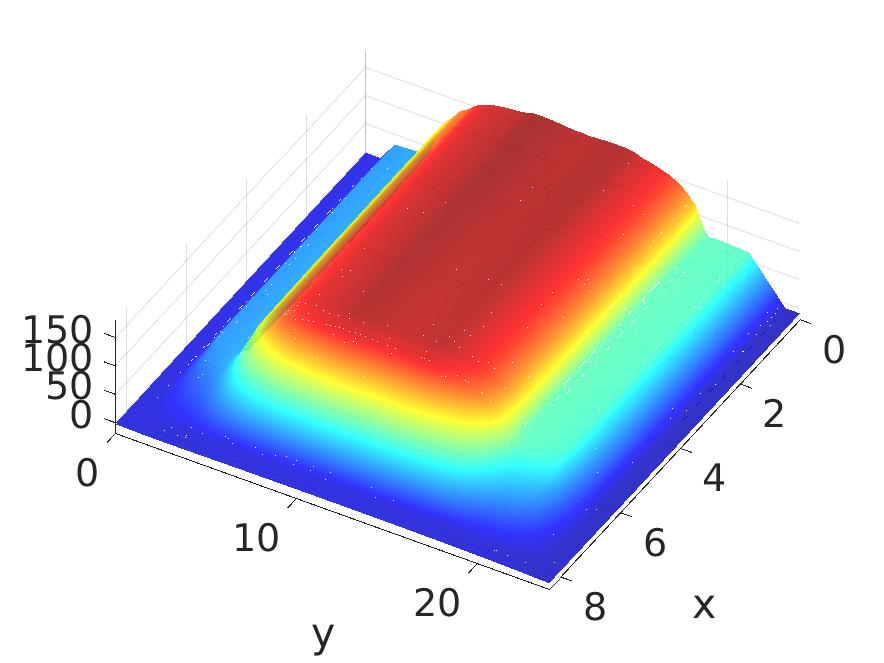}&
\includegraphics[trim=15 0 35 10,clip,width=5cm,align=c]{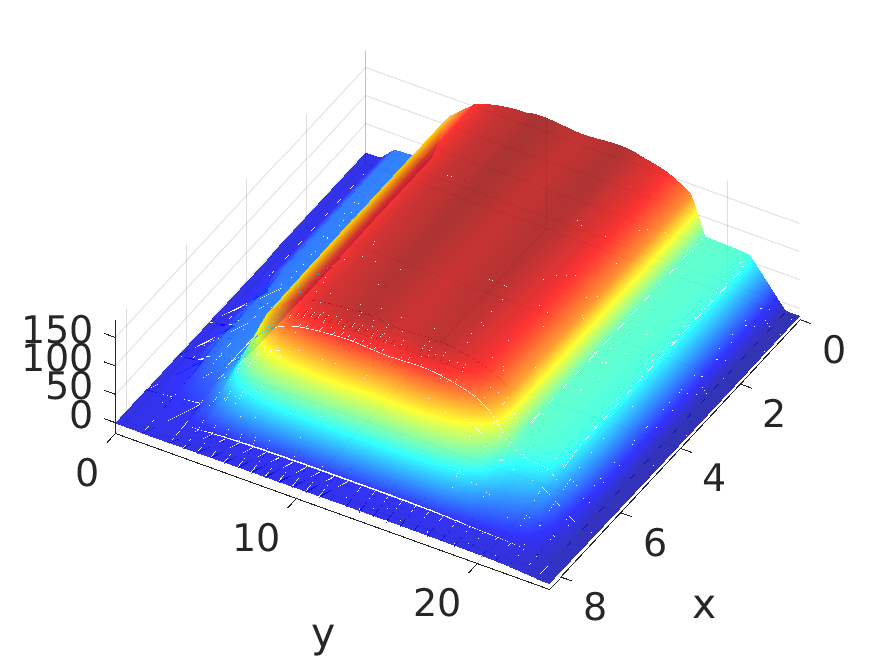}&
\includegraphics[trim=160 0 90 0,clip,width=0.7cm,align=c]{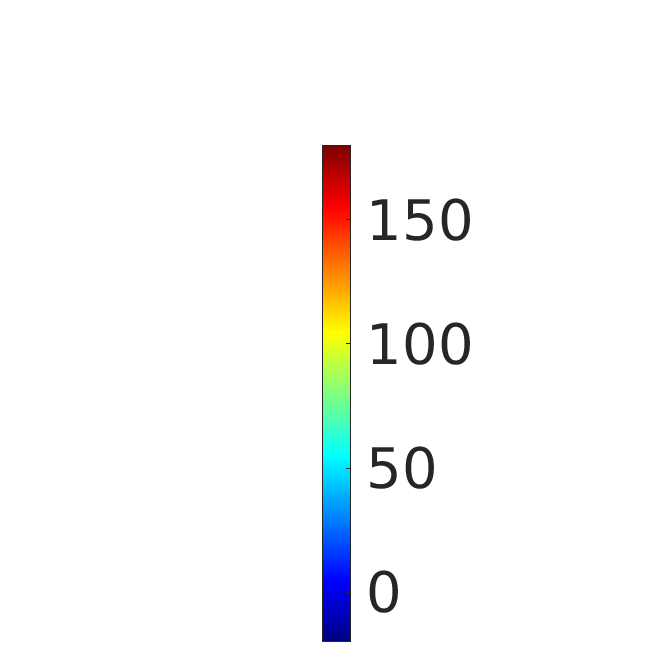}\\
\rotatebox[origin=c]{90}{$\solAppr$ (Contour)}&
\includegraphics[trim=23 15 35 10,clip,width=5cm,align=c]{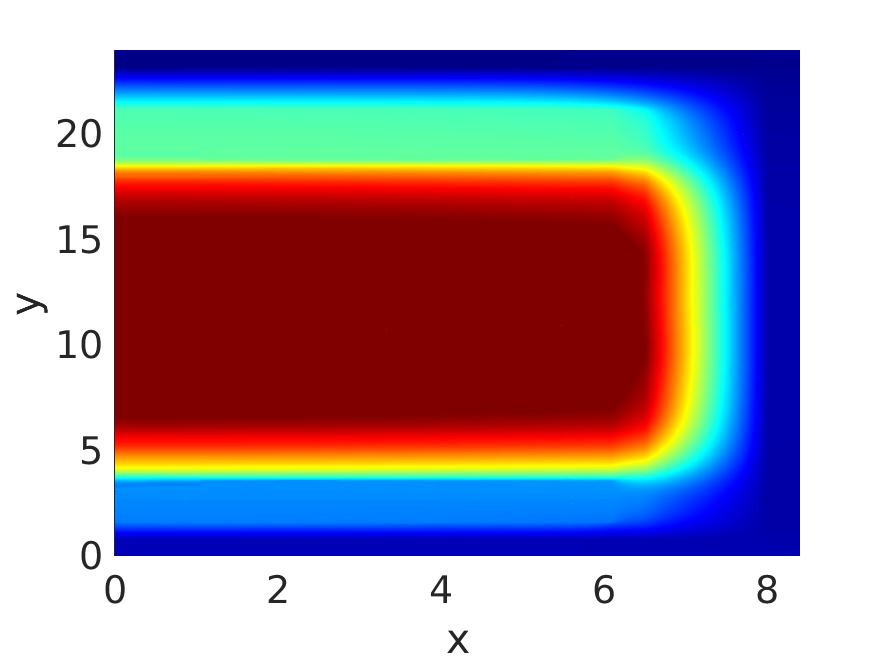}&
\includegraphics[trim=23 15 35 10,clip,width=5cm,align=c]{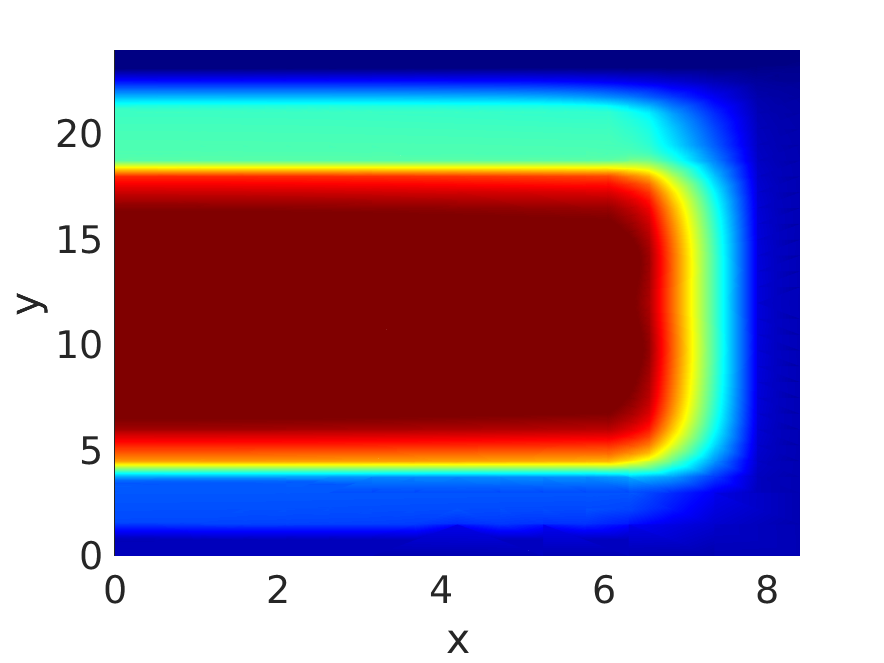}&
\includegraphics[trim=160 0 90 0,clip,width=0.7cm,align=c]{Figures2/Numerical Results/E3/u_colobar_battery.png}\\
\rotatebox[origin=c]{90}{$\polyElem$ map}&
\includegraphics[trim=15 5 35 10,clip,width=5cm,align=c]{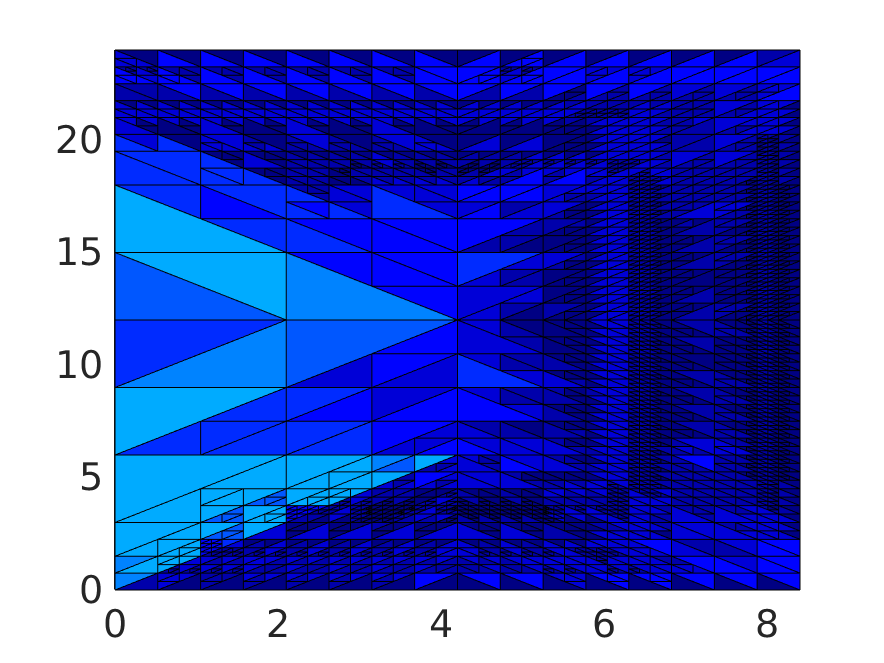}&
\includegraphics[trim=15 5 35 10,clip,width=5cm,align=c]{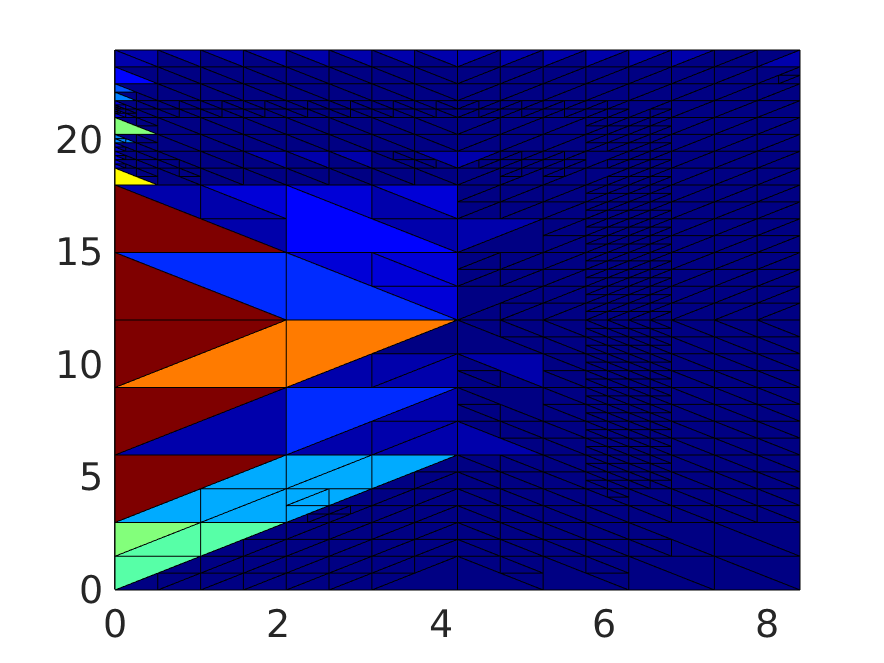}&
\includegraphics[trim=160 0 90 0,clip,width=0.7cm,align=c]{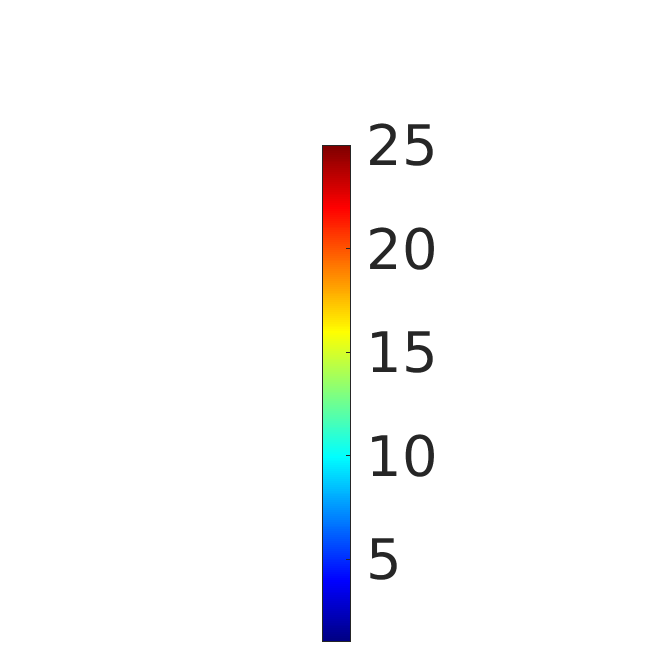}\\
\end{tabular}
    \caption{Numerical results at the final cycle of adaptation for the elliptic problem \itmref{e3}. Two different local error indicators/estimations are used to drive the adaptation process with tolerance $\tol=0.01$. The left column uses Doleji's approach \eqnref{local_error_est_doleji} while the right column uses the adjoint approach \eqnref{local_error_est_adjoint} with regard to the output functional defined in \eqnref{e3_output}. Surface and contour plots of the numerical solution are present in the first two-row, and the mesh configuration along with the arrangement of the degree of approximation $\polyElem$ is presented in the third row.}
    \figlab{e3_result}
\end{figure*}

\begin{figure*}
  \centering
\begin{tabular}{cccc}
 & Doleji's approach & Adjoint approach\\
\rotatebox[origin=c]{90}{$\solAppr$ (Surface)} &
\includegraphics[trim=15 3 35 10,clip,width=5cm,align=c]{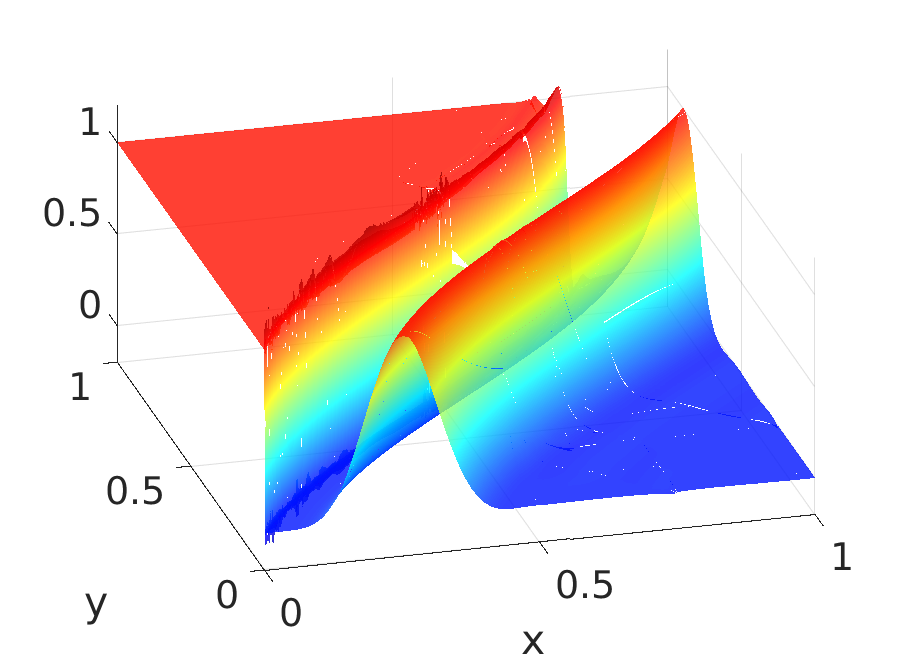}&
\includegraphics[trim=15 3 35 10,clip,width=5cm,align=c]{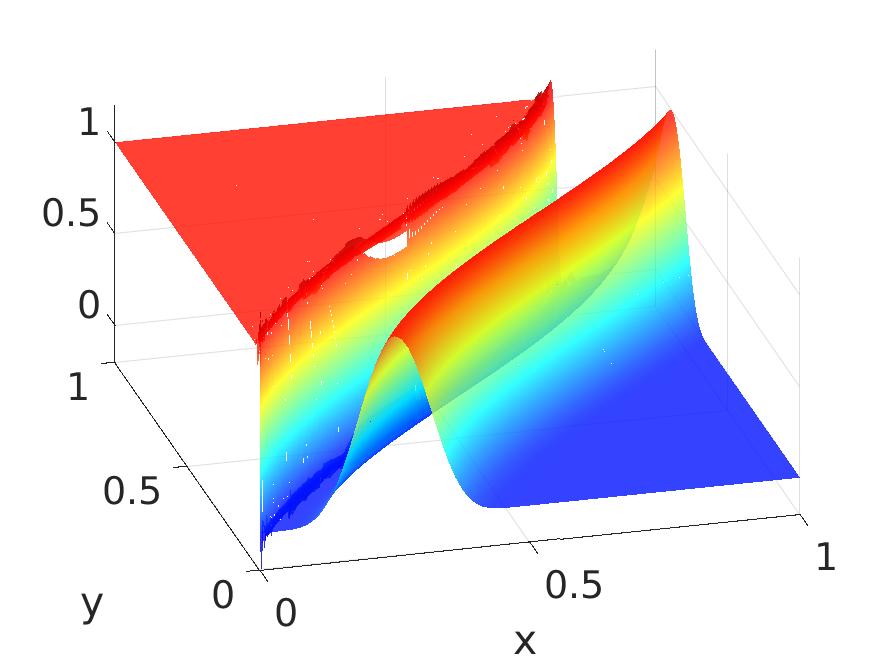}&
\includegraphics[trim=160 0 90 0,clip,width=0.7cm,align=c]{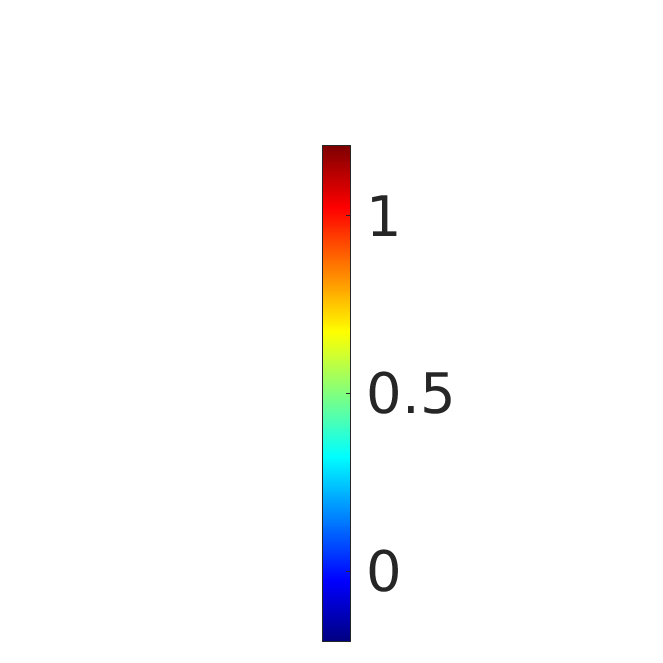}\\
\rotatebox[origin=c]{90}{$\solAppr$ (Contour)} &
\includegraphics[trim=15 20 35 10,clip,width=5cm,align=c]{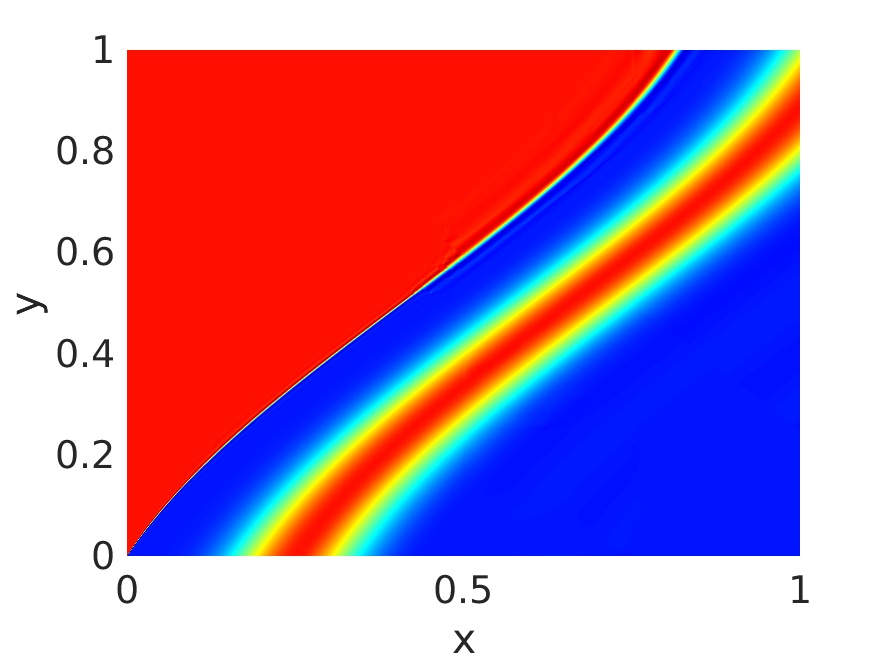}&
\includegraphics[trim=15 20 35 10,clip,width=5cm,align=c]{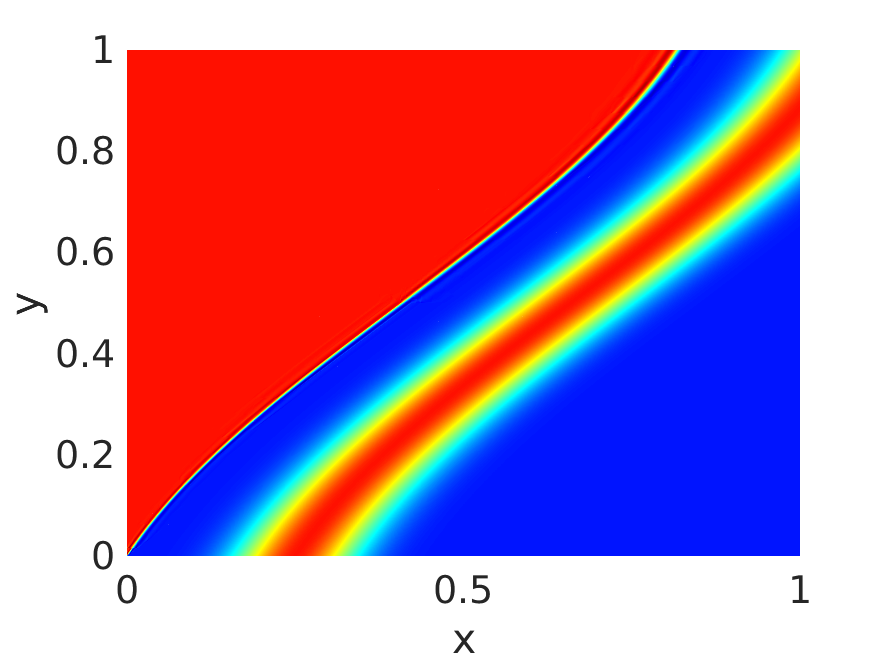}&
\includegraphics[trim=160 0 90 0,clip,width=0.7cm,align=c]{Figures2/Numerical Results/HP1/u_colobar_lin_adv.png}\\
\rotatebox[origin=c]{90}{$\polyElem$ map}&
\includegraphics[trim=10 3 35 10,clip,width=5cm,align=c]{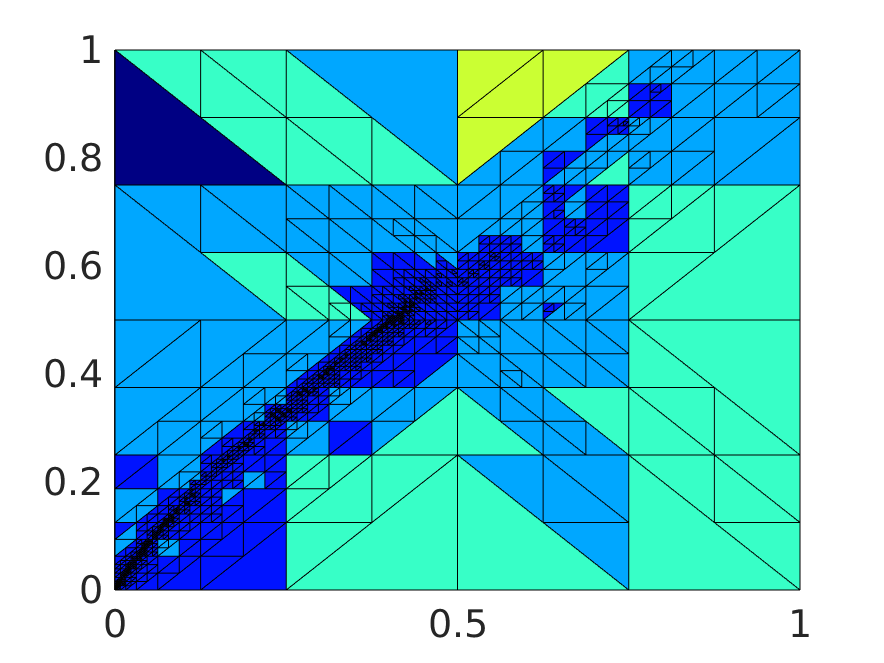}&
\includegraphics[trim=10 3 35 10,clip,width=5cm,align=c]{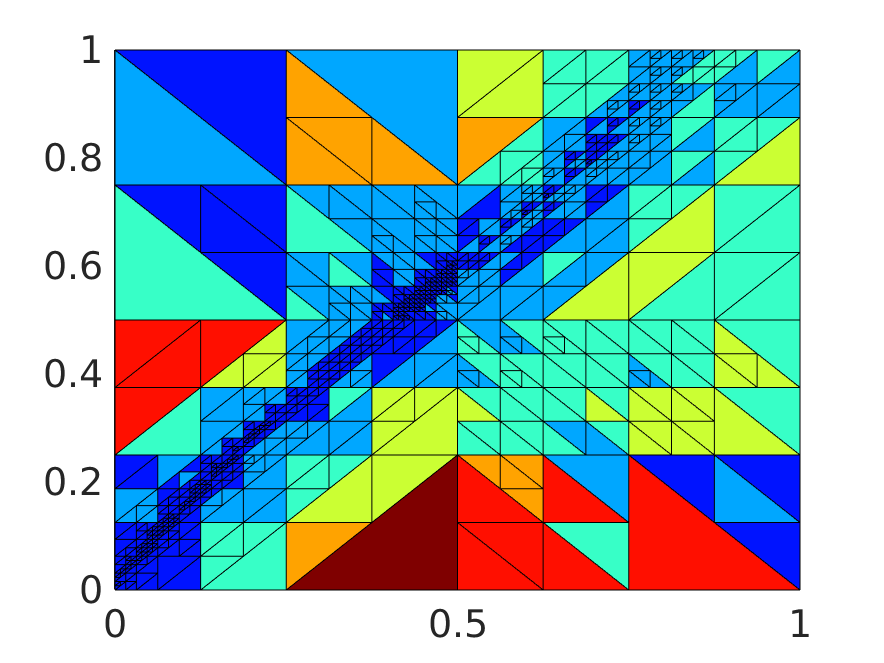}&
\includegraphics[trim=161 0 112 0,clip,width=0.5cm,align=c]{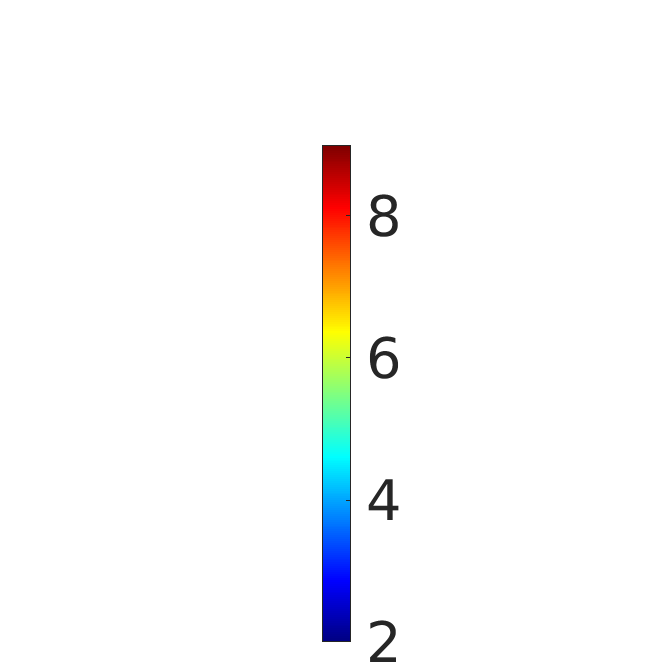}\\
\rotatebox[origin=c]{90}{$\abs{\sol-\solAppr}$}&
\includegraphics[trim=15 15 35 10,clip,width=5cm,align=c]{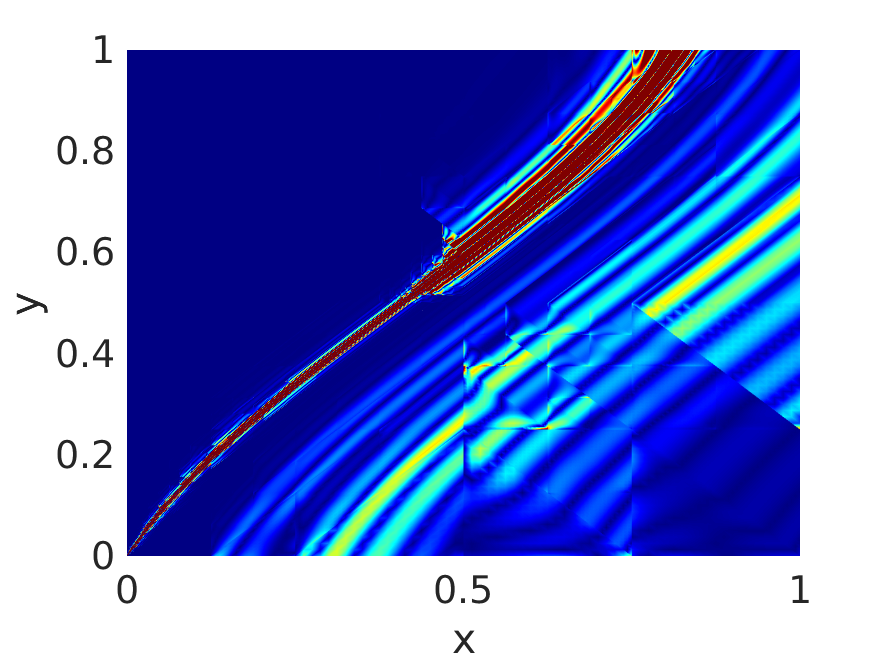}&
\includegraphics[trim=15 15 35 10,clip,width=5cm,align=c]{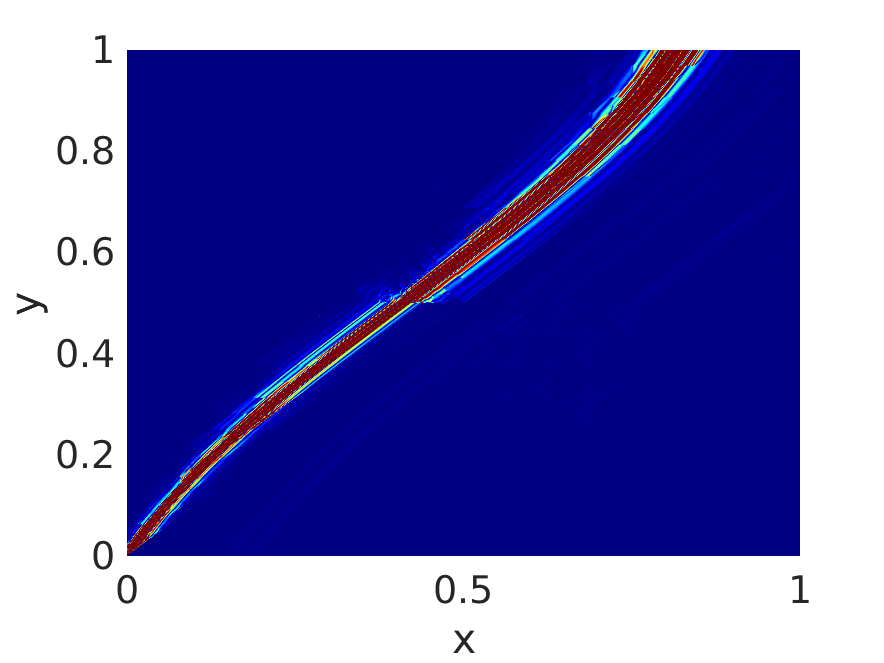}&
\includegraphics[trim=160 0 85 0,clip,width=0.8cm,align=c]{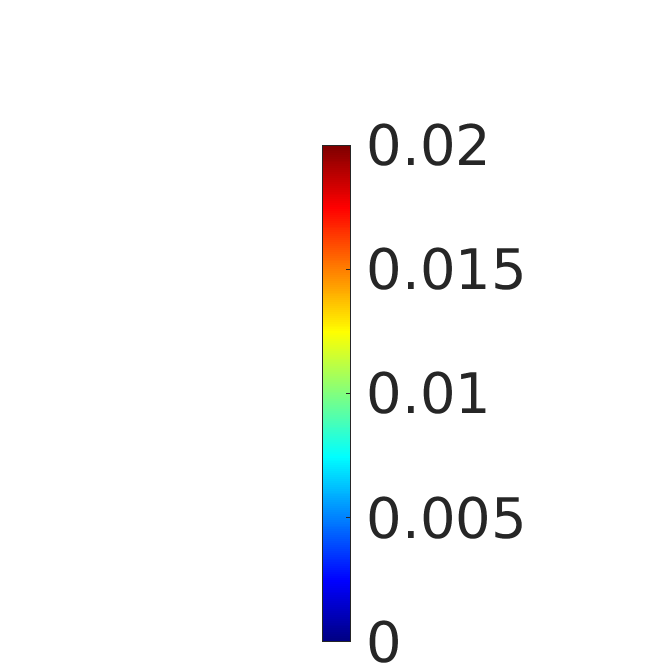}\\
\end{tabular}
    \caption{Numerical results at the final cycle of adaptation for the hyperbolic problem \itmref{hp1} where the exact solution can be found by the method of characteristics. Two different local error indicators/estimations are used to drive the adaptation process with tolerance $\tol=0.05$. The left column uses Doleji's approach \eqnref{local_error_est_doleji} while the right column uses the adjoint approach \eqnref{local_error_est_adjoint} with regard to the output functional defined in \eqnref{hp1_output}. Surface and contour plots of the numerical solution are present in the first two rows, the mesh configuration along with the arrangement of the degree of approximation $\polyElem$ is presented in the third row, and the absolute error is presented in the fourth row.}
    \figlab{hp1_result}
\end{figure*}

\begin{figure*}
  \centering
\begin{tabular}{cccc}
& Doleji's approach & Adjoint approach\\
\rotatebox[origin=c]{90}{$\solAppr$ (Surface)} &
\includegraphics[trim=15 3 35 10,clip,width=5cm,align=c]{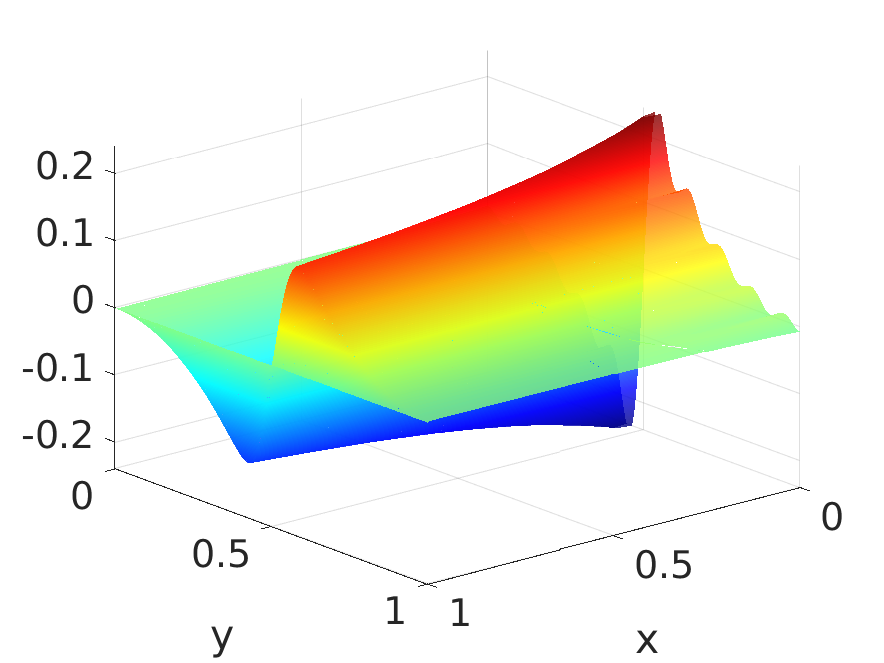}&
\includegraphics[trim=15 3 35 10,clip,width=5cm,align=c]{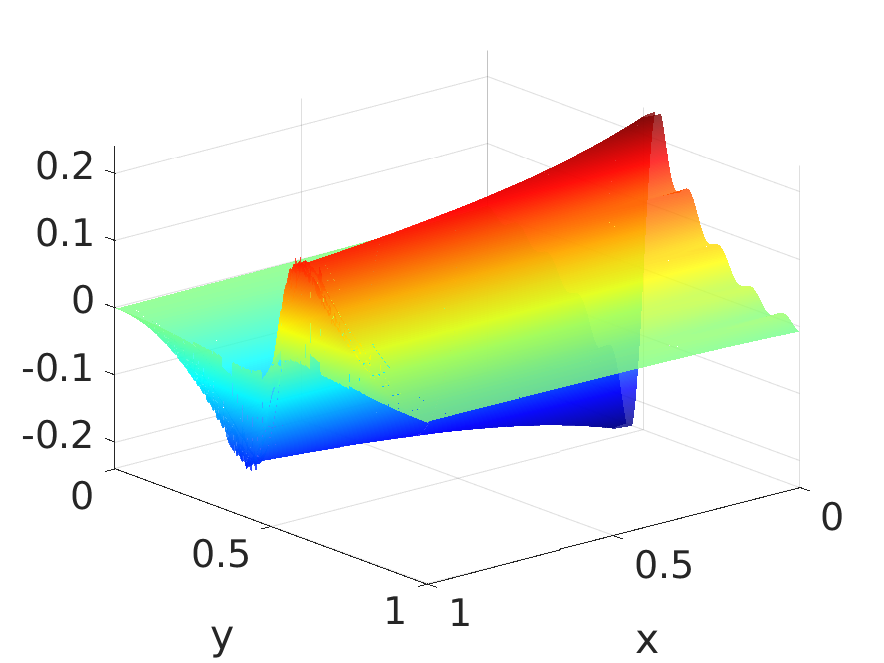}&
\includegraphics[trim=160 0 90 0,clip,width=0.7cm,align=c]{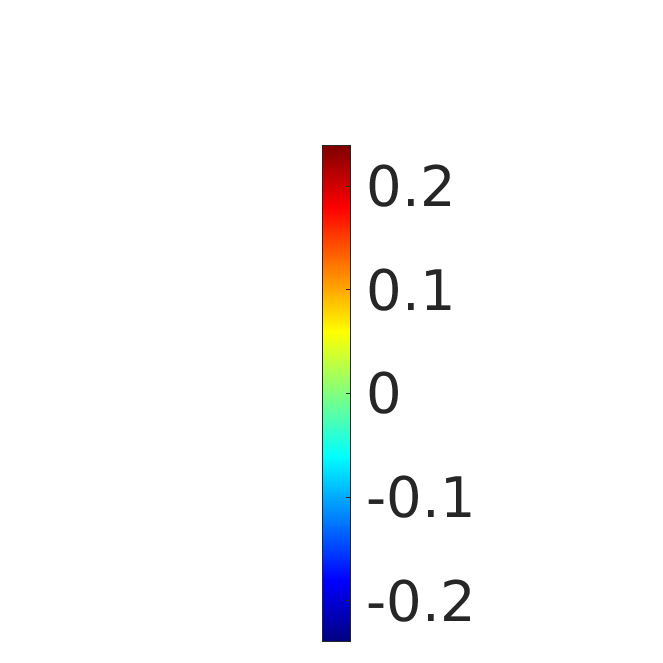}\\
\rotatebox[origin=c]{90}{$\solAppr$ (Contour)} &
\includegraphics[trim=5 3 35 10,clip,width=5cm,align=c]{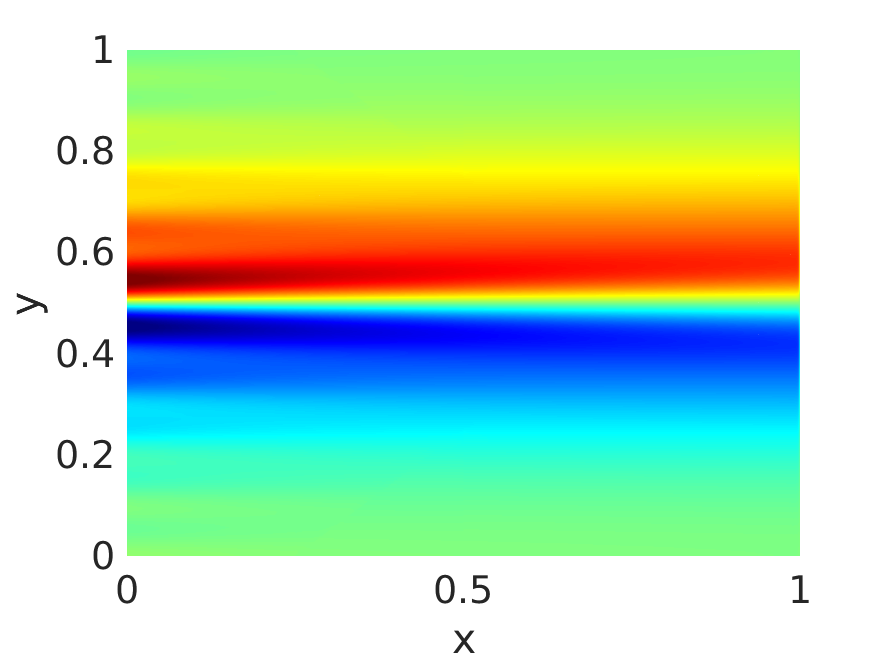}&
\includegraphics[trim=5 3 35 10,clip,width=5cm,align=c]{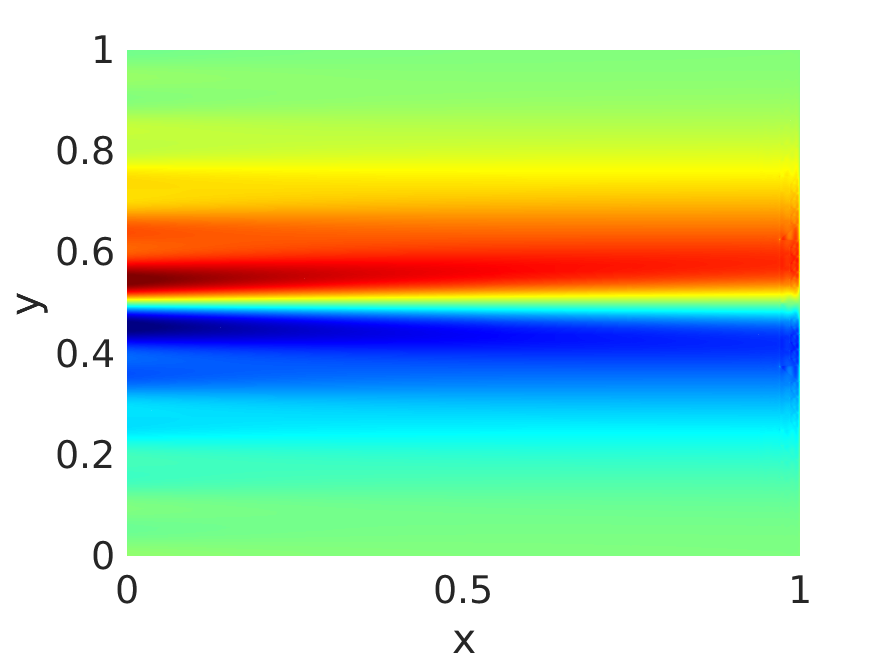}&
\includegraphics[trim=160 0 90 0,clip,width=0.7cm,align=c]{Figures2/Numerical Results/HB1/u_colobar_conv_diff.png}\\
\rotatebox[origin=c]{90}{$\polyElem$ map}&
\includegraphics[trim=10 3 35 10,clip,width=5cm,align=c]{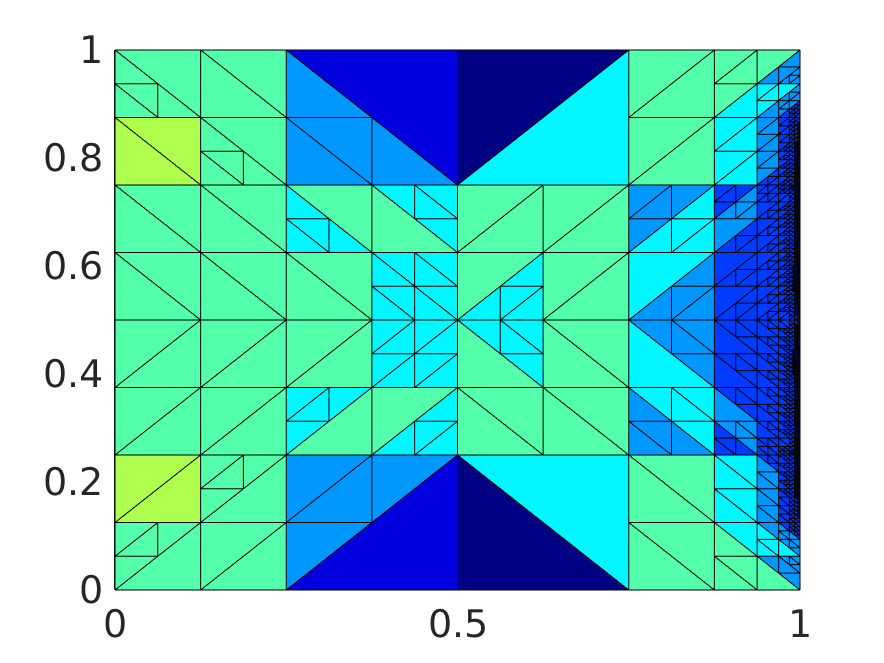}&
\includegraphics[trim=10 3 35 10,clip,width=5cm,align=c]{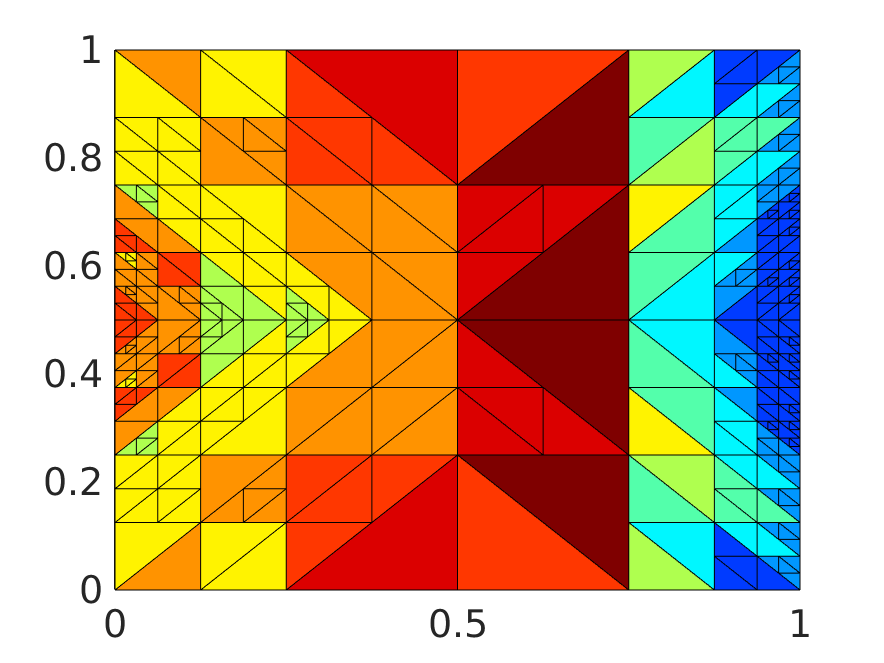}&
\includegraphics[trim=160 0 90 0,clip,width=0.7cm,align=c]{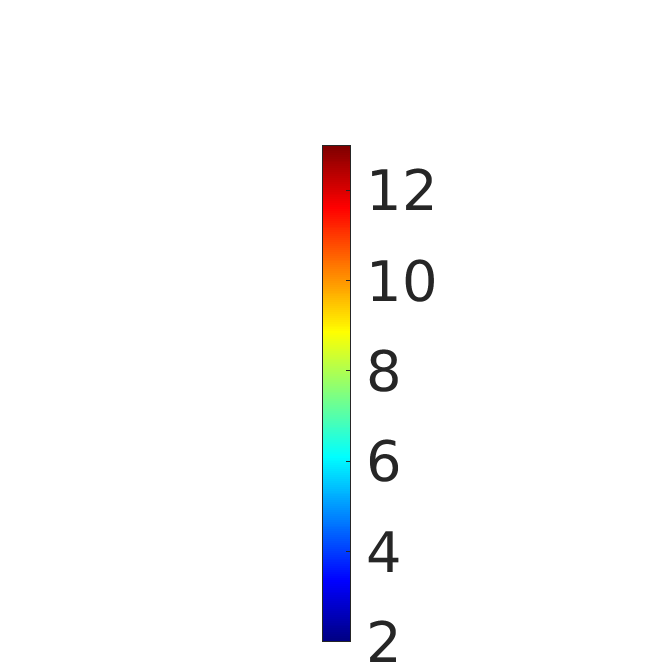}\\
\rotatebox[origin=c]{90}{$\abs{\sol-\solAppr}$}&
\includegraphics[trim=23 16 35 10,clip,width=5cm,align=c]{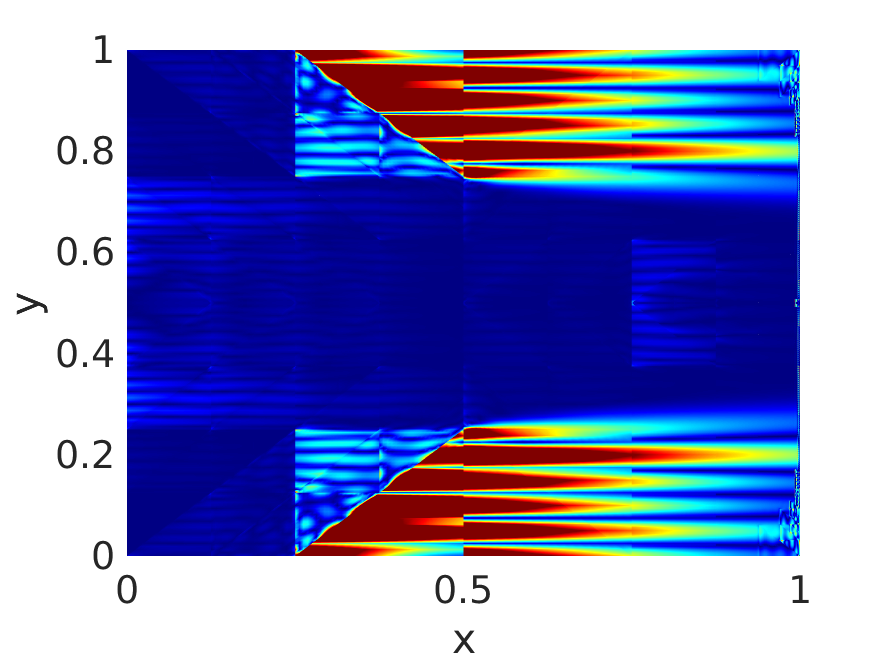}&
\includegraphics[trim=23 16 35 10,clip,width=5cm,align=c]{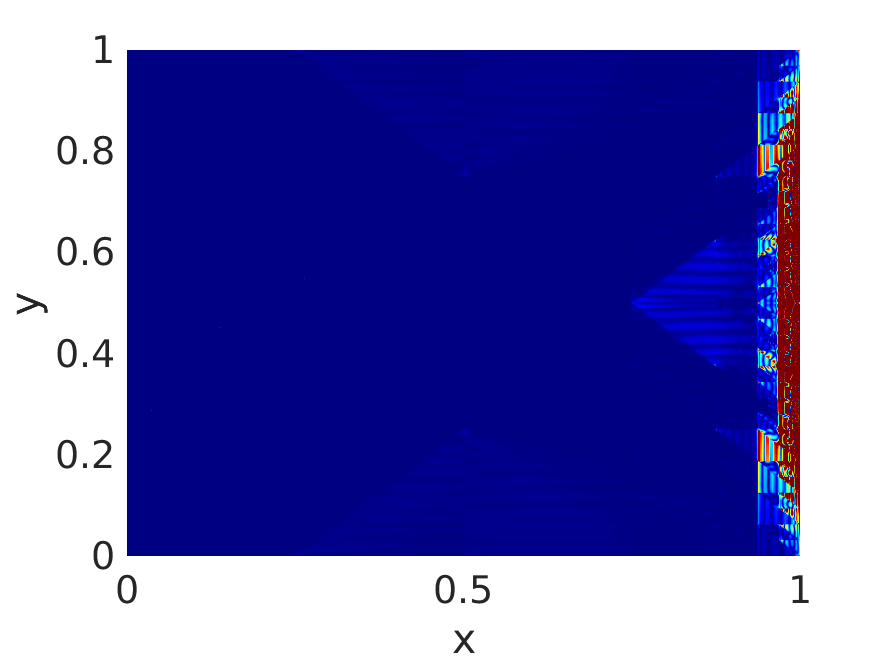}&
\includegraphics[trim=160 0 75 0,clip,width=0.7cm,align=c]{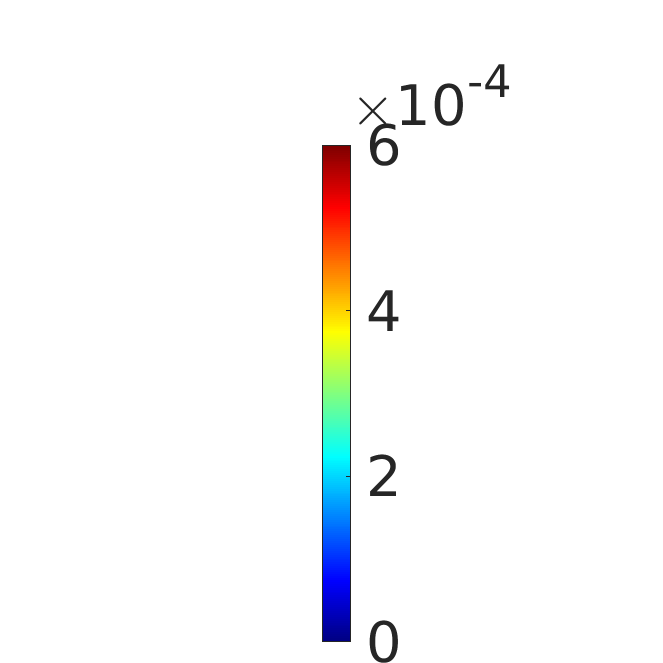}\\
\end{tabular}
    \caption{Numerical results at the final cycle of adaptation for the mixed problem \itmref{hb1} that admits the exact solution stated in \eqnref{hb1_sol}. Two different local error indicators/estimations are used to drive the adaptation process with tolerance $\tol=0.05$. The left column uses Doleji's approach \eqnref{local_error_est_doleji} while the right column uses the adjoint approach \eqnref{local_error_est_adjoint} with regard to the output functional defined in \eqnref{hb1_output}. Surface and contour plots of the numerical solution are present in the first two rows, the mesh configuration along with the arrangement of the degree of approximation $\polyElem$ is presented in the third row, and the absolute error is presented in the fourth row.}
    \figlab{hb1_result}
\end{figure*}
\FloatBarrier

\section{Conclusion} \seclab{conclusion}
In this work, we proposed unified $hp$-HDG frameworks for  Friedrichs' system that embraces a large classes of PDEs. The key ingredient in dealing with $hp$-nonconforming interfaces is to take advantage of natural built-in mortars in HDG methods. By choosing this  mortar type and the trace unknown approximation space, an upwind-based numerical flux can be naturally derived and constructed over these mortars. For the one-field system, we have a general form of the numerical flux. For the two-field system, we exploit the system's intrinsic structure to obtain the numerical flux with the elimination of one of the two trace unknowns. Although the existence of such a numerical flux can only be guaranteed by introducing a few assumptions, this explicit reduced form of the numerical flux is key to proving the well-posedness. All in all, combing all pieces of components together, the resulting $hp$-HDG formulations are parameter-free and are proved to be well-posed for both one-field and two-field Friedrichs' system. The proofs are built upon the assumptions outlined in the Friedrichs' systems and upon a few additional assumptions. These assumptions are verified for all numerical examples considered in the paper. As we have seen,  Friedrichs' systems allow us to systematically construct and analyze $hp$-HDG schemes for one-field and two-field systems. The unification opens an opportunity for us to develop a single universal code to solve various kinds of PDEs.

Besides the analysis, several numerical experiments are also carried out. In the experiments, three distinct types of PDEs are considered: elliptic, hyperbolic, and mixed-type. We showed that these PDEs fall into our framework and their $hp$-HDG formulations are hence well-posed. In addition, to verify the effectiveness of $hp$-adaptation, we also suggest a simple algorithm along with two different approaches for driving the adaptivity: Doleji's approach, and the adjoint approach. Their differences lie in the error indicators. The first one relies on the estimation of the smoothness of the approximate solution with regard to the PDE to be solved. The second one relies on the estimation of the accuracy of a chosen output functional. In terms of performance, both approaches show improvement in the convergence rates for almost all the examples through the adaptation process if an appropriate tolerance is applied. The numerically polluted area induced by singularity, discontinuity, or sharp gradient can decrease as the adaption proceeds. In summary, the performances of these two approaches are comparable though sometimes one is better than the other. 
We also found that both approaches are quite sensitive to the user-defined tolerance value. Specifically speaking, the values have to be small enough to handle the problems whose solutions possess extremely different characteristics in the same domain

Finally, we add a couple of  remarks. In this work, we only considered one-field and two-field Friedrichs' systems. Another significant structure is the three-field, which includes the PDE governed by the linearized incompressible flow. However, such a system requires a significant amount of dedicated discussion and deserves another paper to cover it. Thus, it is left to our future work. The other remark is about time-dependent PDEs. Although we only consider steady-state PDEs in this paper, it is not hard to extend our current work to time-dependent models with the aid of the method of Rothe \cite{rothe1930} where the time-derivative term is first discretized by some single-step time scheme. The discretized terms can be treated as reaction and forcing terms, and the resulting semi-discrete PDE can then be re-written as a general conservation form stated in \eqnref{general_PDE}. Thus, the analysis presented in this study is applicable. However, the algorithm of $hp$-adaption may need to be re-designed because different time steps can cause the solution to behave differently. Thus, $h$-coarsening operations may be required in response to this change. Besides, a proper transfer of the solution between each adaptation needs to be carefully addressed. 

\section*{Acknowledgments}
Thanks to Geonyoung Lee for the fruitful discussion on the $hp$-adaptivity for elliptic equations. This research is partially funded by the National Science Foundation awards NSF-OAC-2212442, NSF-2108320, NSF-1808576 and NSF-CAREER-1845799; by the Department of Energy award DE-SC0018147 and DE-SC0022211.

\appendix
\renewcommand{\theequation}{a.\arabic{equation}}
\section{Adjoint \texorpdfstring{$hp$}{hp}-HDG formulation}\seclab{adjoint_hdg}
Although we use discrete adjoint approach in this work, the derivation of adjoint problem outline in \cite{becker2001} can still be followed. Intead of treating the governing equation \eqnref{general_PDE} as a primal problem, here we consider $hp$-HDG formulation that is stated in \eqnref{Fds_one_hp_hdg} or \eqnref{Fds_two_hp_hdg} as a primal problem. The procedure of derivation of the adjoint problem is briefly outlined below:  
\ben
    \item Re-state the primal problem, either \eqnref{Fds_one_hp_hdg} or \eqnref{Fds_two_hp_hdg}, as a bilinear form.
    \item Define an output functional $\Output{\cdot}$ of the approximate solution $\FdsVarLump$ and use it as an objective function of an optimization problem along with the constraint posed by the bilinear formulation given in step one.  
    \item Re-write the constrained optimization problem constructed in the previous step as the unconstrained one via the Lagrangian approach in which the test function $\FdsTestVarLump$ presented in the bilinear formulation now becomes a Lagrange multiplier.
    \item Solve this trivial optimization problem by taking advantage of the first optimality condition where the adjoint $hp$-HDG formulation can then be derived. 
\een
In the first way, the derivation of the adjoint PDE can also be done by following a similar procedure. However, it should be noted that the adjoint $hp$-HDG formulations derived by the first and second way are not necessarily the same and in fact they are different in this paper. To express an adjoint $hp$-HDG formulation, we have to further introduce the notation of functional derivative $\frac{\delta\OutFcnl\LRs{X}}{\delta Y}$ which is the differential (or variation or first variation) of the functional $\OutFcnl\LRs{X}$.

The adjoint $hp$-HDG formulation with regard to the one-field Friedrichs' system reads: seek $\LRp{\FdsAdVarAppr,\FdsAdTrcVarAppr}\in\FdsAdVarApprSpc{}\times\FdsAdTrcVarApprSpc{}$ such that:
\begin{subequations}\eqnlab{Fds_one_ad_hdg}\small
\begin{align}
\begin{split}\eqnlab{Fds_one_ad_hdg_local}
    -\FdsSum\LRp{\FdsAdVarAppr,{\FdsPartial}\LRp{{\FdsAk}\FdsAdTestAppr}}_{\domPart} 
    + \LRp{\FdsG^T\FdsAdVarAppr,{\FdsAdTestAppr}}_{\domPart} 
    +\LRa{-{\FdsABnd}\FdsAdTrcVarAppr+{\stabPar}\LRp{{\FdsAdVarAppr}-\FdsAdTrcVarAppr},{\FdsAdTestAppr}}_{\domPartBnd} 
    = -\frac{\delta\OutFcnl\LRs{\FdsVarLump}}{\FdsAdTestAppr}&
\end{split}\\
\begin{split}\eqnlab{Fds_one_ad_hdg_global}
    \LRa{\jump{-{\FdsABnd}\FdsAdTrcVarAppr+{\stabPar}\LRp{{\FdsAdVarAppr}-\FdsAdTrcVarAppr}},\FdsAdTrcTestAppr}_{\skel}
    = 
    -\LRa{\half\LRp{-{\FdsABnd}+\FdsMBnd}^T\FdsAdTrcVarAppr, \FdsAdTrcTestAppr}_{\skelBnd}
    -\frac{\delta\OutFcnl\LRs{\FdsVarLump}}{\FdsAdTrcTestAppr},&
\end{split}
\end{align}
\end{subequations}
for all $\LRp{\FdsAdTestAppr,\FdsAdTrcTestAppr}\in\FdsAdVarApprSpc{}\times\FdsAdTrcVarApprSpc{}$. The stabilization parameter is still set as $\FdsMBnd:=\stabPar$. In order to incorporate the boundary condition, the output functional has to include the following term:
\beq\eqnlab{Fds_one_bc_output}
\Output{\FdsVarLump}=-\LRa{\half\LRp{-\FdsABnd-\FdsMBnd}^T\FdsTrcVarAppr,\bs{\BCVal}}_{\skelBnd},
\eeq
where the function $\bs{\BCVal}:\domBnd\rightarrow\real^{\FdsNVar}$ is defined as
\begin{equation}\eqnlab{Fds_one_ad_Dirichlet}
    \bs{g} := \begin{cases}
    \bs{\DirVal}\quad&\text{if}\,\LRp{\FdsABnd+\FdsMBnd}\neq0,\\
    \bs{0}&\text{if}\,\LRp{\FdsABnd+\FdsMBnd}=0,
    \end{cases}
\end{equation}
where $\bs{\DirVal}$ is the Dirichlet data. Comparing \eqnref{Fds_one_Dirichlet} and \eqnref{Fds_one_ad_Dirichlet}, it can be observed that the inflow and outflow boundaries are switched between the primal formulation \eqnref{Fds_one_hp_hdg} and its's adjoint formulation \eqnref{Fds_one_ad_hdg}. It should be also noted that the adjoint problem \eqnref{Fds_one_ad_hdg} will automatically have homogeneous boundary conditions if the output functional does not include the term specified in \eqnref{Fds_one_bc_output} (i.e. $\DirVal=0$). 

The well-posedness analysis is similar to lemma \lemref{Fds_one_hp_localCon} for the local equation \eqnref{Fds_one_ad_hdg_local} and to theorem \theoref{Fds_one_hp_globalCon} for the adjoint $hp$-HDG formulation \eqnref{Fds_one_ad_hdg}. Thus, we simply outline the following lemma and theorem about well-posedness without any proof for the sake of brevity. 
\begin{lemma}[Well-posedness of the local equation]\lemlab{Fds_one_ad_hp_localCon} Suppose that the assumptions \itmref{A1}-\itmref{A4} hold, the local solver \eqnref{Fds_one_ad_hdg_local} is well-posed.
\end{lemma}
\begin{theorem}[Well-posedness of the $hp$-HDG formulation]\theolab{Fds_one_ad_hp_globalCon} Suppose that
\ben
\item the assumptions \itmref{A1}-\itmref{A4} and  \eqnref{Abstrct_M1} hold. 
\item $\Null{\FdsABnd}=\LRc{\bs{0}}$.
\een
 There exists a unique solution $\FdsAdTrcVarAppr\in\FdsAdTrcVarApprSpc{}$ for the adjoint $hp$-HDG formulation defined in \eqnref{Fds_one_ad_hdg}.
\end{theorem}

On the other hand, the adjoint $hp$-HDG formulation with regard to the two-field Friedrichs' system reads: seek $\LRp{\FdsAdVarAuxAppr,\FdsAdVarSolAppr,\FdsAdTrcVarSolAppr}\in\FdsAdVarAuxApprSpc{}\times\FdsAdVarSolApprSpc{}\times\FdsAdTrcVarSolApprSpc{}$ such that
\begin{subequations}\eqnlab{Fds_two_ad_hdg}\small
\begin{align}
\begin{split}\eqnlab{Fds_two_ad_hdg_local1}
    -\FdsSum\LRp{\FdsAdVarSolAppr,{\FdsPartial}\LRp{\FdsBk^T\FdsAdTestAuxAppr}}_{\domPart}
    +\LRp{({\FdsGaa})^T\FdsAdVarAuxAppr+(({\FdsGsa})^T\FdsAdVarSolAppr,\FdsAdTestAuxAppr}_{\domPart} 
    -\LRa{{\FdsBBnd}\FdsAdTrcVarSolAppr,{\FdsAdTestAuxAppr}}_{\domPartBnd}&\\
    = -\frac{\delta\OutFcnl\LRs{\FdsVarLump}}{\FdsAdTestAuxAppr}&
\end{split}\\
\begin{split}\eqnlab{Fds_two_ad_hdg_local2}
    -\FdsSum\LRp{\FdsAdVarAuxAppr,{\FdsPartial}\LRp{{\FdsBk}\FdsAdTestSolAppr}}_{\domPart}
    -\FdsSum\LRp{\FdsAdVarSolAppr,{\FdsPartial}\LRp{{\FdsCk}\FdsAdTestSolAppr}}_{\domPart}
    +\LRp{{\FdsGas})^T\FdsAdVarAuxAppr + ({\FdsGss})^T\FdsAdVarSolAppr,\FdsAdTestSolAppr}_{\domPart}&\\
    + \LRa{-\FdsBBnd^{T}{\FdsAdVarAuxAppr} - \FdsCBnd^T\FdsAdTrcVarSolAppr +{\stabPar}\LRp{{\FdsAdVarSolAppr}-\FdsAdTrcVarSolAppr},{\FdsAdTestSolAppr}}_{\domPartBnd}
    = -\frac{\delta\OutFcnl\LRs{\FdsVarLump}}{\FdsAdTestSolAppr}&
\end{split}\\
\begin{split}\eqnlab{Fds_two_ad_hdg_global}
    \LRa{\jump{-\FdsBBnd^{T}{\FdsAdVarAuxAppr} - \FdsCBnd^T\FdsAdTrcVarSolAppr +{\stabPar}\LRp{{\FdsAdVarSolAppr}-\FdsAdTrcVarSolAppr},{\FdsAdTestSolAppr}},\FdsAdTrcTestSolAppr}_{\skel\backslash\domBndDir}
    = \LRa{{\varrho}\id_{\FdsNVarSol}\FdsAdTrcVarSolAppr,\FdsAdTrcTestSolAppr}_{\skel\cap(\domBndNmn\cup\domBndRb)}
    -\frac{\delta\OutFcnl\LRs{\FdsVarLump}}{\FdsAdTrcTestSolAppr}&
\end{split}\\
\begin{split}\eqnlab{Fds_two_ad_hdg_Dirichlet}
    \LRa{\FdsAdTrcVarSolAppr,\FdsAdTrcTestSolAppr}_{\skelBnd\cap\domBndDir}
    = -\frac{\delta\OutFcnl\LRs{\FdsVarLump}}{\FdsAdTrcTestSolAppr},&
\end{split}
\end{align}
\end{subequations}
for all $\LRp{\FdsAdTestAuxAppr,\FdsAdTestSolAppr,\FdsAdTrcTestSolAppr}\in\FdsAdVarAuxApprSpc{}\times\FdsAdVarSolApprSpc{}\times\FdsAdTrcVarSolApprSpc{}$. The following term has to be added into the output functional to account for the boundary conditions:
\beq\eqnlab{Fds_two_ad_Dirichlet}\small
\Output{\FdsVarLump}
=\LRa{-\FdsBBnd^T\FdsVarAuxAppr-\LRp{\FdsCBnd+\stabPar}\LRp{\FdsVarSolAppr+\FdsTrcVarSolAppr},\proj\bs{\DirVal}}_{\skelBnd\cap\domBndDir}
-\LRa{\FdsTrcVarSolAppr,\proj\bs{\NmnVal}}_{\skelBnd\cap\domBndNmn}
-\LRa{\FdsTrcVarSolAppr,\proj\bs{\RbVal}}_{\skelBnd\cap\domBndRb},
\eeq
where $\proj$ denotes the $\Lsp{2}$-projection into the space $\LRc{\eval{\FdsTrcVarSolAppr}{\domBnd}\,\forall\FdsTrcVarSolAppr\in\FdsTrcVarSolApprSpc{}}$. The similar analysis used in lemma \lemref{Fds_two_full_hp_localCon} and theorem \theoref{Fds_two_full_hp_globalCon}  can still be applied to the adjoint $hp$-HDG formulation \eqnref{Fds_two_ad_hdg} for the primal formulation \eqnref{Fds_two_hp_hdg} with full coercivity. Likewise, the similar argument presented in lemma \lemref{Fds_two_partial_hp_localCon} and theorem \theoref{Fds_two_partial_hp_globalCon} can be used for the same adjoint formulation \eqnref{Fds_two_ad_hdg} but for the primal formulation \eqnref{Fds_two_hp_hdg} with partial coercivity. To be concise, the following lemma and theorem about well-posedness are hence given without proof. 
\begin{lemma}[Well-posedness of the local equation-with full coercivity]\lemlab{Fds_two_full_ad_localCon} The local solver composed by \eqnref{Fds_two_ad_hdg_local1} and \eqnref{Fds_two_ad_hdg_local2} is well-posed provided that:
\ben
    \item the assumptions \itmref{A1}-\itmref{A6} hold, and
    \item $\half\FdsCBnd + \stabPar\geq0$, and
    \item $\FdsBk$ is a constant and is non-zero for $\FdsIndx=1,\dots,d$.
\een
\end{lemma}
\begin{theorem}[Well-posedness of the $hp$-HDG formulation-with full coercivity]\theolab{Fds_two_full_ad_globalCon} Suppose: 
\ben
    \item the assumptions \itmref{A1}-\itmref{A6} and \eqnref{Abstrct_M1} hold, and
    \item $\half\FdsCBnd + \stabPar\geq0$, and
    \item $\FdsBk$ is constant and is nonzero for $\FdsIndx=1,\dots,d$, and
    \item $\bigcap_{\FdsIndx=1}^d\text{Range}\LRp{\FdsBk}=\emptyset$ and $\Null{\FdsBk}=\LRc{\bs{0}}$ for $\forall\FdsIndx=1,\dots,d$. 
\een
Then, there exists a unique solution $\FdsAdTrcVarSolAppr\in\FdsAdTrcVarSolApprSpc{}$ for the $hp$-HDG formulation defined in \eqnref{Fds_two_ad_hdg}.
\end{theorem}
\begin{lemma}[Well-posedness of the local equation-with partial coercivity]\lemlab{Fds_two_partial_ad_localCon} The local solver composed by \eqnref{Fds_two_ad_hdg_local1} and \eqnref{Fds_two_ad_hdg_local2}  is well-posed provided that:
\ben
    \item the assumption \itmref{A1}-\itmref{A3}, \itmref{A4a}-\itmref{A4b} and \itmref{A5}-\itmref{A6} hold, and
    \item $\half\FdsCBnd + \stabPar>0$, and 
    \item $\bigcap_{\FdsIndx=1}^d\text{Range}\LRp{\FdsBk}=\emptyset$ and $\Null{\FdsBk}=\LRc{\bs{0}}$ for $\forall\FdsIndx=1,\dots,d$. 
\een
\end{lemma}
\begin{theorem}[Well-posedness of the $hp$-HDG formulation -with partial coercivity]\theolab{Fds_two_partial_ad_globalCon} Suppose: 
\ben
    \item the assumptions \itmref{A1}-\itmref{A3}, \itmref{A4a}-\itmref{A4b}, \itmref{A5}-\itmref{A6}, and \eqnref{Abstrct_M1} hold, and 
    \item $\half\FdsCBnd + \stabPar>0$, 
    \item $\bigcap_{\FdsIndx=1}^d\text{Range}\LRp{\FdsBk}=\emptyset$ and $\Null{\FdsBk}=\LRc{\bs{0}}$ for $\forall\FdsIndx=1,\dots,d$. 
\een
There exists a unique solution $\FdsAdTrcVarSolAppr\in\FdsAdTrcVarSolApprSpc{}$ for the $hp$-HDG formulation defined in \eqnref{Fds_two_ad_hdg}.
\end{theorem}
It should be noted that the assumptions needed for the well-posedness of primal formulations are exactly what requirements are needed for the corresponding adjoint formulation to be well-posedness. In other words, the well-posedness of the primal problems will imply one of the adjoint problems in this work. In fact, this observation holds true for all HDG methods. To see it, we can express both volume and trace unknowns as discrete vectors (i.e. each item in the vectors represents a nodal value) instead of a function. The system of primal equations can then be rewritten in matrix form and the transpose of the matrix is exactly the matrix in the corresponding system of adjoint equations \cite{Dahm2014,Woopen2014b}. It is easy to see that the transpose of a square matrix is indeed invertible if the original square matrix is invertible.

Finally, we would like to point out that the adjoint $hp$-HDG formulations \eqnref{Fds_one_ad_hdg} and \eqnref{Fds_two_ad_hdg} are not equipped with upwind-based HDG numerical flux discussed in this paper. That is, adjoint $hp$-HDG formulations derived from the adjoint problem of the general PDEs \eqnref{general_PDE} using the Godnouv approach addressed in this paper will be different from the formulations presented in \eqnref{Fds_one_ad_hdg} and \eqnref{Fds_two_ad_hdg}.   
\renewcommand{\theequation}{b.\arabic{equation}}
\section{Proof of the existence of the upwind flux in reduced form}\seclab{proof_reduced_upwind}
We first look at the upwind flux stated in \eqnref{Godunov_flux}. Thanks to the assumption \itmref{A5}, the submatrix $\FdsAkaa$ contribute nothing and hence we have
\beq\eqnlab{flux1}
    \upwindFlux\LRp{\FdsVarAuxAppr^*, \FdsVarSolAppr^*}\normal =
    \begin{bmatrix}
    \FdsBBnd\FdsVarSolAppr^*\\
    \FdsBtBnd\FdsVarAuxAppr^* + \FdsCBnd\FdsVarSolAppr^*
    \end{bmatrix}.
\eeq
By the equality \eqnref{Godunov_flux_eq}, Eq. \eqnref{flux1} can be rewritten as
\beq\eqnlab{flux2}  \upwindFluxMort\LRp{\FdsVarAuxAppr^-,\FdsVarSolAppr^-,\FdsVarAuxAppr^*,\FdsVarSolAppr^*}\normal =
\begin{bmatrix}
\FdsBBnd\FdsVarSolAppr^*\\
\FdsBtBnd\FdsVarAuxAppr^* + \FdsCBnd\FdsVarSolAppr^*
\end{bmatrix}.
\eeq
Since the upwind state $\FdsBBnd\FdsVarSolAppr^*$ is desired to be kept in the numerical flux, we replace the second component in $\upwindFluxMort$ by using Eq. \eqnref{Fds_two_upwind_flux} and arrive at
\beq\eqnlab{flux3}  \upwindFluxMort\LRp{\FdsVarAuxAppr^-,\FdsVarSolAppr^-,\FdsVarAuxAppr^*,\FdsVarSolAppr^*}\normal =
\begin{bmatrix}
\FdsBBnd\FdsVarSolAppr^*\\
\FdsBBnd^T\FdsVarAuxAppr+\FdsCBnd\FdsVarSolAppr+\FdsAbsAsa\LRp{\FdsVarAuxAppr - \FdsVarAuxAppr^*}+ \FdsAbsAss\LRp{\FdsVarSolAppr - \FdsVarSolAppr^*}
\end{bmatrix}.
\eeq
The goal is to eliminate the state $\FdsVarAuxAppr^*$ from the right-hand side of Eq. \eqnref{flux3} via either the assumption \itmref{F1} or \itmref{F2}. 

We first consider that \itmref{F1} holds true. Since the second row in the equality \eqnref{Godunov_flux_eq} is already used to rewrite \eqnref{flux2} as \eqnref{flux3} and the first row remains unused, we can take advantage of it and it reads:
\beq\eqnlab{flux4}
\FdsBBnd\FdsVarSolAppr^*
=
\FdsBBnd\FdsVarSolAppr+\FdsAbsAaa\LRp{\FdsVarAuxAppr - \FdsVarAuxAppr^*}+\FdsAbsAas\LRp{\FdsVarSolAppr - \FdsVarSolAppr^*}.
\eeq
By invoking the assumptions \itmref{F1a} and \itmref{F1b}, the Eq. \eqnref{flux4} can be rearranged as
\beq\eqnlab{flux5}
\FdsAbsAaa\LRp{\FdsVarAuxAppr - \FdsVarAuxAppr^*}
=
-\LRp{\FluxAssumpA^T\FdsBtBnd\FdsBBnd\FluxAssumpA}^{-1}\FluxAssumpA^T\FdsBtBnd\FdsBBnd\LRp{\FluxAssumpB+\id_{\FdsNVarSol}},
\eeq
where the matrix $\LRp{\FluxAssumpA^T\FdsBtBnd\FdsBBnd\FluxAssumpA}^{-1}$ is guaranteed to exist owing to the assumptions \itmref{F1a} ($\FluxAssumpA^{-1}$ exists) and \itmref{F1c} (implies that $\Null{\FdsBBnd}=\LRc{\bs{0}}$). Now substitute \eqnref{flux5} into \eqnref{flux3} and define $\stabPar$ in the way described in \eqnref{stabPar_F1}, we then can arrive at the formation \eqnref{Fds_two_upwind_flux_reduced}.

Now assume that \itmref{F2} holds true, it is obvious that the state $\FdsVarAuxAppr^*$ will vanish. By applying the definition \eqnref{stabPar_F2}, the formation \eqnref{Fds_two_upwind_flux_reduced} can then be obtained. 

\bibliographystyle{unsrt}
\bibliography{References.bib}
\end{document}